\documentclass[twoside,11pt,reqno]{amsart}
\usepackage{amscd}
\usepackage{amssymb,todonotes}
\usepackage{hyperref}
\makeatletter

\numberwithin{equation}{section}
\newtheorem{Proposition}{Proposition}[section]
\newtheorem{Lemma}[Proposition]{Lemma}
\newtheorem{Theorem}[Proposition]{Theorem}
\newtheorem{Corollary}[Proposition]{Corollary}
\newtheorem{iTheorem}{Theorem}
\newtheorem{Remark}[Proposition]{Remark}
\newtheorem{Definition}[Proposition]{Definition}

\newbox\squ  
\setbox\squ=\hbox{\vrule width.6pt
   \vbox{\hrule height.6pt width.4em\kern1ex
         \hrule height.6pt}%
   \vrule width.6pt}
\def\lamdba{\lambda}
\def\lo{{\scriptstyle\!\,\vee}}
\def\up{{\scriptstyle\!\,\wedge}}
\def\scriptlo{{\scriptscriptstyle\!\,\vee}}
\def\scriptup{{\scriptscriptstyle\!\,\wedge}}

\def\C{{\mathbb C}}
\def\K{{\mathbb K}}
\def\Q{{\mathbb Q}}
\def\Z{{\mathbb Z}}
\def\N{{\mathbb N}}
\def\pC{\operatorname{p}\!\mathcal C}
\def\0{{\bar 0}}
\def\1{{\bar 1}}
\def\eps{{\varepsilon}}

\def\bn{{\underline{n}}}

\def\bc{{\underline{c}}}

\def\bi{\text{\boldmath$i$}}
\def\bj{\text{\boldmath$j$}}
\def\bk{\text{\boldmath$k$}}
\def\qdim{\operatorname{dim}_q}
\def\Hom{\operatorname{Hom}}
\def\End{\operatorname{End}}
\def\Ext{\operatorname{Ext}}

\def\mod{\operatorname{\!-mod}}

\def\boldrmod{\operatorname{grmod-}\!\!}
\def\rmod{\operatorname{mod-\!\!}}
\def\rproj{\operatorname{pmod-\!\!}}
\def\rinj{\operatorname{imod-\!\!}}

\def\top{\operatorname{top}}
\def\tail{\operatorname{tail}}
\def\head{\operatorname{head}}
\def\pr{\operatorname{pr}}

\def\V{\mathbb{U}}

\newcommand{\jontodo}{\todo[inline,color=green!20]}

\newcommand{\arxiv}[1]{\href{http://arxiv.org/abs/#1}{\tt
    arXiv:\nolinkurl{#1}}}
\newcommand{\excise}[1]{}

\title[Super Kazhdan-Lusztig]{Tensor product categorifications and the super Kazhdan-Lusztig conjecture}
\author{Jonathan Brundan, Ivan Losev and Ben Webster}

\address{Department of Mathematics, University of Oregon, Eugene, OR 97403}
\email{brundan@uoregon.edu}
\address{Department of Mathematics, Northeastern University,
Boston, MA 02115}
\email{i.loseu@neu.edu}
\address{Department of Mathematics, University of Virginia, Charlottesville,
VA 22904}
\email{btw4e@virginia.edu}
\thanks{2010 {\it Mathematics Subject Classification}: 16E05, 16S38, 17B37.}
\thanks{Authors
supported in part by NSF grant nos. DMS-1161094,
DMS-1161584 and DMS-1151473, respectively.}

\begin{document}

\begin{abstract}
We give a new proof of the ``super Kazhdan-Lusztig conjecture'' for
the Lie superalgebra $\mathfrak{gl}_{n|m}(\C)$ as formulated
originally by the first author. 
We also prove for the first time that any integral block of
category $\mathcal O$ for $\mathfrak{gl}_{n|m}(\C)$ 
(and also all of its parabolic analogs) 
possesses a graded version which is Koszul.
Our approach depends crucially on an application of
the uniqueness of tensor product categorifications established recently by the second two authors.
\tableofcontents
\end{abstract}

\maketitle
\section{Introduction}

In this paper we explain how the uniqueness of tensor product categorifications
established by the second two authors in \cite{LW} yields a quick 
proof of the Kazhdan-Lusztig conjecture for the general linear Lie
superalgebra $\mathfrak{gl}_{n|m}(\C)$.
This conjecture was formulated originally by the first author in
\cite{B} and has been proved already by a different method by 
Cheng, Lam and Wang in \cite{CLW}. 
Actually we prove here a
substantially stronger result, namely, that the analog
of the Bernstein-Gelfand-Gelfand category $\mathcal O$
for $\mathfrak{gl}_{n|m}(\C)$ possesses 
a Koszul graded lift, in the spirit of the classic work \cite{BGS}.

Roughly, the super Kazhdan-Lusztig conjecture asserts that
combinatorics in integral blocks of category $\mathcal O$ for
$\mathfrak{gl}_{n|m}(\C)$ is controlled by various canonical bases in
the $\mathfrak{sl}_\infty$-module $V^{\otimes n} \otimes W^{\otimes
  m}$, where $V$ is the natural $\mathfrak{sl}_\infty$-module and $W$
is its dual.  In fact, we prove a generalization of the conjecture
which is adapted to the highest weight structure on $\mathcal O$ arising from any choice of conjugacy class of Borel subalgebra;
changing the Borel corresponds to shuffling the tensor factors in the
mixed tensor space $V^{\otimes n} \otimes W^{\otimes m}$ 
into more general orders. This generalization was suggested in the introduction of
\cite{Kuj}, then precisely formulated and proved in
\cite{CLW}. 
(We point out also the paper \cite{CMW} which
establishes an equivalence of categories from an arbitrary
non-integral block of category $\mathcal O$ for
$\mathfrak{gl}_{n|m}(\C)$ to an integral block of a direct sum of
other general linear Lie superalgebras of the same total rank.)

The basic idea of our proof is as follows. For a finite interval $I \subset \Z$, let $\mathfrak{sl}_I$
be the special linear Lie algebra consisting 
of (complex) trace zero matrices with rows and columns indexed 
by integers from the set
$I_+ := I \cup (I+1).$
Let $V_I$ be the natural $\mathfrak{sl}_I$-module of column vectors
and $W_{I} := V_I^*$.
We construct a subquotient
$\mathcal O_I$ of the super category $\mathcal O$ which is an
$\mathfrak{sl}_I$-categorification of
the tensor product $V_I^{\otimes n} \otimes W_I^{\otimes m}$ in the sense of 
Chuang and Rouquier \cite{CR},\cite{R}.
Then,
observing that
$$
V_I^{\otimes n} \otimes W_I^{\otimes m}
\cong V_I^{\otimes n} \otimes ({\textstyle\bigwedge^{|I|}}
V_I)^{\otimes m},
$$
one can apply the uniqueness of tensor product categorifications from
\cite{LW} to deduce that $\mathcal O_I$ is equivalent to another
well-known categorification $\mathcal O_I'$ of this tensor 
product arising from the parabolic category $\mathcal O$ 
associated to the Lie algebra $\mathfrak{gl}_{n+m|I|}(\C)$ and its
Levi subalgebra 
$\mathfrak{gl}_{1}(\C)^{\oplus n} \oplus
\mathfrak{gl}_{|I|}(\C)^{\oplus m}$. 
The combinatorics of the latter category is 
understood by the ordinary Kazhdan-Lusztig conjecture
proved in \cite{BB},\cite{BKa}. 
Since the finite interval $I$ can be chosen freely,
this gives enough information to deduce the 
super Kazhdan-Lusztig conjecture. 

By the well-known results from \cite{BGS} and \cite{Back}, 
the category $\mathcal O_I'$ has a graded version which is Koszul.
Hence so does the equivalent category $\mathcal O_I$.
To construct a Koszul grading on 
category $\mathcal O$ for $\mathfrak{gl}_{n|m}(\C)$,
we show further that the 
Koszul gradings on each $\mathcal O_I$
can be chosen in a compatible way so that they
lift to $\mathcal O$ itself.
Again we do this also for all of the parabolic analogs of $\mathcal O$, so that a very special case recovers 
the Koszul grading on the subcategory of
$\mathcal O$ consisting of
finite dimensional representations
that was constructed explicitly in \cite{BS}.
Our main result here can be paraphrased as follows.

\begin{iTheorem}
  Any block of parabolic category $\mathcal O$ for $\mathfrak{gl}_{n|m}(\C)$ with integral
  central character has a
  graded lift which is a standard Koszul highest weight
  category. Moreover its
 graded decomposition numbers can be computed in terms of finite
  type $A$ parabolic Kazhdan-Lusztig
  polynomials, as predicted in \cite{B}.
\end{iTheorem}

In the main body of the article we adopt a more axiomatic approach
in the spirit of \cite{LW}. In Section~\ref{tpc},
we write down the formal definition of an
{\em $\mathfrak{sl}_\infty$-tensor product categorification} of a
tensor product of exterior powers of $V$ and $W$.
This is a category with
  \begin{itemize}
 \item an $\mathfrak{sl}_\infty$-action in the sense discussed in
Definition~\ref{catdef},
\item a highest weight category structure as in Definition~\ref{hwdef}, and
\item some compatibility between these structures
explained in Definition~\ref{tpcdef},
\end{itemize}
such that the complexified Grothendieck group of the underlying category of $\Delta$-filtered objects is isomorphic to the given
  tensor product of exterior powers of $V$ and $W$.
Then the bulk of the article is taken up with proving the following
fundamental result about such categorifications.

\begin{iTheorem}
There exists a unique $\mathfrak{sl}_\infty$-tensor
product categorification
associated to any tensor product of exterior powers of $V$ and $W$.
Moreover such a category has a unique graded
lift compatible with all the above structures.
\end{iTheorem}

The existence part of this theorem is proved in Section~\ref{gl}, simply by verifying that parabolic category $\mathcal O$ for the general linear Lie superalgebra satisfies the axioms; in fact this is the only time Lie superalgebras
enter into the picture.
The uniqueness (up to strongly equivariant equivalence) is proved in Section~\ref{sstable}. It is a non-trivial
extension of the uniqueness theorem
for finite $\mathfrak{sl}_I$-tensor product categorifications 
established in \cite{LW}. The proof for $\mathfrak{sl}_\infty$ depends
on the construction of an interesting new category of {\em stable modules}
for a certain tower of quiver Hecke algebras.
Finally in Section~\ref{gtpc}, we incorporate gradings into the picture,
defining the notion of a
{\em $U_q\mathfrak{sl}_\infty$-tensor product categorification} 
of a tensor product of $q$-deformed exterior powers of $V$ and $W$;
see Definitions~\ref{ghwdef}, \ref{qcatdef} and \ref{qtpcdef}.
We prove the existence and uniqueness of these by exploiting 
{\em graded} stable  modules over our tower of quiver Hecke algebras.
Then
we prove 
that any such category is standard Koszul and
deduce the graded version of the 
Kazhdan-Lusztig conjecture.

\vspace{2mm}
\noindent
{\em Conventions.}
We fix an algebraically closed field $\K$ of characteristic $0$
throughout the article. 
All categories and functors will be assumed to be $\K$-linear without
further notice.
Let $\mathcal V{ec}$ be the category of finite dimensional vector spaces.
For a finite dimensional graded vector space 
$V = \bigoplus_{n \in \Z} V_n$, we write 
$\qdim V$ for its {\em graded dimension}
$\sum_{n \in \Z} (\dim V_n) q^n \in \Z[q,q^{-1}]$.

\section{Tensor product categorifications}\label{tpc}

In this section, we 
review the definition of tensor product categorification from \cite[Definition  3.2]{LW} in the special case of tensor products of exterior powers of the natural and dual natural representations of $\mathfrak{sl}_I$.
We include the possibility that the interval $I \subseteq \Z$ is
infinite, when these are not highest weight modules.
Then we state our first main result asserting the existence and uniqueness of such tensor product categorifications, extending the case of 
finite intervals from \cite{LW}. After that, we make some preparations
for the proof (which actually takes place in Sections~\ref{gl} and
\ref{sstable}), and discuss some first applications.

\subsection{Schurian categories}\label{sc}
By a {\em Schurian category}  we mean an Abelian category
$\mathcal C$
such that all objects are of finite length, there are enough projectives and injectives,
and the endomorphism
algebras of the irreducible objects are one dimensional.
For example, the category $\rmod{A}$ of finite dimensional 
right modules over a finite dimensional $\K$-algebra $A$ is Schurian.
Note throughout this text we will work in terms of projectives, 
but obviously $\mathcal C$ is Schurian if and only if $\mathcal C^{\operatorname{op}}$ is Schurian, 
so that everything could be expressed equivalently in terms of injectives.
We use the made-up word {\em prinjective} for an object that is both projective and injective.

Given a Schurian category $\mathcal C$, we let $\pC$
be the full subcategory consisting of all projective
objects.
Let $\operatorname{Fun}_f(\pC,
\mathcal{V}ec^{\operatorname{op}})$
denote the category of all contravariant functors from $\pC$
to $\mathcal{V}ec$ which are zero on all but finitely many
isomorphism classes of indecomposable projectives.
The Yoneda functor
\begin{equation}\label{yoneda}
\mathcal C \rightarrow
\operatorname{Fun}_f(\pC, \mathcal{V}ec^{\operatorname{op}}),
\qquad
M \mapsto \Hom_{\mathcal C}(-,M)
\end{equation}
is an equivalence.
Hence $\mathcal C$ can be recovered (up to equivalence)
from $\pC$.

This assertion can be formulated in more algebraic terms as follows.
Let $\{L(\lambda)\:|\:\lambda \in \Lambda\}$ be a complete set of pairwise non-isomorphic irreducible objects in $\mathcal C$, and fix a choice of a projective cover $P(\lambda)$ of each $L(\lambda)$.
Let
\begin{equation}\label{assertion}
A := \bigoplus_{\lambda,\mu \in \Lambda} 
\Hom_{\mathcal C}(P(\lambda),P(\mu))
\end{equation}
viewed as an associative algebra with multiplication coming from composition in $\mathcal C$. Let $1_\lambda \in A$ be the identity endomorphism of $P(\lambda)$.
If $\Lambda$ is finite then $A$ is a unital algebra with
$1 = \sum_{\lambda \in \Lambda} 1_\lambda$, indeed, $A$ is the endomorphism algebra of the minimal projective generator $\bigoplus_{\lambda \in \Lambda} P(\lambda)$.
However in general $A$
is only {\em locally unital}, meaning that it is
equipped with the system $\{1_\lambda\:|\:\lambda\in \Lambda\}$
of mutually orthogonal idempotents such that
$A = \bigoplus_{\lambda,\mu\in\Lambda} 1_\mu A 1_\lambda$.
Let $\rmod A$ denote the category of all finite dimensional 
locally unital right $A$-modules, that is, finite dimensional 
right $A$-modules $M$ such that
$M = \bigoplus_{\lambda \in \Lambda} M 1_\lambda$.
Then our earlier assertion about the Yoneda equivalence 
amounts to the statement that the functor
\begin{equation}\label{stupid1}
\mathbb{H}:\mathcal C \rightarrow \rmod A, \qquad
M \mapsto \bigoplus_{\lambda \in \Lambda}\Hom_{\mathcal C}(P(\lambda), M)
\end{equation}
is an equivalence of categories.
Of course this functor sends $P(\lambda)$ to the (necessarily finite dimensional) right ideal
$1_\lambda A$; these are the indecomposable projective modules
in $\rmod A$. The linear duals of the indecomposable injective modules are 
isomorphic to the left ideals $A 1_\lambda$, so that the latter
 are finite dimensional too.
Conversely given any locally unital $\K$-algebra $A$ with
distinguished
idempotents $\{1_\lambda\:|\:\lambda\in\Lambda\}$ such that all of the ideals 
$1_\lambda A$ and $A 1_\lambda$ are finite dimensional, the category $\rmod A$ is Schurian.

Let $K_0(\mathcal C)$ (resp.\ $G_0(\mathcal C)$) 
be the split Grothendieck group of the additive category 
$\pC$ (resp. the Grothendieck group of the
Abelian category $\mathcal C$).
Set $$
[\mathcal C] := \C \otimes_{\Z} K_0(\mathcal C),
\quad\qquad
[\mathcal C]^* :=\C\otimes_{\Z} G_0(\mathcal C).
$$
So
$[\mathcal C]$ is the complex vector space on
basis $\{[P(\lambda)]\:|\:\lambda \in \Lambda\}$, while
$[\mathcal C]^*$ has basis $\{[L(\lambda)]\:|\:\lambda \in \Lambda\}$.
These bases are dual with respect to the bilinear {\em Cartan pairing}
$(-,-):[\mathcal C]\times [\mathcal C]^* \rightarrow \C$ defined from
$([P], [L]) := \dim \Hom_{\mathcal C}(P, L)$.

\subsection{Combinatorics}
Let $I \subseteq \Z$ be a (non-empty) interval and set
$$
I_+ := I \cup (I+1).
$$
Let $\mathfrak{sl}_I$ be the Lie algebra of (complex) trace zero matrices
with rows and columns indexed by $I_+$, all but finitely many of whose
entries are zero.
It is generated by the matrix units $f_i := e_{i+1,i}$ and $e_i := e_{i,i+1}$
for all $i \in I$.
The {\em weight lattice} of $\mathfrak{sl}_I$
is $P_I := 
\bigoplus_{i \in I} \Z \varpi_i$ where $\varpi_i$ is the {\em $i$th
  fundamental weight}.
The {\em root lattice} is $Q_I 
:= \bigoplus_{i \in I} \Z \alpha_i < P_I$ where $\alpha_i$
is the {\em $i$th simple root} defined from $$
\alpha_i := 2 \varpi_i -
\varpi_{i-1}-\varpi_{i+1},$$ 
interpreting $\varpi_i$ as $0$ if $i
\notin I$.
Let $P_I \times Q_I \rightarrow \Z, \,
(\varpi,\alpha) \mapsto \varpi\cdot\alpha$ be the bilinear pairing 
defined from
$\varpi_i\cdot\alpha_j := \delta_{i,j}$, so that $\big(\alpha_i\cdot\alpha_j\big)_{i,j \in I}$
is the Cartan matrix.
Let $P_I^+$ (resp.\ $Q_I^+$) be the positive cone in $P_I$ (resp.\
$Q_I$) generated by the fundamental weights (resp.\ the simple roots).
The {\em dominance order} $\geq$ on $P_I$
is defined by $\beta \geq \gamma$ if $\beta - \gamma \in Q_I^+$.
For any $i \in I_+$ we set
$$
\eps_i := \varpi_i - \varpi_{i-1},
$$
again interpreting $\varpi_i$ as $0$ for $i
\notin I$.
The following lemma is well known.

\begin{Lemma}\label{dominance}
For $\beta = \sum_{i \in I_+} b_i \eps_i$ and $\gamma = \sum_{i \in I_+} c_i \eps_i$
in $P_I$ with $\sum_{i} b_i = \sum_{i} c_i$, we have that
$\beta \geq \gamma$ 
if and only if
$\sum_{i \leq h} b_i \geq \sum_{i \leq h} c_i$
for all $h \in I$.
\end{Lemma}

An $\mathfrak{sl}_I$-module $M$ is {\em integrable} if it decomposes
into weight spaces as
$M = \bigoplus_{\varpi \in P_I} M_\varpi$, and moreover each of the
Chevalley generators $f_i$ and $e_i$ acts locally nilpotently.
Basic examples are the natural $\mathfrak{sl}_I$-module $V_I$ of column vectors with 
standard basis $\{v_i\:|\:i \in I_+\}$ and its dual
$W_I$ with basis $\{w_i\:|\:i \in I_+\}$;
the Chevalley generators act on these basis vectors by
\begin{align*}
f_i v_j &= \delta_{i,j} v_{i+1},
&e_i v_j &= \delta_{i+1,j} v_i,\\
f_i w_j &= \delta_{i+1,j} w_{i},
&
e_i w_j &= \delta_{i,j} w_{i+1}.
\end{align*}
The vector $v_i$ is of weight $\eps_i$ while $w_i$ is of weight $-\eps_i$.

More generally we have the exterior powers
$\bigwedge^n V_I$ and $\bigwedge^n W_I$ for $n \geq 0$;
henceforth we denote these instead by $\bigwedge^{n,0} V_I$ and
$\bigwedge^{n,1} V_I$, respectively.
For $c \in \{0,1\}$ let $\Lambda_{I;n,c}$ denote the set
of
$01$-tuples $\lambda = (\lambda_i)_{i \in I_+}$
such that 
$$
\big|\{i\in I_+ \:|\:\lambda_i \neq c\}\big| = n.
$$
This set parametrizes the natural monomial basis $\{v_\lambda\:|\:\lambda\in\Lambda_{I;n,c}\}$
of $\bigwedge^{n,c}V_I$
defined 
from
$$
v_\lambda :=
\left\{
\begin{array}{ll}
v_{i_1}\wedge\cdots\wedge v_{i_n}
&\text{if $c=0$,}\\
w_{i_1}\wedge\cdots\wedge w_{i_n}&\text{if $c=1$,}
\end{array}\right.$$
where $i_1 < \cdots < i_n$ are chosen so that $\lambda_{i_j} \neq c$
for each $j$.
The actions of the Chevalley generators are given explicitly by
\begin{align}\label{cact1}
f_i v_\lambda &:= \left\{
\begin{array}{ll}
v_{t_i(\lambda)}&\text{if $\lambda_i = 1$ and $\lambda_{i+1}=0$,}\\
0&\text{otherwise,}
\end{array}\right.\\
e_i v_\lambda &:= \left\{
\begin{array}{ll}
v_{t_i(\lambda)}&\text{if $\lambda_i = 0$ and $\lambda_{i+1}=1$,}\\
0&\text{otherwise,}
\end{array}\right.\label{cact2}
\end{align}
where $t_i(\lambda)$ denotes the tuple obtained from $\lambda$ by 
switching $\lambda_i$ and $\lambda_{i+1}$.
Let $|\lambda| \in P_I$ denote the weight of
the vector $v_\lambda$.
We have that
$$
|\lambda| = \sum_{i \in I_+} \lambda_i \eps_i,
$$ 
interpreting the 
sum on the right hand side when $I$ is infinite and $c=1$ using the
convention that
$\cdots + \eps_{i-1}+\eps_i = \varpi_i$ and $\eps_{i+1}+\eps_{i+2}+\cdots
= -\varpi_i$.

We are also going to be interested in tensor products of the
modules $\bigwedge^{n,c} V_I$.
Suppose that we are given 
$\bn = (n_1,\dots,n_l) \in \N^l$ and $\bc = (c_1,\dots,c_l) \in
\{0,1\}^l$;
we refer to the pair $(\bn,\bc)$ as a {\em type} of {\em level} $l$.
Let
\begin{equation}\label{tp}
\textstyle\bigwedge^{\bn,\bc}V_I := \bigwedge^{n_1,c_1}V_I\otimes\cdots\otimes \bigwedge^{n_l,c_l}V_I.
\end{equation}
This module has the obvious basis of monomials
$v_\lambda := v_{\lambda_1} \otimes\cdots \otimes v_{\lambda_l}$
indexed by 
elements $\lambda = (\lambda_1,\dots,\lambda_l)$ of the set
\begin{equation}\label{lal}
\Lambda_{I;\bn,\bc} := 
\Lambda_{I;n_1,c_1} \times\cdots\times \Lambda_{I;n_l,c_l}.
\end{equation}
The vector $v_\lambda$ is of weight
\begin{equation*}
|\lambda| := |\lambda_1|+\cdots+|\lambda_l|.
\end{equation*}
It is often convenient to regard $\lambda \in \Lambda_{I;\bn,\bc}$
as a
$01$-matrix
$\lambda = (\lambda_{ij})_{1 \leq i \leq l, j \in I_+}$
with $i$th row $\lambda_i = (\lambda_{ij})_{j \in I_+}$.  
(There are
several other indexing conventions possible; for example earlier
papers of the first and third authors have used the convention that
$\lambda$ is represented by a column-strict tableau with $l$ columns
such that the $i$th column is filled with all $j \in I_+$
such that $\lambda_{ij}=1$.)

Assume for a moment that $I$ is finite and that $\Lambda_{I;\bn,\bc}$ is non-empty.
Let $\kappa =\kappa_{I;\bn,\bc}$
be the $01$-matrix in $\Lambda_{I;\bn,\bc}$
in which all the entries
$1$ are as far to the left as possible within each row.
Thus $|\kappa| \in P_I$
is the unique
highest weight of $\bigwedge^{\bn,\bc}V_I$ with respect to the
dominance ordering.
For any $\lambda \in \Lambda_{I;\bn,\bc}$ define its {\em defect} by
\begin{equation}\label{defect}
\textstyle\operatorname{def}(\lambda) := \frac{1}{2}(|\kappa|\cdot|\kappa|-|\lambda|\cdot|\lambda|)
=|\kappa|\cdot\alpha - \frac{1}{2}\alpha\cdot\alpha,
\end{equation}
where $\alpha := |\kappa|-|\lambda|$.
In combinatorial terms, this is $\frac{1}{2}\sum_{j \in I_+}
(k_j^2-l_j^2)$ where $k_j$ (resp.\ $l_j$) counts the number of entries
equal to $1$ in the $j$th column of $\kappa$ (resp.\ $\lambda$).
The following lemma extends this definition to 
include infinite
intervals $I$.

\begin{Lemma}\label{defectL}
Suppose that $I$ is an infinite interval and $\lambda \in
\Lambda_{I;\bn,\bc}$.
Let $J \subset I$ be a finite subinterval such that $|J_+| \geq 2
\max(\bn)$ and $\lambda_{i,j} = c_i$ for all $1 \leq i \leq l$ and $j
\in I_+ \setminus J_+$.
Let $\lambda_J \in \Lambda_{J;\bn,\bc}$ be the submatrix
$(\lambda_{i,j})_{1 \leq i \leq l, j \in J_+}$ of $\lambda$.
Let $\kappa_J := \kappa_{J;\bn,\bc}$.
Then the natural number 
\begin{equation*}
\textstyle\operatorname{def}(\lambda) := \frac{1}{2}(|\kappa_J|\cdot|\kappa_J| - |\lambda_J|\cdot|\lambda_J|)
\end{equation*}
is independent of the particular choice of $J$.
\end{Lemma}

\begin{proof}
Define the {\em trivial column} to be the column vector $(c_i)_{1
  \leq i \leq l}$.
Let $J$ and $J'$ be two intervals satisfying the hypotheses of the
lemma 
with $J \subset J'$.
The conditions imply that $\kappa_{J}$ (resp.\ $\lambda_J$) can be obtained from
$\kappa_{J'}$ (resp.\ $\lambda_{J'}$) by removing $|J'|-|J|$ trivial columns.
The lemma follows easily from this using the combinatorial formulation
of the definition
of defect.
\end{proof}

\subsection{Hecke algebras}\label{ha}
To prepare for the definition of an
$\mathfrak{sl}_I$-categorification, we recall the 
definition of certain
associative unital $\K$-algebras, namely, the {(degenerate) affine
  Hecke algebra} $AH_d$, and the {quiver Hecke
  algebra} $QH_{I,d}$ associated to the linear quiver with vertex set $I$
and an edge $i \rightarrow j$ if $i = j+1$.
The latter is also known as a Khovanov-Lauda-Rouquier algebra after \cite{KL1} and \cite{R}.

\begin{Definition}\rm 
 The {\em affine Hecke algebra} $AH_d$ is the vector space
  $$\K[x_1,\dots,x_d] \otimes \K S_d
$$ 
with multiplication defined so
  that the polynomial algebra $\K[x_1,\dots,x_d]$ and the group
  algebra $\K S_d$ of the symmetric group $S_d$ are subalgebras, and
  also
  \begin{itemize}
  \item[(AH)] $t_j x_k -x_{t_j(k)} t_j =\left\{
      \begin{array}{rl}
        1&\text{if $k=j+1$,}\\
        -1&\text{if $k=j$,}\\
        0&\text{otherwise.}
      \end{array}\right.
    $
  \end{itemize}
  Here $t_j$ denotes the transposition $(j\:j\!+\!1) \in S_d$.
\end{Definition}

\begin{Definition}
\rm  The {\em quiver Hecke algebra} $QH_{I,d}$
 is defined by generators \[\{1_\bi\:|\:\bi \in
  I^d\}\cup\{\xi_1,\dots,\xi_d\}\cup\{\tau_1,\dots,\tau_{d-1}\}\]
  subject to relations:
  \begin{itemize}
  \item[(QH1)] the elements $\xi_1,\dots,\xi_d$ commute with each
    other and all $\{1_\bi\:|\:\bi \in I^d\}$;
  \item[(QH2)] the elements $\left\{1_\bi\:|\:\bi \in I^d\right\}$ are
    mutually orthogonal idempotents whose sum is the identity;
  \item[(QH3)] $\tau_j 1_\bi = 1_{t_j(\bi)} \tau_j$ where $t_j(\bi)$
    is the tuple obtained from $\bi = (i_1,\dots,i_d)$ by flipping its
    $j$th and $(j+1)$th entries;
  \item[(QH4)] $(\tau_j \xi_k -\xi_{t_j(k)} \tau_j)1_\bi =\left\{
      \begin{array}{rl}
        1_\bi&\text{if $k=j+1$ and $i_j = i_{j+1}$,}\\
        -1_\bi&\text{if $k=j$ and $i_j = i_{j+1}$,}\\
        0&\text{otherwise;}
      \end{array}\right.
    $
  \item[(QH5)] $\tau_j \tau_k = \tau_k \tau_j$ if $|j-k|>1$;
  \item[(QH6)] $\tau_j^2 1_\bi = \left\{
      \begin{array}{ll}
        0&\text{if $i_j=i_{j+1}$,}\\
        (\xi_{j}-\xi_{j+1}) 1_\bi&\text{if $i_{j} = i_{j+1}-1$,}\\
        (\xi_{j+1}-\xi_j) 1_\bi&\text{if $i_{j} = i_{j+1}+1$,}\\
        1_\bi&\text{otherwise;}
      \end{array}\right.$
  \item[(QH7)] $(\tau_{j+1} \tau_{j} \tau_{j+1} -
    \tau_{j}\tau_{j+1}\tau_{j}) 1_\bi= \left\{\begin{array}{ll}
        1_\bi&\text{if $i_{j} = i_{j+1}-1=i_{j+2}$,}\\
        -1_\bi&\text{if $i_{j} = i_{j+1}+1=i_{j+2}$,}\\
        0&\text{otherwise.}
      \end{array}\right.
    $
  \end{itemize}
\end{Definition}

  An important feature of $QH_{I,d}$ is that it possesses
  a non-trivial $\Z$-grading.  This is defined by declaring that each
  idempotent $1_\bi$ is in degree $0$, each $\xi_j$ in degree $2$, and
  finally $\tau_k 1_\bi$ is in degree
  $-\alpha_{i_k}\cdot\alpha_{i_{k+1}}$.

The algebras $AH_d$ and $QH_{I,d}$ are closely related as explained in
\cite[Proposition 3.15]{R}. This result can also be formulated as an
isomorphism between certain {\em cyclotomic quotients} of $AH_d$ and
$QH_{I,d}$
as in \cite{BKl}. 
Let $\varpi \in P_I^+$  be a dominant weight.
Define $AH_d^\varpi$ (resp.\ $QH_{I,d}^\varpi$) to be the quotient of $AH_d$
(resp.\ $QH_{I,d}$) 
by the two-sided
ideal generated by the polynomial $\prod_{i \in I} (x_1-i)^{\varpi\cdot\alpha_i}$
(resp.\ by the elements $\{\xi_1^{\varpi\cdot\alpha_{i_1}} 1_\bi\:|\:\bi \in I^d\}$).
These are finite dimensional algebras.
The image of the polynomial algebra $\K[x_1,\dots,x_d]$
in $AH_d^\varpi$ is a finite dimensional commutative algebra, 
hence it contains mutually orthogonal idempotents $\{1_\bi\:|\:\bi \in \K^d\}$
such that $1_\bi$
projects any module $M$ onto its $\bi$-th {\em word space}
$$
M_\bi := \left\{v \in M\:\big|\:(x_j-i_j)^N v = 0\text{ for each $j=1,\dots,d$
  and $N \gg 0$}\right\}.
$$
Then let
$$
AH_{I,d}^{\varpi} := \bigoplus_{\bi,\bj \in I^d} 1_\bi AH_d^{\varpi}
1_\bj
= AH_d^{\varpi} \big/ \langle 1_\bi\:|\:\bi \notin I^d\rangle,
$$
which is a sum of blocks of the algebra $AH^\varpi_d$.
The following theorem gives an explicit choice of isomorphism between
$QH_{I,d}^\varpi$ and $AH_{I,d}^\varpi$;
any other reasonable choice of isomorphism
such as the one from \cite[Proposition 3.15]{R} could be used instead
throughout this article.

\begin{Theorem}[{\cite{BKl}, \cite{R}}]\label{grth}
For $\varpi \in P_I^+$ there is an isomorphism
$QH_{I,d}^\varpi \stackrel{\sim}{\rightarrow} AH_{I,d}^\varpi$
defined on generators by
\begin{align}
1_\bi &\mapsto 1_\bi;\label{iso1}\\
\xi_j 1_\bi &\mapsto (x_j - i_j) 1_\bi;\label{iso2}\\
\tau_j 1_\bi &\mapsto
\left\{
\begin{array}{ll}
(1+t_j)(1-x_j+x_{j+1})^{-1} 1_\bi&\text{if $i_j = i_{j+1}$,}\\
 (1+t_j x_j-t_j x_{j+1})1_\bi&\text{if $i_j = i_{j+1}+1$,}\\
 (1+t_j x_j-t_j x_{j+1})(1-x_j+x_{j+1})^{-1} 1_\bi&\text{otherwise.}
\end{array}
\right.\label{iso3}
\end{align}
\end{Theorem}
\noindent(In fact, this isomorphism can be extended to the completions
$\widehat{QH}_{I,d}$ and $\widehat{AH}_{I,d}$
with respect to these systems of quotients as discussed in \cite{W4}.)
\begin{proof}
This follows by \cite[Main Theorem]{BKl}. To
get exactly this isomorphism
one needs to choose the power series $q_j(\bi)$ of
\cite[(3.27)--(3.29)]{BKl} so that $q_j(\bi) = p_j(\bi)$ if $i_j = i_{j+1}+1$
and $q_j(\bi) = 1 - p_j(\bi)$ if $i_j \notin\{ i_{j+1}, i_{j+1}+1\}$.
Note also that the opposite orientation of
the quiver was used in \cite{BKl} so that the elements
$\psi_j e(\bi)$ in \cite{BKl} are our $\tau_j 1_\bi$ if
$i_j \in \{i_{j+1}, i_{j+1}+1\}$ and our $-\tau_j 1_\bi$ otherwise;
the elements $y_j e(\bi)$ in \cite{BKl} are our elements $\xi_j 1_\bi$.
\end{proof}

Henceforth we will simply {\em identify} $QH_{I,d}^\varpi$ and $AH_{I,d}^\varpi$
via the isomorphism from the theorem.

\subsection{Categorification}\label{scat}
Following their work \cite{CR}, Chuang and Rouquier introduced
the notion of an {$\mathfrak{sl}_I$-categorification}, also 
known as a {categorical $\mathfrak{sl}_I$-action}. 
The following is essentially \cite[Definition 5.32]{R}  (taking $q=1$ and switching the roles of $E$ and $F$).

\begin{Definition}\label{catdef}
\rm
An {\em $\mathfrak{sl}_I$-categorification} is a 
Schurian category
$\mathcal C$
together 
with an
endofunctor $F$, a right adjoint $E$ to $F$ (with a specified
adjunction),
and natural transformations $x \in \End(F)$ and $t \in
\End(F^2)$ satisfying the axioms (SL1)--(SL4) formulated below.
For the first axiom, we let $F_i$ be the subfunctor of $F$ defined by 
the generalized $i$-eigenspace of $x$,
i.e. $F_i M = \sum_{k \geq 0} \ker (x_M-i)^k$ for each $M \in \mathcal C$.
\begin{itemize}
\item[(SL1)]
We have that $F = \bigoplus_{i \in I} F_i$,
i.e. $F M = \bigoplus_{i \in I} F_i M$ for each $M \in \mathcal C$.
\item[(SL2)]
For $d \geq 0$ the endomorphisms $x_j := F^{d-j} x F^{j-1}$
and $t_k := F^{d-k-1} t F^{k-1}$ of $F^d$ satisfy the relations
of the degenerate affine Hecke algebra $AH_d$.
\item[(SL3)]
The functor $F$ is isomorphic to a right adjoint of $E$.
\end{itemize}
For the final axiom, we let
$c:\operatorname{id} \rightarrow EF$ and $d:FE \rightarrow \operatorname{id}$ be the unit and counit of the given adjunction, respectively.
The
endomorphisms $x$ and $t$ of $F$ and $F^2$ induce endomorphisms
$x$ and $t$ of $E$ and $E^2$ too:
\begin{align}\label{xp}
x&:E \stackrel{cE}{\rightarrow} EFE \stackrel{ExE}{\rightarrow} EFE 
\stackrel{Ed}{\rightarrow} E,\\
t&:E^2 \stackrel{cE^2}{\rightarrow} EF E^2
\stackrel{EcFE^2}{\rightarrow} E^2 F^2 E^2
\stackrel{E^2 t E^2}{\rightarrow} E^2 F^2 E^2
\stackrel{E^2 FdE}{\rightarrow} E^2 FE \stackrel{E^2 d}{\rightarrow} E^2.\label{ttp}
\end{align}
(We remark that 
these satisfy slightly different relations to the original $x$ and $t$: the signs on the right hand side of the
degenerate affine Hecke algebra relation (AH) must be reversed.) 
Let $E_i$ be the subfunctor of $E$ defined by the generalized $i$-eigenspace of 
$x \in \End(E)$.
The
axioms so far imply that 
$E = \bigoplus_{i \in I} E_i$ and moreover
$F_i$ and $E_i$ are biadjoint, so they are both exact and send projectives to projectives.
\begin{itemize}
\item[(SL4)]
The endomorphisms $f_i$ and $e_i$
of $[\mathcal C] = \C\otimes_\Z K_0(\mathcal C)$ induced by $F_i$ and $E_i$, respectively,
make $[\mathcal C]$
into an integrable representation of
$\mathfrak{sl}_I$. Moreover the classes of the indecomposable projective
objects are weight vectors.
\end{itemize}
The axiom (SL4) has the following equivalent dual formulation.
\begin{itemize}
\item[(SL4$^*$)]
The endomorphisms $f_i$ and $e_i$
of $[\mathcal C]^* = \C \otimes_{\Z} G_0(\mathcal C)$ induced by $F_i$
and $E_i$, respectively,
make $[\mathcal C]^*$
into an integrable representation of
$\mathfrak{sl}_I$. Moreover the classes of the irreducible
objects are weight vectors.
\end{itemize}
\end{Definition}

The axiom (SL1) implies that $F^d$ decomposes as $\bigoplus_{\bi \in I^d}
F_\bi$ where 
$$
F_\bi := F_{i_d} \circ\cdots\circ F_{i_1}.
$$ This further shows that the action of $AH_d$ factors through the
completion $\widehat{AH}_{I,d}$ of the inverse system of cyclotomic quotients $\{AH_{I,d}^\varpi\:|\:\varpi\in P_I^+\}$.
Letting $1_\bi \in \End(F^d)$ be the projection onto
$F_{\bi}$ we can then use the isomorphism of completions given by (\ref{iso1})--(\ref{iso3}) 
to convert the homomorphism $AH_d \rightarrow \widehat{AH}_{I,d}\rightarrow
\End(F^d)$ into a homomorphism $QH_{I,d} \rightarrow \widehat{QH}_{I,d}
\rightarrow \End(F^d)$.
Hence the definition of
 an $\mathfrak{sl}_I$-categorification can be formulated equivalently using
the quiver Hecke algebra $QH_{I,d}$ in place of the degenerate affine Hecke
algebra $AH_d$. In this incarnation, 
$\mathcal C$ should be equipped with 
adjoint pairs
$(F_i,E_i)$ of endofunctors for all $i \in I$ (with specified adjunctions),
together with
natural transformations $\xi \in \End(F)$ and $\tau \in
\End(F^2)$ where $F := \bigoplus_{i \in I} F_i$, satisfying
the axioms (SL1$'$)--(SL4$'$).
\begin{itemize}
\item[(SL1$'$)]
 The endomorphism $\xi$ is locally nilpotent, i.e.
$F M = \sum_{k \geq 0} \ker \xi_M^k$ for each $M \in \mathcal C$.
\item[(SL2$'$)]
For $d \geq 0$ the endomorphisms $\xi_j := F^{d-j} \xi F^{j-1}$
and $\tau_k := F^{d-k-1} \tau F^{k-1}$ of $F^d$ 
plus the projections $1_\bi$ of $F^d$ onto its summands
$F_\bi$ for each $\bi
\in I^d$ satisfy the relations
of the quiver Hecke algebra $QH_{I,d}$.
\item[(SL3$'$)]
Each functor $F_i$ is isomorphic to a right adjoint of $E_i$.
\item[(SL4$'$)] 
Same as (SL4).
\end{itemize}
In fact this is just the first of several alternate definitions of
$\mathfrak{sl}_I$-categorification in the literature. Notably in \cite[Theorem 5.30]{R} Rouquier proves  
that the data of an $\mathfrak{sl}_I$-categorification as above is equivalent to
the data of an integrable $2$-representation of the Kac-Moody
2-category
associated to $\mathfrak{sl}_{I}$ in the sense of \cite[Definition
5.1]{R}; see also \cite{KL3} and \cite{CL} for closely related
notions.
(We point out also the recent article \cite{Bnew}, which shows that the
seemingly different definitions in \cite{R, KL3, CL} actually yield
isomorphic $2$-categories.)

\begin{Definition}\label{sev}\rm
Given two
$\mathfrak{sl}_I$-categorifications $\mathcal C$
and $\mathcal C'$, and denoting $F, E, x, t$ for $\mathcal C'$
instead by $F', E', x', t'$ for clarity,
a functor $\mathbb{G}:\mathcal C \rightarrow \mathcal C'$
is
{\em strongly equivariant}
if there exists an
isomorphism of functors 
$\zeta:
F' \circ \mathbb{G}
\stackrel{\sim}{\rightarrow} 
\mathbb{G}\circ F$
such that
\begin{itemize}
\item[(E1)]
the natural transformation
$E' \mathbb{G} \eps \circ E' \zeta E \circ \eta'
\mathbb{G} E:\mathbb{G}\circ E \rightarrow E' \circ \mathbb{G}$
is an isomorphism;
\item[(E2)]
$\zeta \circ x' \mathbb{G}  = \mathbb{G} x \circ \zeta$
in $\Hom(F' \circ \mathbb{G},\mathbb{G}\circ F)$;
\item[(E3)]
$\zeta F \circ F' \zeta
\circ t' \mathbb{G}=\mathbb{G} t \circ \zeta F \circ F' \zeta$
in $\Hom({F'}^2\circ \mathbb{G}, \mathbb{G} \circ F^2)$.
\end{itemize}
If it happens that $\mathbb{G}$ is an equivalence of categories then the axiom (E1) holds
automatically, and
we call $\mathbb{G}$ a 
{\em strongly equivariant equivalence}.
\end{Definition}

As usual, the definition of strongly equivariant functor
can be
formulated in terms of quiver Hecke algebras. In that setting, the
isomorphism $\zeta$ is induced by isomorphisms
$\zeta_i:F_i' \circ \mathbb{G} \stackrel{\sim}{\rightarrow}
\mathbb{G} \circ F_i$ for each $i$, and the endomorphisms $x$
and $t$ in (E2)--(E3)
are replaced by $\xi$ and $\tau$.

\subsection{Recollections about highest weight categories}\label{rhw}
We must also make a few reminders about (artinian) highest weight categories
in the sense of \cite{CPS}; see also \cite[Appendix]{Donkin} which is a good source for
proofs of all the results stated in this subsection (although it only treats finite weight posets).

\begin{Definition}\rm\label{hwdef}
A {\em highest weight category} is
a Schurian category $\mathcal C$
together with an interval-finite poset $(\Lambda, \leq)$
indexing a complete set
of pairwise non-isomorphic irreducible objects 
$\{L(\lambda)\:|\:\lambda\in\Lambda\}$ of $\mathcal C$,
such that the following axiom holds.
\begin{itemize}
\item[(HW)]
Let $P(\lambda)$ be a projective cover of $L(\lambda)$ in $\mathcal C$.
Define the {\em standard object}
$\Delta(\lambda)$ to be the largest quotient of $P(\lambda)$
such that $[\Delta(\lambda):L(\mu)] = \delta_{\lambda,\mu}$
for $\mu\not<\lambda$.
Then $P(\lambda)$ has a finite filtration with top section isomorphic to
$\Delta(\lambda)$ and other sections of the form $\Delta(\mu)$ for $\mu > \lambda$.
\end{itemize}
It is well known that this is equivalent to the axiom (HW$^*$)
below; in other words $\mathcal C$ is highest weight if and only if $\mathcal C^{\operatorname{op}}$ is highest weight.
\begin{itemize}
\item[(HW$^*$)]
Let $I(\lambda)$ be an injective hull of $L(\lambda)$ in $\mathcal C$.
Define the {\em costandard object}
$\nabla(\lambda)$ to be the largest subobject of $I(\lambda)$
such that $[\nabla(\lambda):L(\mu)] = \delta_{\lambda,\mu}$
for $\mu\not<\lambda$.
 Then $I(\lambda)$ has a finite filtration with bottom section isomorphic to
$\nabla(\lambda)$ and other sections of the form $\nabla(\mu)$ for $\mu > \lambda$.
\end{itemize}
\end{Definition}

If $\mathcal C$ is a highest weight category, we write
$\mathcal C^\Delta$ and $\mathcal C^\nabla$ for the exact subcategories
consisting of objects with a $\Delta$-flag 
and objects with a $\nabla$-flag, respectively.
Their complexified Grothendieck groups will be denoted
$[\mathcal C^\Delta]$ and $[\mathcal C^\nabla$];
they have distinguished bases $\{[\Delta(\lambda)]\:|\:\lambda\in\Lambda\}$
and $\{[\nabla(\lambda)]\:|\:\lambda\in\Lambda\}$, respectively.
The natural inclusion functors induce
linear maps
$[\mathcal C] \hookrightarrow [\mathcal C^\Delta] \hookrightarrow
[\mathcal C]^* \hookleftarrow [\mathcal C^\nabla]$.
When $\Lambda$ is finite all these maps are actually isomorphisms so
that all the Grothendieck groups are usually identified.

There are a couple of well-known
constructions which will be essential later on.
Suppose that we are given a decomposition $\Lambda = \Lambda_\lo \sqcup \Lambda_\up$ such that $\Lambda_\lo$ is an ideal (lower set); equivalently
$\Lambda_\up$ is a coideal (upper set).
Let $\mathcal C_\lo$ be the Serre subcategory of $\mathcal C$
generated by $\{L(\lambda)\:|\:\lambda \in \Lambda_\lo\}$.
We write $\iota:\mathcal C_\lo \hookrightarrow \mathcal C$ for the natural inclusion,
and $\iota^!$ (resp.\ $\iota^*$) for the left (resp.\ right) adjoint
to $\iota$ which sends an object $M$ to its largest quotient (resp.\ subobject)
belonging to $\mathcal C_\lo$.
The category $\mathcal C_\lo$ is itself a highest weight category with weight poset
$\Lambda_\lo$. Its irreducible, standard and costandard objects are the same as the ones in $\mathcal C$ indexed by the set $\Lambda_\lo$.
For $\lambda \in \Lambda_\lo$ the projective cover (resp.\ injective hull) of $L(\lambda)$
in $\mathcal C_\lo$ is $\iota^! P(\lambda)$ (resp.\ $\iota^* I(\lambda)$),
which will in general be a proper quotient of $P(\lambda)$ (resp.\ 
a proper subobject of $I(\lambda)$).
For any $M, N \in \mathcal C_\lo$ we have that
\begin{equation}\label{ext1}
\Ext^n_{\mathcal C}(M, N) \cong \Ext^n_{\mathcal C_\scriptlo}(M,N)
\end{equation}
for all $n \geq 0$.
This is proved by a Grothendieck spectral sequence argument
exactly like in \cite[A3.2--A3.3]{Donkin}.
A key step is to check that the higher right derived functors $R^n \iota^*$
vanish on objects from $\mathcal C^\nabla$;
dually the higher left derived functors $L_n \iota^!$ vanish on objects from 
$\mathcal C^\Delta$.

As well as the subcategory $\mathcal C_\lo$, we can consider the Serre
quotient category
$\mathcal C_\up := \mathcal C / \mathcal C_\lo$;
we stress that according to 
the definition of quotient category
the objects of $\mathcal C_\up$ are {\em the same} 
as the objects of $\mathcal C$; 
morphisms $M \rightarrow N$ in $\mathcal C_\up$
are obtained by taking a direct limit of the morphisms 
$M' \rightarrow N / N'$
 in $\mathcal C$ over all subobjects $M'$ of $M$ and $N'$ of $N$ such
 that $M / M'$ and $N'$ belong to $\mathcal C_\lo$.
Let $\pi:\mathcal C \rightarrow \mathcal C_\up$ be the quotient functor,
and fix a choice $\pi^!$ (resp.\ $\pi^*$) of a left (resp.\ right)
adjoint to $\pi$.
Note that the unit (resp.\ counit) of adjunction gives a canonical isomorphism
$\operatorname{id}\stackrel{\sim}{\rightarrow}\pi \circ \pi^!$
(resp.\ $ \pi \circ \pi^*\stackrel{\sim}{\rightarrow}\operatorname{id}$).
The irreducible, standard, costandard, indecomposable 
projective and indecomposable injective objects in $\mathcal C_\up$ are 
the same as the ones in $\mathcal C$ indexed by weights from $\Lambda_\up$.
Also for $\lambda \in \Lambda_\up$ we have that
$\pi^! P(\lambda) \cong P(\lambda),
\pi^* I(\lambda) \cong I(\lambda)$,
$\pi^! \Delta(\lambda) \cong \Delta(\lambda)$
and $\pi^* \nabla(\lambda)\cong\nabla(\lambda)$ in $\mathcal C$;
the first two isomorphisms here follow from properties of adjunctions;
see Lemma~\ref{shriek} below for justification of the latter two.
Finally
for $M, N \in \mathcal C$ such that 
either $M$ has a $\Delta$-flag with sections of the form $\Delta(\lambda)$ indexed by weights $\lambda \in \Lambda_\up$,
or $N$ has a $\nabla$-flag with sections 
$\nabla(\lambda)$ for $\lambda \in \Lambda_\up$,
we have that
\begin{equation}\label{ext2}
\Ext^n_{\mathcal C}(M,N) \cong \Ext^n_{\mathcal C_\scriptup}(M,N)
\end{equation}
for all $n \geq 0$.
This is \cite[A3.13]{Donkin}.


\begin{Lemma}\label{shriek}
Let $\pi:\mathcal C \rightarrow \mathcal C_\up$ be the quotient
associated to a coideal $\Lambda_\up \subseteq \Lambda$.
For $\lambda \in \Lambda_\up$ there are canonical isomorphisms
$\pi^! \Delta(\lambda) \cong \Delta(\lambda)$
and $\nabla(\lambda)
 \cong \pi^* \nabla(\lambda)$
in $\mathcal C$
induced by the counit 
and unit 
of the fixed adjunctions.
\end{Lemma}

\begin{proof}
Let $\mathcal C_{\leq \lambda}$ (resp.\ $\mathcal C_{< \lambda}$) be the highest weight
subcategory of $\mathcal C$ associated to the ideal 
$\{\mu \in \Lambda\:|\:\mu \leq \lambda\}$ (resp. 
$\{\mu \in \Lambda\:|\:\mu < \lambda\}$).
Let $\mathcal C_\lambda := \mathcal C_{\leq \lambda} / \mathcal C_{< \lambda}$.
This category is a copy of $\mathcal{V}ec$ with unique (up to isomorphism) irreducible object $L(\lambda)$.
Let $\pi_\lambda:\mathcal C_{\leq \lambda} \rightarrow \mathcal C_\lambda$ be the quotient functor with left adjoint $\pi_\lambda^!$. The projective cover of $L(\lambda)$ in $\mathcal C_{\leq \lambda}$ is $\Delta(\lambda)$, hence 
by properties of adjunctions we have that $\Delta(\lambda) \cong \pi_\lambda^! L(\lambda)$ in $\mathcal C$.

Similarly, working with $\mathcal C_{\up}$ in place of $\mathcal C$, 
we define 
subcategories 
$\mathcal C_{\up,\leq \lambda}$
and $\mathcal C_{\up, < \lambda}$.
The quotient $\mathcal C_{\up,\leq \lambda}/ \mathcal C_{\up, < \lambda}$
is another copy of $\mathcal{V}ec$, hence is equivalent to $\mathcal C_\lambda$.
This means that there is another quotient functor $\pi_{\up,\lambda}:
\mathcal C_{\up,\leq \lambda} \rightarrow \mathcal C_\lambda$
such that $\pi_\lambda = \pi_{\up,\lambda}\circ \pi$,
hence $\pi_\lambda^! \cong \pi^! \circ \pi_{\up,\lambda}^!$.
Again we have that
$\Delta(\lambda) \cong \pi_{\up,\lambda}^! L(\lambda)$
in $\mathcal C_\up$.
Hence we get isomorphisms in $\mathcal C$: 
$$
\pi^! \Delta(\lambda) \cong 
\pi^!(\pi_{\up,\lambda}^! L(\lambda)) \cong
\pi_\lambda^! L(\lambda)
\cong \Delta(\lambda).
$$
It remains to observe that 
the counit $\pi^! \Delta(\lambda)
= \pi^! (\pi \Delta(\lambda))\rightarrow \Delta(\lambda)$
is an epimorphism as $\Delta(\lambda)$ has irreducible head $L(\lambda)$
and $\lambda \in \Lambda_\up$; hence this 
gives a canonical 
choice for the isomorphism.

The argument for $\nabla$ is similar.
\end{proof}

\subsection{Tensor product categorifications}\label{stpc}
Suppose we are given a type $(\bn,\bc)$ of level $l$. Recall the $\mathfrak{sl}_I$-module
$\bigwedge^{\bn,\bc}V_I$ from (\ref{tp}).

\begin{Definition}\rm\label{tpcdef}
An {\em $\mathfrak{sl}_I$-tensor product categorification} of {type} $(\bn,\bc)$ 
means a highest weight category
$\mathcal C$
together with an endofunctor $F$ of $\mathcal C$, a right adjoint $E$ to
$F$ (with specified adjunction), and natural transformations $x
\in \End(F)$ and $t \in \End(F^2)$
satisfying 
axioms (SL1)--(SL3)
and (TP1)--(TP3).
\begin{itemize}
\item[(TP1)]
The weight poset $\Lambda$ is the set $\Lambda_{I;\bn,\bc}$ from (\ref{lal})
partially ordered by $\lambda \leq \mu$ if and only if
$|\lambda|=|\mu|$ and
$|\lambda_1|+\cdots+|\lambda_k| \geq |\mu_1|+\cdots+|\mu_k|$
for all $k$.
\item[(TP2)]
The exact functors $F_i$ and $E_i$ send objects with $\Delta$-flags to objects with $\Delta$-flags.
\item[(TP3)]
The linear isomorphism
$[\mathcal C^\Delta] \stackrel{\sim}{\rightarrow} \bigwedge^{\bn,\bc} V_I,
[\Delta(\lambda)] \mapsto v_\lambda$ intertwines the endomorphisms
$f_i$ and $e_i$ of $[\mathcal C^\Delta]$ induced by 
$F_i$ and $E_i$ with the endomorphisms of $\bigwedge^{\bn,\bc} V_I$ arising from
the actions of the Chevalley generators $f_i$ and $e_i$
of $\mathfrak{sl}_I$.
\end{itemize}
Since $[\mathcal C]$ embeds into $[\mathcal C^\Delta] \cong
\bigwedge^{\bn,\bc} V_I$,
we deduce immediately from the axioms that $[\mathcal C]$ is itself an
integrable $\mathfrak{sl}_I$-module, i.e.
the axiom (SL4) holds automatically.
Thus tensor product categorifications are
categorifications in the sense of Definition~\ref{catdef} too.
\end{Definition}

\begin{Remark}\rm
This definition is a slightly modified version of \cite[Definition
3.2]{LW}, where a general notion of tensor product
categorification for arbitrary
Kac-Moody algebras was introduced.
The definition in \cite{LW} is expressed in terms of
quiver Hecke algebras rather than affine Hecke algebras; but of course
the above definition 
can be formulated equivalently with the axioms (SL1$'$)--(SL3$'$)
replacing (SL1)--(SL3); so this is a superficial difference. 
More significantly, in our formulation of the axioms (TP2)--(TP3), 
we have incorporated the explicit monomial basis $\{v_\lambda\:|\:\lambda \in \Lambda\}$ which is
only available in our special minuscule situation.
The analogous axioms (TPC2)--(TPC3) in \cite{LW} 
are couched in terms of some commuting categorical
$\mathfrak{sl}_I$-actions
on the associated graded category $\operatorname{gr} \mathcal C := \bigoplus_{\lambda \in
  \Lambda} \mathcal C_\lambda$ (where $\mathcal C_\lambda$ is as in the proof of Lemma~\ref{shriek}).
The functors ${_i}F_j$ defining these actions
can be recovered by taking a sum of equivalences
$\mathcal C_\lambda \rightarrow \mathcal C_{t_{ij}(\lambda)}$
for all $\lambda \in \Lambda$ such that $\lambda_{ij}=1$ and
$\lambda_{i(j+1)}=0$, where $t_{ij}(\lambda)$ is obtained from $\lambda$
by interchanging $\lambda_{ij}$ and $\lambda_{i(j+1)}$.
Such functors exist since for a highest weight category
each $\mathcal C_\lambda$ is equivalent to 
$\mathcal{V}ec$.
\end{Remark}

Any $\mathfrak{sl}_I$-tensor product categorification decomposes as
\begin{equation}\label{wtdec}
\mathcal C = \bigoplus_{\varpi \in P_I} \mathcal C_\varpi
\end{equation}
where $\mathcal C_\varpi$ is the Serre subcategory of $\mathcal C$ generated
by the irreducible objects $\{L(\lambda)\:|\:\lambda\in\Lambda, |\lambda|=\varpi\}$.
In particular, two irreducible objects $L(\lambda)$
and $L(\mu)$ belong to the same block of $\mathcal C$ only if
$|\lambda|=|\mu|$; see Theorem~\ref{blockclass} for the converse.

Given another type $(\bn',\bc')$ of the same level, we say that $(\bn,\bc)$ and
$(\bn',\bc')$
are {\em equivalent} if one of the following holds for each $i$:
either $c_i=c_i'$ and $n_i = n_i'$; or $I$ is finite,
$c_i \neq c_i'$ and
$n_i = |I_+|-n_i'$.
Observe in that case that the posets of $01$-matrices
$\Lambda_{I;\bn,\bc}$ 
and $\Lambda_{I;\bn',\bc'}$
are simply {\em equal},
and there is an $\mathfrak{sl}_I$-module isomorphism
$\bigwedge^{\bn,\bc}V_I \stackrel{\sim}{\rightarrow} \bigwedge^{\bn',\bc'}V_I, v_\lambda
\mapsto v_\lambda$.

We can now state the first main result of the article.

\begin{Theorem}\label{lwmain}
For any interval $I \subseteq \Z$ and type $(\bn,\bc)$, 
there exists an $\mathfrak{sl}_I$-tensor product categorification $\mathcal C$ of
type $(\bn,\bc)$. 
Moreover $\mathcal C$ is unique in the sense that if
 $\mathcal C'$ is another tensor product categorification 
of an equivalent type
$(\bn',\bc')$ then
there is 
a strongly equivariant equivalence 
$\mathbb{G}:\mathcal C \stackrel{\sim}{\rightarrow} \mathcal C'$
with $\mathbb{G} L(\lambda) \cong L'(\lambda)$
for each weight $\lambda$.
\end{Theorem}

In the case that $I$ is finite, Theorem~\ref{lwmain} is 
a special case of the main
result of \cite{LW}; see $\S$\ref{finite} for some further discussion of that.
For infinite intervals, Theorem~\ref{lwmain} is new and will be proved later in the article.
Specifically, 
we will establish existence for $I = \Z$ in $\S$\ref{secsuper},
then existence for the other infinite but
bounded above or below intervals follows by the truncation argument
explained in $\S$\ref{trun}.
The uniqueness will be established in $\S$\ref{uniq}.

\begin{Corollary}\label{duality}
Any $\mathfrak{sl}_I$-tensor product categorification $\mathcal C$ admits a 
duality
$\circledast$
such that 
$F_i \circ \circledast \cong \circledast \circ F_i, E_i \circ \circledast \cong \circledast \circ E_i$ and
$L(\lambda) \cong L(\lambda)^\circledast$ for each weight
$\lambda$.  Similarly its category of projectives has a 
duality $\#$ such that 
$F_i \circ \# \cong \# \circ F_i, E_i \circ \# \cong \# \circ E_i$ and
$P(\lambda) \cong P(\lambda)^\#$ for each $\lambda$.
\end{Corollary}

\begin{proof}
Using the homological criteria for $\Delta$- and $\nabla$-flags,
one checks that the axioms (TP2)--(TP3) are equivalent to the axioms
(TP2$^*$)--(TP3$^*$) obtained from them by replacing all occurrences of
$\Delta$ with $\nabla$. In other words
$\mathcal C$ is a tensor product categorification if and only if $\mathcal C^{\operatorname{op}}$ is one; when $I$ is finite this assertion is \cite[Proposition 3.9]{LW}.
Now apply the uniqueness from Theorem~\ref{lwmain} with $\mathcal C'
:= \mathcal C^{\operatorname{op}}$ to get $\circledast$.

To obtain the duality $\#$ on projectives,
one can use (\ref{yoneda}) to reduce to the problem of
defining a duality $\#$
on the subcategory of $\operatorname{Fun}_f(\pC, \mathcal{V}ec^{\operatorname{op}})$ consisting of all exact functors;
there one sets
$\Hom_{\mathcal C}(P, -)^\# := * \circ \Hom_{\mathcal C}(P, -) \circ \circledast$
(where the final $*$ is the duality on $\mathcal{V}ec$).
Transporting through the Yoneda equivalence this yields a
duality $\#$ on $\pC$ such that
\begin{equation}
\Hom_{\mathcal C}(P^\#,M)\cong \Hom_{\mathcal C}(P,M^\circledast)^*\label{star-rep}
\end{equation}
for all $M \in \mathcal C$.
It is clear from (\ref{star-rep}) that $P(\lambda)^\# \cong P(\lambda)$, while
the fact that 
$\#$ commutes with $F_i$ and $E_i$ follows
by adjunction as $\circledast$ commutes with $E_i$ and $F_i$.
(Alternatively this definition can be understood via (\ref{stupid1})
in terms of the algebra $A$:
it corresponds to the composition $\circledast \circ \mathcal{N}:\rproj{A}\rightarrow \rproj{A}$
where $\mathcal{N}$ is the Nakayama functor $\Hom_A(-,A)^*:\rproj{A}
\rightarrow \rinj{A}$,
and $\rproj{A}$ and $\rinj{A}$ denote the categories of projective and
injective $A$-modules, respectively.)
\end{proof}

\subsection{Review of the proof of Theorem~\ref{lwmain} for finite intervals}\label{finite}
In this subsection, we assume that $I$ is finite and 
recall for future reference some of the key ideas behind the proof of
Theorem~\ref{lwmain} from \cite{LW}.
Suppose we are given a type $(\bn,\bc)$.
To avoid trivialities we assume that $n_i \leq |I_+|$ for each $i$.
There are two general approaches to the construction of the tensor
product categorification $\mathcal C$ in Theorem~\ref{lwmain}. First
it can be realized in terms of certain blocks of the parabolic
category $\mathcal O$ associated to the general linear Lie algebra;
see \cite[Definition 3.13]{LW}.
Alternatively, $\mathcal C$
can be constructed using 
the tensor product algebras of \cite{W1}; see \cite[Theorem 3.12]{LW}.

Turning our attention to uniqueness, let $\mathcal C$ be some given 
$\mathfrak{sl}_I$-tensor product categorification with weight poset
$\Lambda = \Lambda_{I, \bn,\bc}$.
Let $A$ be the algebra 
\begin{equation}\label{A1}
A := 
\bigoplus_{\lambda,\mu\in\Lambda} \Hom_{\mathcal C}
(P(\lambda), P(\mu))
\end{equation}
from (\ref{assertion}),
and
$\mathbb{H}:\mathcal C \rightarrow \rmod{A}$
be the canonical equivalence of categories from (\ref{stupid1}).
There is a formal way to transport the categorical
$\mathfrak{sl}_I$-action
from $\mathcal C$ to $\rmod{A}$ in such a way that
$\mathbb{H}:
\mathcal C\rightarrow \rmod{A}$ becomes a strongly equivariant
equivalence.
The appropriate functor $F:\rmod{A} \rightarrow \rmod{A}$ is the functor
defined by tensoring over $A$ with the $(A,A)$-bimodule
\begin{equation}\label{B1}
B :=
\bigoplus_{\lambda,\mu \in \Lambda} \Hom_{\mathcal C}(P(\lambda), F
P(\mu)).
\end{equation}
The natural transformations $x \in \End(F)$
and $t \in \End(F^2)$ come from
bimodule endomorphisms $x:B \rightarrow B$ and $t:B \otimes_A B \rightarrow B \otimes_A B$ defined as follows:
let $x:B \rightarrow B$ be defined on the summand $\Hom_{\mathcal
  C}(P(\lambda), FP(\mu))$ of $B$ by composing
with 
$x_{P(\mu)}:F P(\mu) \rightarrow F P(\mu)$;
let $t:B \otimes_A B \rightarrow B \otimes_A B$ be induced similarly by
$t_{P(\mu)}:F^2 P(\mu) \rightarrow F^2 P(\mu)$
using also the following canonical isomorphism
$$
B \otimes_A B \cong 
\bigoplus_{\lambda,\mu \in \Lambda} \Hom_{\mathcal C}(P(\lambda),F^2 P(\mu)).
$$
Then we may take $E:\rmod{A}\rightarrow \rmod{A}$ to be the canonical
right adjoint to $F$ given by the functor $\bigoplus_{\lambda \in
  \Lambda} \Hom_{A}(1_\lambda B, -)$.
In this way we have made explicit the categorical $\mathfrak{sl}_I$-action
on $\rmod{A}$.

The strategy for the proof of uniqueness is as follows. Suppose that we are 
given another $\mathfrak{sl}_I$-categorification $\mathcal C'$
of an equivalent type $(\bn',\bc')$. We repeat all of the above,
defining
its associated basic algebra 
\begin{equation}\label{A2}
A' := \bigoplus_{\lambda,\mu\in\Lambda} \Hom_{\mathcal C'}
(P'(\lambda), P'(\mu)),
\end{equation}
and an $(A',A')$-bimodule
\begin{equation}\label{B2}
B' := \bigoplus_{\lambda,\mu \in \Lambda} \Hom_{\mathcal C'}(P'(\lambda), F' P'(\mu))
\end{equation}
together with endomorphisms $x':B' \rightarrow B'$ and
$t':B'\otimes_{A'} B'$ leading to a categorical $\mathfrak{sl}_I$-action on
$\rmod{A'}$ too, such that the equivalence
$\mathbb{H}':\mathcal C'\rightarrow
\rmod{A'}$ is strongly equivariant.
Then the point is to construct an algebra isomorphism $A \cong A'$,
inducing an isomorphism of categories $\rmod{A}
\stackrel{\sim}{\rightarrow} \rmod{A'}$.
To see that this isomorphism of categories is strongly
equivariant, we must also define an isomorphism $B \cong B'$ that
intertwines the actions of $A, x$ and $t$ with $A', x'$
and $t'$.
Composing 
the isomorphism 
$\rmod{A}\stackrel{\sim}{\rightarrow} \rmod{A'}$ on one side with $\mathbb{H}$ and with the canonical
adjoint equivalence to $\mathbb{H}'$ on the other, we
obtain the
desired strongly equivariant equivalence $\mathbb{G}:\mathcal C\rightarrow\mathcal C'$
from the statement of Theorem~\ref{lwmain}.

Let us begin.
Recall that $\kappa = \kappa_{I;\bn,\bc}$ is the $01$-matrix indexing
the basis vector of maximal
weight in $\bigwedge^{\bn,\bc}V_I$.
The irreducible object $L(\kappa)$ is the only irreducible in its
block, i.e. 
$L(\kappa)=\Delta(\kappa)=\nabla(\kappa)=P(\kappa)=I(\kappa)$.
Since the functor $F$ has both a left and right adjoint it sends prinjectives to prinjectives, hence the object
\begin{equation}\label{prinjec}
T = \bigoplus_{d \geq 0} T_d := \bigoplus_{d \geq 0}F^d L(\kappa) \in \mathcal C
\end{equation}
is prinjective.
The modules $T_d$ and $T_{d'}$ for $d \neq
d'$ belong to different sums of the blocks from (\ref{wtdec}), hence
we have that $\Hom_{\mathcal C}(T_d, T_{d'}) = 0$ for $d \neq d'$.
We say that $M \in \mathcal C$ is {\em homogeneous of degree $d$} if it
belongs to the same sum of blocks as $T_d$; then we have that $\Hom_{\mathcal
  C}(T,M) = \Hom_{\mathcal C}(T_d,M)$.
Note further that $T_d = 0$ for $d \gg 0$.
Let
\begin{equation}
H = \bigoplus_{d \geq 0} H_d :=
\bigoplus_{d \geq 0} AH_{I,d}^{|\kappa|}.
\end{equation}
The following theorem is at the heart of everything; 
see \cite[Proposition 3.2]{LW} for the first assertion,
\cite[Theorem 5.1]{LW} for the second, and the proof of \cite[Theorem 6.1]{LW} for the final one; in the special case that $\mathcal C$ 
is the tensor product categorification arising from parabolic category $\mathcal O$ from \cite[Definition 3.13]{LW} the results here go back to \cite{SWD}.

\begin{Theorem}[\cite{LW}]\label{dcp}
The action of 
$AH_d$ 
on $T_d$
induces a canonical isomorphism between $H_d$ and
$\End_{\mathcal C}(T_d)$; hence $H = \End_{\mathcal C}(T)$.
Moreover the exact functor
\begin{equation}
\V:= \Hom_{\mathcal C}(T, -):\mathcal C \rightarrow
\rmod{H}
\end{equation}
is fully faithful on projectives.
Finally for each weight $\lambda \in \Lambda$
the $H$-module 
\begin{equation}
Y(\lambda) := \V P(\lambda)
\end{equation}
is independent (up to isomorphism) of
the particular choice of $\mathcal C$.
\end{Theorem}

\begin{Remark}\label{stronger}\rm
The proof of \cite[Theorem 5.1]{LW}
establishes a slightly stronger result:
the map $\V:\Hom_{\mathcal C}(M, P)
\rightarrow \Hom_H(\V M, \V P)$ is an isomorphism
for any $M, P \in \mathcal C$ with $P$ projective.
\end{Remark}

Thus the functor $\V$
has similar properties to Soergel's combinatorial functor $\mathbb{V}$
from \cite{Soergel}.
The modules $Y(\lambda)$ may be called {\em Young modules} by analogy with the modular representation theory of symmetric groups.
The second assertion of Theorem~\ref{dcp} is a version of the
{\em double centralizer property}, which has already appeared in
numerous related contexts in representation theory; see e.g. \cite[Example
2.7]{MS} where several are listed.
It implies that the functor $\V$ defines an algebra
isomorphism
\begin{equation}\label{C1}
A \cong \bigoplus_{\lambda,\mu \in \Lambda} \Hom_{H}(Y(\lambda),
Y(\mu)).
\end{equation}
Similarly for the primed category we get that
\begin{equation}\label{C2}
A' \cong \bigoplus_{\lambda,\mu \in\Lambda} \Hom_H(Y'(\lambda), 
Y'(\mu))
\end{equation}
where $Y'(\lambda) := \V' P'(\lambda)$ for $\V'$
defined analogously to $\V$.
Then, applying the final assertion of Theorem~\ref{dcp}, we
choose $H$-module isomorphisms $Y(\lambda) \cong Y'(\lambda)$
for each $\lambda$. These choices induce the desired algebra isomorphism $A
\cong A'$.

It remains to construct the bimodule isomorphism $B \cong B'$. This needs just a little more preparation.
The category $\rmod{H}$ is also equipped with a categorical
$\mathfrak{sl}_I$-action.
The endofunctors 
\begin{equation}\label{forthis}
F:\rmod{H}\rightarrow \rmod{H},
\qquad
E:\rmod{H}\rightarrow \rmod{H}
\end{equation}
for this 
are the induction and restriction functors
associated to the homomorphisms $H_d \rightarrow H_{d+1}$ induced by the 
natural inclusions
$AH_d \hookrightarrow AH_{d+1}$ for all $d \geq 0$; so 
for a right $H_d$-module $M$ we have that
$F M := \operatorname{ind}_{H_d}^{H_{d+1}} M = M \otimes_{H_d} H_{d+1}$
and $E M := \operatorname{res}^{H_d}_{H_{d-1}} M$.
The canonical adjunction makes $(F, E)$ into an adjoint pair.
Left multiplication by $x_{d+1}$ defines an $(H_d, H_{d+1})$-bimodule 
endomorphism of 
$H_{d+1}$, from which we obtain the 
natural transformation $x \in \End(F)$.
Also, by transitivity of induction, $F^2$
is isomorphic to the functor sending a right $H_d$-module $M$ to
$M \otimes_{H_d} H_{d+2}$; then left multiplication by $t_{d+1}$
defines an $(H_d, H_{d+2})$-bimodule endomorphism of $H_{d+2}$
inducing $t \in \End(F^2)$.
This gives us the data of an $\mathfrak{sl}_I$-categorification in the
sense of Definition~\ref{catdef}. The fact that the axioms
(SL1)--(SL4) hold goes back at least to
\cite[Remark 7.13]{CR}; see also \cite[Corollary 7.7.5]{Kbook} for the proof that $F$ is isomorphic to
a right adjoint of $E$.

The following lemma was noted already in the first paragraph of
\cite[$\S$5.1]{LW}; the alternative proof given below is a bit more
explicit.

\begin{Lemma}\label{seq}
The quotient functor $\V:\mathcal C \rightarrow \rmod{H}$ is strongly equivariant.
\end{Lemma}

\begin{proof}
First we construct the required natural transformation
$\zeta:F \circ \V \rightarrow \V \circ F$.
Take $M \in \mathcal C$ that is homogeneous of degree $d$.
We need to produce a 
natural $H_{d+1}$-module homomorphism
$\zeta_M:\Hom_{\mathcal C}(T_d, M) \otimes_{H_d} H_{d+1} \rightarrow
\Hom_{\mathcal C}(T_{d+1}, F M)$.
The functor $F$ defines a natural $H_d$-module homomorphism from $\Hom_{\mathcal C}(T_d, M)$
to the restriction of the $H_{d+1}$-module $\Hom_{\mathcal C}(T_{d+1},
FM)$.
Then we use the adjunction of induction and restriction to convert this
into the desired homomorphism
$\zeta_M$.

Next we show that $\zeta$ is
an isomorphism of functors.
It is certainly an isomorphism on $T$
as for that it reduces to an identity map.
Hence it is an isomorphism on any direct sum of summands of $T$.
By \cite[Lemma 5.3]{LW}, any projective object $P\in \mathcal C$ 
fits into an exact sequence
$0 \rightarrow P \rightarrow J \rightarrow K$ 
such that $J$ and $K$ are direct sums of summands of $T$.
Note further that both $F \circ \V$ and $\V \circ F$ are exact functors.
Hence when we apply $\zeta$ to our exact sequence we obtain a commuting diagram
with exact rows:
$$
\begin{CD}
0&@>>>&F \circ \V(P)&@>>>&F \circ \V(J)&@>>>&F \circ \V(K)\\
&&&&@VVV&&@VVV&&@VVV\\
0&@>>>&\V \circ F(P)&@>>>&\V \circ F(J)&@>>>&\V \circ F(K).
\end{CD}
$$
We know already that the right hand vertical maps are isomorphisms,
hence so too is the first one. Now we have proved that $\zeta$ 
defines an isomorphism on every projective object. For an arbitrary
object $M$ we pick a projective resolution $Q \rightarrow P
\rightarrow M \rightarrow 0$ and make another argument with the Five Lemma.

It remains to check the axioms (E1)--(E3) from Definition~\ref{sev}.
For (E1), we observe on some homogeneous $M \in \mathcal C$ of degree $(d+1)$
that
the natural transformation $E \V \eps \circ E \zeta E \circ \eta \V E$
defines the $H_d$-module homomorphism
$$
\Hom_{\mathcal C}(T_d, EM) \rightarrow \Hom_{\mathcal C}(F T_d, M)
$$ 
defined by the given
adjunction between $F$ and $E$. Hence it is an isomorphism.
For (E2),
assume that $M \in \mathcal C$ is homogeneous of degree $d$.
Then it suffices to show that the following diagram commutes:
$$
\begin{CD}
\Hom_{\mathcal C}(T_d, M) \otimes_{H_d} H_{d+1}&@>(x \V)_M>>&
\Hom_{\mathcal C}(T_d, M) \otimes_{H_d} H_{d+1}\\
@V\zeta_MVV&&@VV\zeta_M V\\
\Hom_{\mathcal C}(T_{d+1}, FM)&@>(\V x)_M>>&
\Hom_{\mathcal C}(T_{d+1}, FM).
\end{CD}
$$
Take $\theta \in \Hom_{\mathcal C}(T_d, M)$ and $h \in H_{d+1}$.
Going south then east, $\theta\otimes h$ maps
to the homomorphism $T_{d+1} \rightarrow FM,
v \mapsto x_M ((F \theta)(hv))$, while going east then south produces
the homomorphism $v \mapsto (F \theta)(x_{d+1} hv)$.
These are equal by the naturality of $x:F \rightarrow F$ with respect to the homomorphism
$\theta:T_d \rightarrow M$.
The proof of (E3) is similar.
\end{proof}

Applying Lemma~\ref{seq}, we deduce that the functor $\V$ defines
$(A,A)$-bimodule isomorphisms
\begin{align}\label{D1}
B &\cong \bigoplus_{\lambda,\mu \in \Lambda} \Hom_H(Y(\lambda),
FY(\mu)),\\
B \otimes_A B &\cong \bigoplus_{\lambda,\mu \in \Lambda} \Hom_H(Y(\lambda), 
F^2 Y(\mu)).\label{E1}
\end{align}
Under these isomorphisms, 
$x:B \rightarrow B$ and $t:B \otimes_A B \rightarrow B \otimes_A B$
correspond to the endomorphisms of the bimodules on the right
induced by all of the homomorphisms $x_{Y(\mu)}:F Y(\mu)\rightarrow FY(\mu)$
and $t_{Y(\mu)}:F^2 Y(\mu) \rightarrow F^2 Y(\mu)$, respectively.
Similarly
\begin{align}\label{D2}
B' &\cong \bigoplus_{\lambda,\mu \in \Lambda} \Hom_H(Y'(\lambda), FY'(\mu)),
\\
B' \otimes_A B' &\cong \bigoplus_{\lambda,\mu \in \Lambda} \Hom_H(Y'(\lambda), F^2 Y'(\mu)).\label{E2}
\end{align}
Then the $H$-module isomorphisms $Y(\lambda) \cong Y'(\lambda)$ chosen
earlier
induce the desired isomorphism $B \cong B'$.
It is immediate that it intertwines the actions of $A, x$ and $t$ with
$A', x'$ and $t'$.
This completes our sketch of the proof of uniqueness in
Theorem~\ref{lwmain} for finite intervals.

\subsection{Truncation}\label{trun}
In this subsection we introduce our key tool for proving results about 
tensor product categorifications when $I$ is infinite.
Throughout we fix a type $(\bn,\bc)$ and any interval $I$,
and set $\Lambda := \Lambda_{I;\bn,\bc}$.
Given a subinterval $J \subseteq I$,
there is an obvious embedding $\mathfrak{sl}_J \hookrightarrow
\mathfrak{sl}_I$.
Let $\Lambda_J$ be the subposet of $\Lambda$ consisting 
of all $01$-matrices
$\lambda$ such that $\lambda_{ij} = c_i$ whenever $j \notin J_+$.
This is order-isomorphic to the poset $\Lambda_{J;\bn,\bc}$
via the map sending $\lambda =(\lambda_{ij})_{1 \leq i \leq l,
  j \in I_+} \in \Lambda_J$ to its submatrix
$\lambda_J := (\lambda_{ij})_{1 \leq i \leq l, j \in J_+} \in \Lambda_{J;\bn,\bc}$. 
In turn, the $\mathfrak{sl}_J$-module $\bigwedge^{\bn,\bc} V_J$
can be identified with the
$\mathfrak{sl}_J$-submodule of $\bigwedge^{\bn,\bc} V_I$
spanned by 
$\{v_\lambda\:|\:\lambda \in \Lambda_J\}$.
We then have that
$$
\textstyle
\bigwedge^{\bn,\bc} V_I = \bigcup_J \bigwedge^{\bn,\bc} V_J,
$$ 
taking the union just over the finite
subintervals $J \subseteq I$.
We are going to develop a categorical analog of this decomposition.

\begin{Lemma}\label{pset}
For $\lambda,\mu \in \Lambda$, the following are equivalent:
\begin{itemize}
\item[(i)]
$\lambda \leq \mu$;
\item[(ii)] 
for all $h \in I$ and $1\leq k\leq l$ we have that 
$\displaystyle
\sum_{i=1}^k \!\!\sum_{\substack{j \leq h \\ \lambda_{ij} \neq c_i}} (-1)^{c_i}\!
\geq \sum_{i=1}^k \!\!\sum_{\substack{j \leq h \\ \mu_{ij} \neq c_i}} (-1)^{c_i},$
with equality when $k=l$;
\item[(iii)]
for all $h\in I$ and $1 \leq k \leq l$ we have that 
$\displaystyle
\sum_{i=1}^k\!\! \sum_{\substack{j > h \\ \lambda_{ij} \neq c_i}} (-1)^{c_i}
\!\leq \sum_{i=1}^k \!\!\sum_{\substack{j > h \\ \mu_{ij} \neq c_i}} (-1)^{c_i}$,
with equality when $k=l$.
\end{itemize}
\end{Lemma}

\begin{proof}
The equivalence of (i) and (ii) follows from Lemma~\ref{dominance}.
The equivalence of (ii) and (iii) is obvious.
\end{proof}

Again let $J \subseteq I$ be any subinterval.
Let 
$\Lambda_{\leq J}$ 
denote the set of all $\lambda \in \Lambda$ which satisfy the conditions
\begin{equation}\label{ineq}
\left\{
\begin{array}{ll}
\displaystyle
\sum_{i=1}^k \sum_{\substack{j \leq h \\ \lambda_{ij} \neq c_i}} (-1)^{c_i}
\geq 0&
\text{for all $h < \min(J)$ and $1\leq k\leq l$,}
\\
\displaystyle
\sum_{i=1}^k \sum_{\substack{j > h \\ \lambda_{ij} \neq c_i}} (-1)^{c_i}
\leq 0&
\text{for all $h > \max(J)$ and $1 \leq k\leq l$.}\\
\end{array}
\right.
\end{equation}
Also let $\Lambda_{< J}$ denote the
set of all $\lambda \in \Lambda_{\leq J}$
such that at least one of the above inequalities is strict.
Lemma~\ref{pset} implies that both $\Lambda_{\leq J}$
and $\Lambda_{< J}$ are ideals in the poset $\Lambda$.
Moreover $\Lambda_J = \Lambda_{\leq J} \setminus \Lambda_{<J}$.

Suppose next that we are given an $\mathfrak{sl}_I$-tensor product 
categorification $\mathcal C$ of type $(\bn,\bc)$.
Let $\mathcal C_{\leq J}$ (resp.\ $\mathcal C_{< J}$) 
be the highest weight subcategory of
$\mathcal C$ associated to the ideal $\Lambda_{\leq J}$
(resp. $\Lambda_{< J}$).
Let $\mathcal C_J := \mathcal C_{\leq J} / \mathcal C_{< J}$.
We denote the 
quotient functor by $\pi_J:\mathcal C_{\leq J} \rightarrow \mathcal C_J$.

\begin{Lemma}
For $j \in J$
the functors $F_j$ and $E_j$ preserve the subcategories
$\mathcal C_{\leq J}$ and $\mathcal C_{< J}$ of $\mathcal C$.
\end{Lemma}

\begin{proof}
We just explain for $\mathcal C_{\leq J}$; the same argument works for
$\mathcal C_{< J}$.
Take any $\lambda \in \Lambda_{\leq J}$.
We need to show that $F_j L(\lambda)$ and $E_j L(\lambda)$
both belong to $\mathcal C_{\leq J}$.
Since $L(\lambda)$ is a quotient of $\Delta(\lambda)$,
this follows if
we can show that $F_j \Delta(\lambda)$ and $E_j \Delta(\lambda)$ 
belong to $\mathcal C_{\leq J}$.
These objects have filtrations
with sections of the form $\Delta(\mu)$
for weights $\mu$
obtained from $\lambda$ by applying the transposition $t_j$ to one of
its rows.
The integers on the left hand side of the inequalities
(\ref{ineq}) are the same for each of these $\mu$ as they are for $\lambda$,
so that each $\mu$ arising is an element of $\Lambda_{\leq J}$
and $\Delta(\mu)$ does indeed belong to $\mathcal C_{\leq J}$.
\end{proof}

Hence for $j \in J$ 
the functors $F_j$ and $E_j$
induce a well-defined biadjoint pair of 
endofunctors of $\mathcal C_J$. 
Let $F_J := \bigoplus_{j \in J} F_j$ 
and $E_J := \bigoplus_{j \in J}
E_j$.
The natural transformations $x$ and $s$ restrict to 
endomorphisms of $F_J$ and $F_J^2$, respectively, such that
the associated endomorphisms $x_j$ and $t_k$ of $F_J^d$ 
satisfy the
degenerate affine Hecke algebra relations as in (SL2).
The axioms (TP1)--(TP3) for $\mathcal C$
imply the analogous statements for $\mathcal C_J$.
Thus we have proved:

\begin{Theorem}\label{subq}
The subquotient $\mathcal C_J$ of $\mathcal C$ equipped with the endofunctors $F_J$ and $E_J$
is an $\mathfrak{sl}_J$-tensor
product categorification of type $(\bn,\bc)$.
\end{Theorem}

We record one more technical lemma for later use.

\begin{Lemma}\label{we}
Suppose that $I$ is infinite and $J \subset I$ is a finite subinterval
with $|J_+|\geq 2 \max(\bn)$.
Let $\kappa \in \Lambda$ be the unique weight such that
$\kappa_J =\kappa_{J;\bn,\bc} \in \Lambda_{J;\bn,\bc}$.
Then $L(\kappa)$ is the unique indecomposable object in its block; in particular it is prinjective in $\mathcal C$.
\end{Lemma}

\begin{proof}
In view of (\ref{wtdec}) it suffices to show 
for $\lambda \in \Lambda$
that
$|\lambda| = |\kappa| \Rightarrow\lambda = \kappa$.
We proceed by induction on $l$, 
the case $l=0$ being trivial.

For the induction step assume first that $c_i = 0$ for some $i$.
Amongst all the $i$ with $c_i=0$ choose one for which $n_i$ is minimal.
Thus $n_i \leq n_j$ for all $j$ with $c_j = 0$, and $n_i \leq 2 \max(\bn) - n_j
\leq |J_+|-n_j$ for all $j$ with $c_j = 1$.
Letting $s := \min(J_+)-1$, it follows that the columns
$s+1,\dots,s+n_i$ of the $01$-matrix $\kappa$
have all entries equal to $1$. Since $|\lambda| = |\kappa|$
the number of entries $1$ in each column of $\lambda$ is the same as in $\kappa$.
Hence all $n_i$ of the entries $1$ in the $i$th row of $\lambda$
appear in columns $s+1,\dots,s+n_i$. 
Thus the $i$th row of $\lambda$ is 
the same as the $i$th row of $\kappa$.
Then we remove this row and proceed by induction.

This just leaves us with the case that $c_i = 1$ for all $i$.
Choose $i$ so that $n_i$ 
is minimal and let $s := \max(J_+)-n_i$.
Then columns $s+1,\dots,s+n_i$ of $\kappa$, hence also of $\lambda$,
have all entries equal to $0$. So the $i$th row of $\lambda$ is 
the same as in $\kappa$, and then we can induct as before.
\end{proof}

\subsection{Decomposition numbers and blocks}
For any $I$ and $(\bn,\bc)$, 
let $\mathcal C$ be an $\mathfrak{sl}_I$-tensor product categorification of
type $(\bn,\bc)$. 
In the finite case, 
Theorem~\ref{lwmain} shows that $\mathcal C$ is equivalent to some
blocks of parabolic category $\mathcal O$ for the general linear Lie algebra,
hence we can exploit the extensive literature about parabolic category
$\mathcal O$ to deduce
results about $\mathcal C$.
In the infinite case, many 
questions about $\mathcal C$ can be answered by picking
a sufficiently large finite subinterval $J \subset I$, then passing
to the subquotient $\mathcal C_J$ and invoking Theorem~\ref{subq}.

For example, the following theorem shows that decomposition
numbers in $\mathcal C$
can be computed by reducing to the Kazhdan-Lusztig conjecture (which
describes the decomposition numbers in parabolic
category $\mathcal O$); see $\S$\ref{cc} for more about the explicit combinatorics here.
We will appeal to this observation in the next section to prove the super Kazhdan-Lusztig
conjecture for $\mathfrak{gl}_{n|m}(\C)$.

\begin{Theorem}\label{obviously}
Given $\lambda,\mu \in \Lambda$, choose a finite subinterval
$J \subseteq I$ such that $\lambda$ and $\mu$ both belong to $\Lambda_J$.
Then the composition multiplicity $[\Delta(\lambda):L(\mu)]$ in $\mathcal C$ coincides with the
multiplicity $[\Delta(\lambda_J):L(\mu_J)]$
computed in $\mathcal C_J$.
Hence, recalling that $\mathcal C_J$ is equivalent to a sum of integral blocks of
parabolic category $\mathcal O$ for the general linear Lie algebra,
these multiplicities can be computed via the Kazhdan-Lusztig conjecture.
\end{Theorem}

\begin{proof}
This is immediate from the exactness of the quotient functor $\pi_J$.
\end{proof}

As another illustration of the truncation technique, we classify the blocks of $\mathcal C$.

\begin{Theorem}\label{blockclass}
For $\lambda,\mu \in \Lambda$,
the irreducible objects $L(\lambda)$ and $L(\mu)$ lie in the same
block of $\mathcal C$ if and only if $|\lambda|=|\mu|$ in the weight
lattice $P_I$.
\end{Theorem}

\begin{proof}
When $I$ is finite, the theorem has been proved already in \cite{Bcent} (working in the
 parabolic category $\mathcal O$ setting).
Now suppose that $I$ is infinite.
We observed already from (\ref{wtdec}) that $\lambda$ and $\mu$ lie in the same block only if $|\lambda|=|\mu|$.
Conversely suppose that $|\lambda|=|\mu|$. Pick a finite interval $J \subset I$
such that $\lambda,\mu \in \Lambda_J$.
Then $|\lambda_J| = |\mu_J|$ in $P_J$, so by the finite result
there exists a sequence of weights
$\lambda=\lambda_0,\dots,\lambda_n =\mu$
in $\Lambda_J$ such that 
one of
$[\Delta(\lambda_i):L(\lambda_{i-1})]$
or
$[\Delta(\lambda_{i-1}):L(\lambda_i)]$ is non-zero for each $i=1,\dots,n$.
Since these composition multiplicities are the same in $\mathcal C$ or
$\mathcal C_J$ this does the job.
\end{proof}

\subsection{Classification of prinjectives}\label{more}
Let $\mathcal C$ be as in the previous subsection.
To avoid trivialities assume moreover that 
$\Lambda := \Lambda_{I;\bn,\bc}$ is non-empty.
The goal in this subsection is to classify the indecomposable
prinjective objects in $\mathcal C$. For finite $I$, this is a
generalization of an old result of Irving \cite{I} 
which was established already in the context of
parabolic category $\mathcal O$ in \cite[Theorem 5.1]{MS} or
\cite[Theorem 4.8]{SWD}. The formulation for infinite $I$ given here
is new; we prove it by using truncation to reduce to the finite case.

We start by noting that there is a crystal graph structure on $\Lambda$.
This is a certain $I$-colored 
directed graph with vertex set $\Lambda$, such that 
there is
at most one edge of each color entering 
and one edge of each color
leaving any given vertex.
To determine 
the edges of color $i$ incident with vertex $\lambda$
one proceeds as follows.
First label rows of the matrix $\lambda$ by the sign
$-$ if the $i$th and $(i+1)$th entries of the row are $1\,0$,
or by $+$ 
if these entries are $0\,1$;
leave all the other rows unlabeled.
Then reduce the labels by 
repeatedly erasing $+-$-pairs of labels whenever the $+$-row is above 
the $-$-row
and all the rows in between are unlabeled.
If at the end of this process 
a $-$-row (resp.\ a $+$-row) remains, 
then there is an edge $\lambda \stackrel{i}{\rightarrow} \mu$ 
(resp.\ $\lambda\stackrel{i}{\leftarrow} \mu$)
in the crystal graph,
where
$\mu$ is obtained from $\lambda$ by 
switching the $i$th and $(i+1)$th entries
of the lowest $-$-row (resp.\ the highest $+$-row).

\begin{Lemma}\label{Crystal}
For $\lambda \in \Lambda$ and $i \in I$, we have that 
$F_i L(\lambda) = 0$ (resp.\
$E_i L(\lambda) = 0$) 
unless there is an edge $\lambda\stackrel{i}{\rightarrow} \mu$
(resp.\ $\lambda\stackrel{i}{\leftarrow}\mu$)
in the crystal graph, 
in which case $F_i L(\lambda)$ (resp.\ $E_i L(\lambda)$)
is indecomposable with irreducible head and socle isomorphic to $L(\mu)$.
\end{Lemma}

\begin{proof}
This is already known for finite $I$;
see \cite[Theorem 7.2]{LW} for the most recent but also most conceptual proof
(actually the arguments of \cite{L} are sufficient here since $\mathcal C$ is a highest weight category).
In the infinite case we 
pick $J \subset I$ containing $i$ such that $\lambda$ and all the weights $\mu$ indexing the composition factors of $E_i L(\lambda)$ and $F_i L(\lambda)$ lie in $\Lambda_J$, and then pass to
the subquotient $\mathcal C_J$.
\end{proof}

To formulate the main result, we need slightly
different notation according to whether $I$ is finite or infinite:
\begin{itemize}
\item[-]
In the finite case, we let
$\Lambda^\circ$ be the 
vertex set of the connected component of the crystal graph containing
$\kappa := \kappa_{I;\bn,\bc}$. Let $T \in \mathcal C$ denote the object from (\ref{prinjec}).
\item[-]
In the infinite case, 
we fix finite subintervals $I_1 \subset I_2 \subset \cdots \subset I$
such that $I = \bigcup_{r \geq 1} I_r$, $|I_1|+1 \geq 2 \max(\bn)$, and $|I_{r+1}| =|I_r|+1$ for each $r$.
Let $\Lambda_r := \Lambda_{I_r} \subset \Lambda$ and $\mathcal C_r :=
\mathcal C_{I_r}$.
Let $\kappa^r$ be the element of $\Lambda_r$ corresponding to
$\kappa_{I_r;\bn,\bc}$
and $\Lambda^\circ_r$
be the vertex set of the 
connected subgraph of the crystal graph generated by $\kappa^r$ and the edges of colors from $I_r$.
Then set
\begin{equation}\label{circledweights}
\Lambda^\circ := \bigcup_{r \geq 1} \Lambda^\circ_r.
\end{equation}
For each $r$ we have that $\kappa^r \in \Lambda^\circ_{r+1}$;
see (\ref{bandon}) below
for an explicit directed path $\kappa^{r+1} \rightarrow \kappa^r$.
So we have that $\Lambda_1^\circ \subset \Lambda_2^\circ \subset\cdots$, and $\Lambda^\circ$ can be described equivalently 
as the vertex set of the 
unique connected component of the entire crystal graph that contains
every $\kappa^r$.
Finally, let $T^r$ be the object 
\begin{align}\label{tr}
T^r = \bigoplus_{d \geq 0} T^r_d &:= \bigoplus_{d \geq 0} F_{I_r}^d 
L(\kappa^r) \in \mathcal C.
\end{align}
\end{itemize}

\begin{Theorem}\label{prinj}
Let notation be as above.
For $\lambda \in \Lambda$, the following are equivalent:
\begin{itemize}
\item[(i)]
$\lambda \in \Lambda^\circ$;
\item[(ii)] $L(\lambda) \hookrightarrow T$
  (resp. $L(\lambda)\hookrightarrow T^r$ for some $r \geq 1$)
if $I$ is finite (resp. infinite);
\item[(ii$'$)] $T \twoheadrightarrow L(\lambda)$ 
  (resp. $T^r\twoheadrightarrow L(\lambda)$ for some $r \geq 1$)
if $I$ is finite (resp. infinite);
\item[(iii)] $P(\lambda)\cong I(\lambda)$.
\item[(iv)] $I(\lambda)$ is projective.
\item[(iv$'$)] $P(\lambda)$ is injective;
\item[(v)] $L(\lambda)$ is a constituent of the socle of a standard
  object;
\item[(v$'$)] 
$L(\lambda)$ is a constituent of the cosocle of a costandard object;
\end{itemize}
\end{Theorem}

\begin{proof}
We just prove the equivalence of (i), (ii), (iii), (iv) and (v); the
equivalence of (i), (ii$'$), (iii), (iv$'$) and (v$'$) is similar.

(i)$\Rightarrow$(ii).
In the finite case, the connected component of the crystal graph with vertex set $\Lambda^\circ$
is a copy of Kashiwara's crystal graph associated to the irreducible $\mathfrak{sl}_{I}$-module
of highest weight $|\kappa|$, and $\kappa$ is its highest vertex.
Hence there is a directed path 
$\kappa \stackrel{i_1}{\rightarrow} \cdots\stackrel{i_d}{\rightarrow}\lambda$ in the crystal graph for some $d \geq 0$ and $i_1,\dots,i_d \in I$.
Similarly in the infinite case there is a path
$\kappa^r \stackrel{i_1}{\rightarrow} \cdots\stackrel{i_d}{\rightarrow}\lambda$ in the crystal graph for some $r\geq 1$, $d \geq 0$ and $i_1,\dots,i_d \in I_r$.
It remains to apply Lemma~\ref{Crystal} to deduce that $L(\lambda)$
appears in 
the socle of $F_{i_d} \cdots F_{i_1} L(\kappa)$
(resp.\ $F_{i_d} \cdots F_{i_1} L(\kappa^r)$), which is a summand of $T$
(resp.\ $T^r$). 

(ii)$\Rightarrow$(iii).
In the finite case, we exploit the duality $\circledast$ from
Corollary~\ref{duality} as follows. The module $T$ is
self-dual and prinjective. By (ii), it has $I(\lambda)$ as
a summand. Hence $I(\lambda)^\circledast$ is an indecomposable
injective object too. Since $\circledast$ fixes irreducibles (up to
isomorphism), $I(\lambda)^\circledast$ has all the same composition multiplicities as
$I(\lambda)$.
Since the classes $[I(\lambda)]$ of indecomposable injectives are linearly independent in the
Grothendieck group, we must therefore have
that $I(\lambda)^\circledast \cong I(\lambda)$. But obviously
$I(\lambda)^\circledast \cong P(\lambda)$, so this shows that
$I(\lambda) \cong P(\lambda)$.

In the infinite case, we have not yet established Corollary~\ref{duality}, so
must argue indirectly. 
Since $T^r$ is prinjective, (ii) implies that
$I(\lambda) \cong P(\mu)$ for some (possibly different) $\mu \in
\Lambda_r$.
But 
the previous
paragraph shows that 
$I(\lambda) \cong P(\lambda)$ in $\mathcal C_r$, hence $P(\lambda) \cong
P(\mu)$ in $\mathcal C_r$. Since both $\lambda$ and $\mu$ lie in
$\Lambda_r$, this implies in fact that $\lambda = \mu$.

(iii)$\Rightarrow$(iv). Clear.

(iv)$\Rightarrow$(v). This follows because projectives have
$\Delta$-flags.

(v)$\Rightarrow$(i).
We first prove this in the finite case.
Suppose that $L(\lambda) \hookrightarrow \Delta(\nu)$
for some $\nu\in\Lambda$.
We show that $\lambda \in \Lambda^\circ$ by induction on the height of $|\kappa|-|\nu| \in Q_{I}^+$. The base case is clear as then $\lambda = \nu = \kappa$.
For the induction step we apply \cite[Proposition 5.2]{LW} to deduce that
there exists $i \in I$ such that $E_i L(\lambda) \neq 0$ in $\mathcal C$.
The socle of $E_i L(\lambda)$ is isomorphic to $L(\mu)$
where $\mu \stackrel{i}{\rightarrow} \lambda$ in the crystal graph. Now it suffices to show that $\mu \in \Lambda^\circ$.
This follows by the induction hypothesis, noting that $L(\mu) \hookrightarrow
E_i L(\lambda) \hookrightarrow E_i \Delta(\nu)$
and $E_i \Delta(\nu)$ has a $\Delta$-flag.

For the infinite case suppose that $L(\lambda) \hookrightarrow \Delta(\mu)$.
Then we pick $r \geq 1$ so that all composition factors of $\Delta(\mu)$
lie in $\Lambda_r$.
Passing to the subquotient $\mathcal C_r$,
we still have that $L(\lambda)$ is a constituent of the socle of $\Delta(\mu)$
in $\mathcal C_r$. Hence $\lamdba \in \Lambda^\circ_r \subset \Lambda^\circ$ thanks to the previous paragraph. 
\end{proof}

\section{The general linear Lie superalgebra}\label{gl}

Throughout the section we fix a type $(\bn,\bc)$ of level $l$.
The goal is to give an explicit construction of an $\mathfrak{sl}_\Z$-tensor
product categorification $\mathcal C$ of type $(\bn,\bc)$, thus establishing the
existence in Theorem~\ref{lwmain}.
We will obtain $\mathcal C$ from the parabolic analog of the
BGG category $\mathcal O$ for the general linear Lie superalgebra.
It is a (mostly known) generalization of \cite[$\S$7.4]{CR} where
category $\mathcal O$ for the general linear Lie algebra was
considered. 
The super Kazhdan-Lusztig conjecture (in its most general parabolic
form)
then follows using also Theorem~\ref{obviously}.

\subsection{Super category $\mathcal O$}
For $i=1,\dots,l$ let $U_i$ be a vector superspace of
dimension $n_i$ concentrated in degree $\bar c_i \in \Z / 2$ (over the ground field
$\K$ as always).
Then set 
\begin{equation}\label{directsum}
U := U_1\oplus\cdots\oplus U_l,
\end{equation}
so that $U$ is a vector
superspace of even dimension $n := \sum_{c_i= 0} n_i$ and odd
dimension $m :=
\sum_{c_i =1} n_i$.
Let $\mathfrak{g}$ denote the Lie superalgebra $\mathfrak{gl}(U) \cong
\mathfrak{gl}_{n|m}(\K)$ consisting of
all linear endomorphisms of $U$ under the supercommutator $[-,-]$.
We choose a homogeneous basis $u_1,\dots,u_{m+n}$ for $U$ by concatenating bases for 
$U_1,\dots,U_l$ in order
 and let $\{e_{i,j}\:|\:1 \leq i,j \leq m+n\}$
be the resulting basis of matrix units for $\mathfrak{g}$. We then have that
$$
[e_{i,j}, e_{k,l}] = \delta_{k,j} e_{i,l}  - (-1)^{(p_i+ 
p_j)(p_k+p_l)}
\delta_{i,l} e_{k,j},
$$
where $p_i \in \Z / 2$ is the parity of the $i$th basis vector $u_i$.

Let $\mathfrak{t}$ (resp.\ $\mathfrak{b}$) 
be the Cartan (resp.\ Borel) subalgebra of $\mathfrak{g}$
consisting of diagonal (resp.\ upper triangular) matrices
relative to the ordered basis just chosen.
Let $\delta_1,\dots,\delta_{m+n}$ be the basis for $\mathfrak{t}^*$ dual
to the basis $e_{1,1},\dots,e_{m+n,m+n}$, and define a
non-degenerate symmetric bilinear form $(-,-)$ on $\mathfrak{t}^*$
by setting $(\delta_i,\delta_j) := (-1)^{p_i} \delta_{i,j}$.
The {\em root system} of $\mathfrak{g}$ is
$$
R := \{\delta_i - \delta_j\:|\:1 \leq i,j \leq m+n, i \neq j\},
$$
which decomposes into even and odd roots $R = R_\0 \sqcup R_\1$
so that $\delta_i - \delta_j$ is of parity $p_i + p_j$.
Let $R^+ = R^+_{\0} \sqcup R^+_{\1}$ denote the 
positive roots associated to the Borel subalgebra $\mathfrak{b}$, i.e.
$\delta_i - \delta_j$ is positive if and only if $i < j$.
The {\em dominance order} $\unrhd$ on $\mathfrak{t}^*$
is defined so that $\lambda \unrhd \mu$
if $\lambda-\mu \in \N R^+$.

An important role is played by the weight $\bar\rho \in
\mathfrak{t}^*$ that is one half of the sum of the positive even roots
minus one half of the sum of the positive odd roots.
Translating by a multiple of supertrace, we obtain 
the weight
\begin{equation}\label{rhodef}
\rho := \bar\rho + 
\frac{1}{2}(m-n+1) \sum_{i=1}^{m+n}(-1)^{p_i} \delta_i,
\end{equation}
which is uniquely determined by the properties
\begin{equation}\label{rhoequiv}
(\rho,\delta_1) = 
\left\{
\begin{array}{ll}
0&\text{if $p_1 = \0$,}\\
1&\text{if $p_1 = \1$;}
\end{array}\right.
\qquad
(\rho,\delta_i-\delta_{i+1}) = \left\{
\begin{array}{cl}(-1)^{p_i} &\text{if $p_i = p_{i+1}$},\\
0&\text{if $p_i \neq p_{i+1}$}.\end{array}\right.
\end{equation}
Note that $\rho$ is an 
{\em integral weight}, that is, it belongs to 
$\mathfrak{t}_\Z^* := \Z \delta_1\oplus\cdots\oplus\Z \delta_{m+n} \subset \mathfrak{t}^*$.

We are ready to introduce the category $\mathcal O$ associated to 
$\mathfrak{t} \subset \mathfrak{b} \subset \mathfrak{g}$,
restricting our attention from the outset just to blocks corresponding
integral weights.

\begin{Definition}\rm
For $\lambda \in \mathfrak{t}_\Z^*$ let $p_\lambda := \sum_{p_i = \1} 
(\lambda,\delta_i) \in \Z / 2$. 
Then define $\mathcal O$ to be the category of all finitely generated
$\mathfrak{g}$-supermodules $M = M_\0 \oplus M_\1$
which are locally finite dimensional over $\mathfrak{b}$ and satisfy
\begin{equation}\label{parity}
M = \bigoplus_{\lambda \in \mathfrak{t}_\Z^*} M_{\lambda,p_\lambda},
\end{equation}
where for $\lambda \in \mathfrak{t}^*$ and $p \in \Z / 2$ we write $M_{\lambda,p}$ for the
$\lambda$-weight space of $M_p$ with respect to $\mathfrak{t}$ defined
in the standard way. Morphisms in $\mathcal O$ mean arbitrary
$\mathfrak{g}$-supermodule homomorphisms. The parity
assumption in (\ref{parity}) ensures that all morphisms are automatically even,
hence $\mathcal O$ is an Abelian category.
\end{Definition}
\begin{Remark}\rm
  One can also define an equivalent category where we remove the
  parity assumption above, but allow inhomogeneous morphisms.  See the
  discussions in \cite[\S 4-e]{B}, \cite[\S 2.5]{ChL} and \cite[Remarks 2.1-2.3]{BcatO}.
\end{Remark}

Note that both the natural $\mathfrak{g}$-supermodule
$U$ and its dual belong to $\mathcal O$, and $\mathcal O$ is closed
under tensoring with these objects.
As usual, to classify the irreducible objects in $\mathcal O$, one
starts from 
the {\em Verma supermodules} $\{M(\lambda)\:|\:\lambda \in \mathfrak{t}_\Z^*\}$ defined from
$$
M(\lambda) := U(\mathfrak{g}) \otimes_{U(\mathfrak{b})}
\K_{\lambda,p_\lambda}
$$
where $\K_{\lambda,p_\lambda}$ is the one-dimensional $\mathfrak{b}$-supermodule of
weight $\lambda$ with $\Z / 2$-grading concentrated in degree
$p_\lambda$.
The weight $\lambda$ is the highest weight of $M(\lambda)$ with
respect to the dominance ordering, and the corresponding weight space
is one dimensional.
Therefore, by the 
usual arguments of highest weight theory,
$M(\lambda)$ has a unique irreducible quotient
$L(\lambda)$, and the supermodules $\{L(\lambda)\:|\:\lambda \in \mathfrak{t}_\Z^*\}$
give a complete set of pairwise non-isomorphic irreducible objects
of $\mathcal O$.

For later use, we note that the {\em Casimir element}
\begin{equation}\label{casimir}
c := \sum_{1 \leq i,j \leq m+n} (-1)^{p_j} e_{i,j} e_{j,i} \in
Z(U(\mathfrak{g}))
\end{equation}
acts on $M(\lambda)$ by the scalar
\begin{equation}\label{casscalar}
c_\lambda := (\lambda+2\bar\rho,\lambda).
\end{equation}
We need one other non-trivial result about linkage. 
For $\alpha \in R^+_{\0}$ 
and $\lambda \in \mathfrak{t}_\Z^*$
let $s_\alpha \cdot \lambda := \lambda - (\lambda+\rho,\alpha^\vee)
\alpha$,
where $\alpha^\vee := 2 \alpha / (\alpha,\alpha)$. 
Let
\begin{align*}
A(\lambda) &:= \{\alpha \in R^+_\0\:|\:(\lambda+\rho,\alpha^\vee) >
0\},\\
B(\lambda) &:= \{\alpha \in R^+_\1\:|\:(\lambda+\rho,\alpha) = 0\}.
\end{align*}
Then introduce a relation $\uparrow$ on $\mathfrak{t}_\Z^*$
by declaring that $\mu \uparrow \lambda$ if we either have that
$\mu = s_\alpha \cdot \lambda$ for some $\alpha \in A(\lambda)$
or we have that $\mu = \lambda - \beta$ for some $\beta \in B(\lambda)$.

\begin{Lemma}\label{musson}
Suppose $\lambda,\mu \in \mathfrak{t}_\Z^*$ satisfy $[M(\lambda):L(\mu)] \neq 0$. Then there exists $r \geq 0$ and weights
$\nu_0,\dots,\nu_r \in \mathfrak{t}_\Z^*$ such that
$\mu = \nu_0 \uparrow \nu_1 \uparrow\cdots\uparrow \mu_r = \lambda$.
\end{Lemma}

\begin{proof}
This is a consequence of the superalgebra analog of the Jantzen sum formula from \cite[Theorem 10.3.1]{M}; see also \cite{Gorelik}.
In more detail, the Jantzen filtration on $M(\lambda)$ is a certain exhaustive
descending filtration $M(\lambda) = M(\lambda)_0 \supset M(\lambda)_1 \supseteq M(\lambda)_2 \supseteq \cdots$
such that $M(\lambda)_0 / M(\lambda)_1 \cong L(\lambda)$, and the sum formula shows that
$$
\sum_{k \geq 1} \operatorname{ch} M(\lambda)_k
= \sum_{\alpha \in A(\lambda)}
\operatorname{ch} M(s_\alpha\cdot\lambda)
+ \sum_{\beta \in B(\lambda)} 
\sum_{k \geq 1}
(-1)^{k-1}\operatorname{ch} M(\lambda - k \beta).
$$
To deduce the lemma from this, suppose that $[M(\lambda):L(\mu)] \neq 0$.
Then $\mu \unlhd \lambda$, so that $\lambda-\mu$ is a sum of $N$ simple roots
$\delta_i-\delta_{i+1}$ for some $N \geq 0$. We proceed by induction on $N$, the case $N=0$ being vacuous. If $N > 0$ then $L(\mu)$ is a composition factor of 
$M(\lambda)_1$ and the sum formula implies that $L(\mu)$ is a composition factor either of $M(s_\alpha\cdot\lambda)$ for some $\alpha \in A(\lambda)$
or that $L(\mu)$ is a composition factor of $M(\lambda-k\beta)$ for some 
odd $k \geq 1$ and $\beta \in B(\lambda)$. It remains to apply the induction hypothesis and the definition of $\uparrow$.
\end{proof}

There is another partial order $\leq$ on
$\mathfrak{t}_\Z^*$, which we call the {\em Bruhat order}, 
defined from
\begin{equation}\label{love2}
\lambda \leq \mu \quad\Leftrightarrow\quad
\sum_{\substack{1 \leq i \leq j \\ (\lambda+\rho,\delta_i) \leq h}}
(-1)^{p_i}
\geq \sum_{
\substack{1 \leq i \leq j \\ (\mu+\rho,\delta_i) \leq h}}
(-1)^{p_i}
\end{equation}
for all $h \in \Z$ and $1\leq j\leq m+n$, 
with equality whenever $j=m+n$.

\begin{Lemma}\label{comb}
For $\lambda,\mu \in \mathfrak{t}_\Z^*$ we have that
$\lambda \uparrow \mu \Rightarrow  \lambda \leq \mu \Rightarrow \lambda \unlhd \mu$.
\end{Lemma}

\begin{proof}
The implication $\lambda \uparrow \mu\Rightarrow \lambda \leq \mu$ is
an easy exercise.
To prove that $\lambda \leq \mu$ implies $\lambda \unlhd \mu$, observe
as in Lemma~\ref{dominance} that
$\lambda \unlhd \mu$ if and only if
$$
\sum_{i=1}^j (-1)^{p_i} 
(\lambda+\rho,\delta_i) \leq\sum_{i=1}^j (-1)^{p_i}
(\mu+\rho,\delta_i)
$$
for each $j=1,\dots,m+n$, with equality when $j=m+n$.
Moreover, for any integer $k$ such that $k \geq
(\lambda+\rho,\delta_i)$
for all $i$,
 we have that
\begin{align*}
\sum_{i=1}^j (-1)^{p_i} (\lambda+\rho,\delta_i)
&=\sum_{h \leq k} \left(
\sum_{\substack{1 \leq i \leq j\\ (\lambda+\rho,\delta_i) \leq h}} (-1)^{p_i} h
-
\sum_{\substack{1 \leq i \leq j\\ (\lambda+\rho,\delta_i) < h}} (-1)^{p_i} h\right)\\
&=
\sum_{h \leq k} \left(
h\sum_{\substack{1 \leq i \leq j\\ (\lambda+\rho,\delta_i) \leq h}} (-1)^{p_i}
-
(h+1)\sum_{\substack{1 \leq i \leq j\\ (\lambda+\rho,\delta_i) \leq
    h}} (-1)^{p_i}\right)
\\&= - \sum_{h \leq k}
\sum_{\substack{1 \leq i \leq j\\ (\lambda+\rho,\delta_i) \leq h}} (-1)^{p_i}.
\end{align*}
Now use the definition (\ref{love2}).
\end{proof}

\begin{Remark}\rm
The Bruhat order $\leq$ coincides with the transitive closure of the relation
$\uparrow$
if all the $0$'s appear either before or after all of the $1$'s in the
sequence $\bc$; see \cite[Lemma 2.5]{B}. 
However in general $\leq$ is a proper
refinement of the transitive closure of $\uparrow$; see \cite[Remark 2.3]{CLW} for an example.
\end{Remark}

\begin{Theorem}
The category $\mathcal O$ is a highest weight category with weight
poset
$(\mathfrak{t}_\Z^*, \leq)$.
Its 
standard objects are the Verma supermodules $\{M(\lambda)\:|\:\lambda\in\mathfrak{t}_\Z^*\}$.
\end{Theorem}

\begin{proof}
By Lemmas~\ref{musson} and \ref{comb}, 
all composition factors of $M(\lambda)$ are of the form $L(\mu)$
for $\mu \leq \lambda$.
The theorem follows from this and the 
BGG reciprocity established
in a general graded Lie superalgebra setting in 
\cite[(6.6)]{Btilt}; one also needs to repeat the argument of \cite[Lemma
7.3]{Btilt} to check that projectives are of finite length.
\end{proof}

\subsection{Super parabolic category $\mathcal O$}\label{secsuper}
Let $\mathfrak{h} \cong \mathfrak{gl}_{n_1}(\K)\oplus\cdots\oplus
\mathfrak{gl}_{n_l}(\K)$
be the subalgebra of $\mathfrak{g}$ that is the stabilizer of the
direct sum decomposition (\ref{directsum})
and set $\mathfrak{p} := \mathfrak{h}+\mathfrak{b}$.

\begin{Definition}\rm
Let $\mathcal C$  be the full subcategory of
$\mathcal O$ consisting of all objects that are locally finite
dimensional over $\mathfrak{p}$.
We refer to this as {\em super parabolic category $\mathcal O$} of type $(\bn,\bc)$.
\end{Definition}

The isomorphism classes of irreducible objects in $\mathcal C$ are
represented by the supermodules 
$L(\lambda)$ for $\lambda$ in the set 
\begin{equation}\label{lambdan}
\Lambda := 
\left\{\lambda \in \mathfrak{t}^*_\Z\:\bigg|\:
\begin{array}{l}
(-1)^{c_k}(\lambda+\rho,\delta_r-\delta_{r+1}) > 0 \text{ for $1 \leq k\leq l$}\\
\text{and $n_1+\cdots+n_{k-1}<r<n_1+\cdots+n_k$}
\end{array}\right\}.
\end{equation}
To see this, note first that the condition in the definition of $\Lambda$
is the usual dominance condition for
finite dimensionality of irreducible highest weight modules for
$\mathfrak{h}$,
so clearly for $L(\lambda)$ 
to belong to $\mathcal C$ it is necessary that
$\lambda \in \Lambda$. For the sufficiency, 
for each $\lambda
\in \Lambda$, let $V(\lambda)$ be a finite dimensional irreducible
$\mathfrak{h}$-supermodule of highest weight $\lambda$
with $\Z/2$-grading concentrated in degree $p_\lambda$. The
corresponding {\em parabolic Verma supermodule}
\begin{equation*}
\Delta(\lambda) := U(\mathfrak{g}) \otimes_{U(\mathfrak{p})}
V(\lambda)
\end{equation*}
belongs to $\mathcal C$ and is a quotient of $M(\lambda)$.
Hence $\Delta(\lambda)$ has irreducible head $L(\lambda)$,
implying that
$L(\lambda)$ belongs to $\mathcal C$.
We see moreover from this argument (and our knowledge of composition
factors of $M(\lambda)$) that all other composition factors of $\Delta(\lambda)$
are of the form $L(\mu)$ for $\mu < \lambda$ in the Bruhat order.
The following theorem now follows on making another application 
of \cite[(6.6)]{Btilt}.

\begin{Theorem}
The category $\mathcal C$ is a highest weight category with weight
poset
$(\Lambda, \leq)$.
Its 
standard objects are the parabolic 
Verma supermodules $\{\Delta(\lambda)\:|\:\lambda\in\Lambda\}$.
\end{Theorem}

It is time to switch to more combinatorial
notation by {\em identifying}
the set $\Lambda \subset \mathfrak{t}_\Z^*$ 
from (\ref{lambdan})
with the set $\Lambda_{\Z;\bn,\bc}$ of $01$-matrices
from (\ref{lal}).
We do this 
so that
$\lambda \in \Lambda$ 
corresponds to the $01$-matrix
$(\lambda_{ij})_{1 \leq i \leq l, j \in \Z}$
defined from
\begin{equation}\label{wtdict}
\lambda_{ij} = \left\{
\begin{array}{ll}
1-c_i&
\text{if $j = (\lambda+\rho,\delta_r)$
for $n_1+\cdots+n_{i-1} < r \leq n_1+\cdots+n_i$,}\\
c_i&\text{otherwise.}
\end{array}
\right.
\end{equation}
When compared with Lemma~\ref{pset}(ii),
the following lemma checks under this identification that the Bruhat
order $\leq$ on $\Lambda$ agrees with the partial order $\leq$ from
the axiom (TP1) in the previous section.

\begin{Lemma}\label{orderdesc}
For $\lambda,\mu \in \Lambda$ we have that
$\lambda \leq \mu$ in the Bruhat order if and only if 
$$
\sum_{\substack{1 \leq r \leq n_1+\cdots+n_k\\ (\lambda+\rho,\delta_r) \leq h}} (-1)^{p_r}
\geq 
\sum_{\substack{1 \leq r \leq n_1+\cdots+n_k\\ (\mu+\rho,\delta_r) \leq h}} (-1)^{p_r}
$$
for all $h \in \Z$ and $1\leq k\leq l$, with equality whenever $k=l$.
\end{Lemma}

\begin{proof}
The forward implication is 
clear from the definition (\ref{love2}).
For the converse suppose that 
$\lambda \not\leq \mu$ in the Bruhat order.
Let $1 \leq s \leq m+n$ be minimal so that
$$
\sum_{\substack{1 \leq r \leq s\\ (\lambda+\rho,\delta_r) \leq h}}
 (-1)^{p_r} 
<
\sum_{\substack{1 \leq r \leq s\\ (\mu+\rho,\delta_r) \leq h}}(-1)^{p_r}
$$
for some $h \in \Z$.
To complete the proof, we show for $k$ defined from $n_1+\cdots+n_{k-1} < s \leq n_1+\cdots+n_k$
that
$$
\sum_{\substack{1 \leq r \leq n_1+\cdots+n_k\\ (\lambda+\rho,\delta_r) \leq h}}
(-1)^{p_r}
< 
\sum_{\substack{1 \leq r \leq n_1+\cdots+n_k\\  (\mu+\rho,\delta_r) \leq h}}
(-1)^{p_r}.
$$
Suppose first that $p_s = \0$. Then by the minimality of $s$ 
we must have that $(\lambda+\rho,\delta_s) > h$ and $(\mu+\rho,\delta_s) = h$, hence $(\mu+\rho,\delta_r) < h$
for all $s < r \leq n_1+\cdots+n_k$.
We deduce that
\begin{align*}
\sum_{\substack{1 \leq r \leq n_1+\cdots+n_k\\ (\lambda+\rho,\delta_r) \leq h}}
(-1)^{p_r}
&\leq
n_1+\cdots+n_k-s+\sum_{\substack{1 \leq r \leq s\\ (\lambda+\rho,\delta_r) \leq h}}
(-1)^{p_r}\\
&< 
n_1+\cdots+n_k-s+\sum_{\substack{1 \leq r \leq s\\ (\mu+\rho,\delta_r) \leq h}}
(-1)^{p_r}=\!\!\!
\sum_{\substack{1 \leq r \leq n_1+\cdots+n_k\\ (\mu+\rho,\delta_r) \leq h}}
(-1)^{p_r}.
\end{align*}
Instead suppose that $p_s = \1$. Then $(\lambda+\rho,\delta_s) = h$ and $(\mu+\rho,\delta_s) > h$, 
hence $(\lambda+\rho,\delta_r) > h$ and $(\mu+\rho,\delta_r) > h$
for all $s < r \leq n_1+\cdots+n_k$.
We deduce that
$$
\sum_{\substack{1 \leq r \leq n_1+\cdots+n_k\\ (\lambda+\rho,\delta_r) \leq h}}
(-1)^{p_r}
=\sum_{\substack{1 \leq r \leq s\\ (\lambda+\rho,\delta_r) \leq h}}
(-1)^{p_r}
< 
\sum_{\substack{1 \leq r \leq s\\ (\mu+\rho,\delta_r) \leq h}}
(-1)^{p_r}=
\sum_{\substack{1 \leq r \leq n_1+\cdots+n_k\\ (\mu+\rho,\delta_r) \leq h}}
(-1)^{p_r}.
$$
We are done.
\end{proof}

Finally we introduce a categorical $\mathfrak{sl}_\Z$-action
on $\mathcal C$.
Let $F$ 
(resp.\ $E$) be the endofunctor of $\mathcal C$
defined by tensoring with $U$ (resp. its dual $U^*$).
Let $x \in \End(F)$ be the endomorphism 
defined on a supermodule $M$ by letting $x: M \otimes U \rightarrow M
\otimes U$ 
be the endomorphism defined by multiplication by 
$$
\Omega := \sum_{1 \leq r,s \leq m+n} (-1)^{p_s} e_{r,s} \otimes
e_{s,r}.
$$
Let $s \in \End(F^2)$ be the endomorphism arising from the flip
$U \otimes U\rightarrow U \otimes U, u_r \otimes u_s
\mapsto (-1)^{p_r p_s} u_s \otimes u_r$.

\begin{Theorem}\label{catact}
The preceding definitions make $\mathcal C$
into an $\mathfrak{sl}_\Z$-tensor product categorification of type $(\bn,\bc)$.
\end{Theorem}

\begin{proof}
The verification of the axioms (SL1)--(SL3) 
is standard; cf. \cite[$\S$7.4]{CR} and also \cite[Proposition 5.1]{CW}
where the degenerate affine Hecke algebra relations are checked in the super case.
Also we have seen already that $\mathcal C$ is a highest weight category 
with weight poset $(\Lambda, \leq)$ 
as required by axiom (TP1).

Now we take $\lambda \in \Lambda$ and consider the supermodule
$F \Delta(\lambda)$. By the tensor identity we have that
$$
F \Delta(\lambda)
 = (U(\mathfrak{g}) \otimes_{U(\mathfrak{p})} V(\lambda)) \otimes U
\cong U(\mathfrak{g}) \otimes_{U(\mathfrak{p})} (V(\lambda) \otimes
U).
$$
Since $U$ has an obvious $\mathfrak{p}$-filtration with sections $U_1,
\dots, U_l$,
we deduce that $F \Delta(\lambda)$ has a filtration with sections
$U(\mathfrak{g}) \otimes_{U(\mathfrak{p})} (V(\lambda) \otimes U_i)$
for $i=1,\dots,l$.
It is well known how to decompose the
tensor product of a finite dimensional irreducible
$\mathfrak{gl}_{n_i}(\K)$-module
with the
natural module, so that
$$
V(\lambda) \otimes U_i \cong \bigoplus_{\substack{n_1+\cdots+n_{i-1} < r
    \leq n_1+\cdots+n_i
\\
\lambda+\delta_r\in \Lambda}}
V(\lambda+\delta_r).
$$
This discussion proves that the $\mathfrak{g}$-supermodule $F \Delta(\lambda)$ has a multiplicity-free $\Delta$-flag with
sections
$$
\{\Delta(\lambda+\delta_r)
\:\big|\:
\text{for all }1 \leq r \leq m+n\text{ with }\lambda+\delta_{r} \in \Lambda
\}
$$
Converting to the $01$-matrix notation using (\ref{wtdict})
this shows equivalently that $F \Delta(\lambda)$ has a $\Delta$-flag with sections
$$
\left\{\Delta(t_{ij}(\lambda))\:\big|\:\text{for all }
1 \leq i \leq l\text{ and }j \in \Z\text{ such that }
\lambda_{ij} = 1\text{ and }
\lambda_{i(j+1)} = 0\right\},
$$
where $t_{ij}(\lambda)$ denotes
the $01$-matrix obtained by applying the transposition $t_j$ to the $i$th row of
$\lambda$.
To aid in translating between this description and the previous one, we note
that
$t_{ij}(\lambda) = \lambda+\delta_r$,
where $r$ is uniquely determined by
\begin{equation}\label{det}
n_1+\cdots+n_{i-1} < r \leq n_1+\cdots + n_i,
\quad
j = (\lambda+\rho,\delta_r) - (1-(-1)^{c_i})/2.
\end{equation}

We claim further that the endomorphism $\Omega$ preserves the
filtration just constructed and induces the endomorphism of the
section
$\Delta(t_{ij}(\lambda))$ that is multiplication by the scalar $j \in \Z$.
To see this note that $$
\Omega = 
(\Delta(c) - c \otimes 1 - 1 \otimes c) /
2
$$ 
where $\Delta$ is the comultiplication on $U(\mathfrak{g})$
and $c$ is the Casimir element (\ref{casimir}).
This shows already that $\Omega$ preserves the filtration.
Also 
$c$ acts on $\Delta(\lambda)$ by the scalar $c_\lambda$ from (\ref{casscalar}).
Hence the endomorphism of
$\Delta(t_{ij}(\lambda))$ induced by $\Omega$
is the scalar $(c_{t_{ij}(\lambda)} - c_\lambda - n+m) / 2$.
Letting $r$ be defined according to (\ref{det}),
we now calculate using (\ref{casscalar}) and (\ref{rhodef}):
\begin{align*}
(c_{t_{ij}(\lambda)}-c_\lambda-n+m)/2
&=
(c_{\lambda+\delta_r} - c_\lambda-n+m)/2\\
&=((\lambda+2\bar \rho+\delta_r, \lambda+\delta_r) -
(\lambda+2\bar\rho,\lambda)-n+m)/2\\
&= (\lambda+\bar\rho,\delta_r) + ((-1)^{c_i}-n+m)/2\\
&=
(\lambda+\rho,\delta_r) + ((-1)^{c_i} -1)/2
= j.
\end{align*}
This proves our claim.

Thus we have shown that $F_j \Delta(\lambda)$ has a $\Delta$-flag and that
$$
[F_j \Delta(\lambda)] = \sum_{i} [\Delta(t_{ij}(\lambda))]
$$
summing over all $i=1,\dots,l$ such that $\lambda_{ij} = 1$ and
$\lambda_{i(j+1)} = 0$.
Similarly one checks that $E_j \Delta(\lambda)$ has a $\Delta$-flag and
that
$$
[E_j \Delta(\lambda)] = \sum_{i} [\Delta(t_{ij}(\lambda))]
$$
summing over all $i=1,\dots,l$ such that $\lambda_{ij} = 0$ and
$\lambda_{i(j+1)} = 1$; it helps to know for this that $E_j M$  is
the generalized $(m-n-j)$-eigenspace of $\Omega$ on $M \otimes U^*$.
These formulae are consistent with the formulae for the actions of $f_j$
and $e_j$ on $v_\lambda \in \bigwedge^{\bn,\bc} V_\Z$. So we have
proved
(TP2)--(TP3).
\end{proof}

\subsection{Proof of the super Kazhdan-Lusztig conjecture}
We now have in place all of the theory needed
to prove the
super Kazhdan-Lusztig conjecture. 
Continue with $\mathcal C$ being as in Theorem~\ref{catact}.
Recall for each finite interval $J \subset \Z$ that we have defined a subset
$\Lambda_J \subset \Lambda$ and a poset isomorphism
$\Lambda_J \stackrel{\sim}{\rightarrow} \Lambda_{J;\bn,\bc},\lambda
\mapsto  \lambda_J$; see the first paragraph of $\S$\ref{trun}.
Moreover we have defined a subquotient $\mathcal C_J$ of $\mathcal C$ which is an
$\mathfrak{sl}_J$-tensor product categorification of type $(\bn,\bc)$;
see Theorem~\ref{subq}.
In view of the uniqueness from Theorem~\ref{lwmain} (for finite intervals),
$\mathcal C_J$ is equivalent to 
a block of parabolic category $\mathcal O$ for some general linear Lie algebra.

Now Theorem~\ref{obviously} implies that the multiplicities of Verma
supermodules are computed by certain parabolic Kazhdan-Lusztig polynomials
evaluated at $q=1$.
The super Kazhdan-Lusztig conjecture follows from this assertion;
we will give a more detailed discussion in
$\S$\ref{cc} after we have introduced the techniques to discuss gradings on super
parabolic category $\mathcal O$; see also the recent survey \cite{BcatO} for more about the combinatorics.

Various other known results about super parabolic category $\mathcal
O$ can be deduced from the finite case in a similar fashion.
In particular, Theorem~\ref{blockclass} gives a classification of the
blocks generalizing \cite[Theorem 3.12]{CMW},
 Lemma~\ref{Crystal} gives another proof of the main
result of \cite{Kuj}, and 
Theorem~\ref{prinj} gives a classification of prinjectives in
super parabolic category $\mathcal O$ which appears to be new.

\section{Stable modules}\label{sstable}
In this section the goal is to prove Theorem~\ref{lwmain}
for infinite intervals.
The strategy is almost exactly the same as the strategy for finite intervals 
recalled in $\S$\ref{finite}.
The main issue is to find a suitable substitute for the quotient functor $\V:\mathcal C \rightarrow \rmod{H}$ from Theorem~\ref{dcp}.
This arises from a new category $\rmod{H}$ of ``stable modules,'' which seems quite interesting in its own right. 
Throughout the section $(\bn,\bc)$ will be a fixed type and $I$ will be an
infinite interval.
Like in Theorem~\ref{prinj}, we pick finite subintervals 
$I_1\subset I_2 \subset\cdots$
of $I$ such that
$I = \bigcup_{r \geq 1} I_r$, $|I_1|+1 \geq 2 \max(\bn)$, and $|I_{r+1}| = |I_r|+1$ for each $r$.

\subsection{Tower of Hecke algebras}\label{tower}
Let $\mathcal C$ be some given $\mathfrak{sl}_I$-tensor product categorification of 
type $(\bn,\bc)$.
For each $r$ we denote the subcategories $\mathcal C_{\leq I_r}$
and $\mathcal C_{< I_r}$ from $\S$\ref{trun} by
$\mathcal C_{\leq r}$ and $\mathcal C_{< r}$, respectively.
Then each subquotient 
$\mathcal C_r := \mathcal C_{\leq r} / \mathcal C_{< r}$
gets the induced structure of an $\mathfrak{sl}_{I_r}$-tensor product categorification with weight poset 
$\Lambda_r := \Lambda_{I_r} \subset \Lambda$.
Denote the element of $\Lambda_r$ corresponding to $\kappa_{I_r;\bn,\bc} \in \Lambda_{I_r;\bn,\bc}$ by 
$\kappa^r$.
Thus $v_{\kappa^r}$ is the highest weight vector of 
$\bigwedge^{\bn,\bc} V_{I_r} \subset\bigwedge^{\bn,\bc} V_I$.

Let $T^r$ be the $r$th ``tensor space'' from (\ref{tr}).
It
has a $\Delta$-flag with sections of the form $\Delta(\mu)$
for $\mu \in \Lambda_r$. Hence by (\ref{ext2}) the quotient functor $\pi_r:\mathcal C_{\leq r} \rightarrow \mathcal C_r$ 
 defines a isomorphism between $\End_{\mathcal C}(T^r)$ 
and $\End_{\mathcal C_r}(T^r)$.
The following theorem follows immediately from this observation, Theorem~\ref{grth} and
the first part of Theorem~\ref{dcp}.

\begin{Theorem}
The action of 
$QH_d$ on $T^r_d$
induces a canonical isomorphism between 
the algebra 
\begin{equation}\label{hr}
H^{r} = \bigoplus_{d \geq 0} H^r_d
:= 
\bigoplus_{d \geq 0} QH^{|\kappa^r|}_{I_r,d}
\end{equation}
and the endomorphism algebra
$\End_{\mathcal C}(T^r)$.
\end{Theorem}

For each $r \geq 1$ and $i \in I_r$, we introduce the idempotent
\begin{equation}\label{inductionidempts}
1_{d;i}^r :=\sum_{\substack{\bi \in I_r^{d+1}\\i_{d+1}=i}} 1_\bi \in H^r_{d+1}.
\end{equation}
There is a natural algebra homomorphism
\begin{equation}\label{theinclusion}
\iota^r_{d;i}:H^r_d \rightarrow 1_{d;i}^r H^r_{d+1} 1_{d;i}^r,
\qquad
1_\bj \mapsto 1_{\bj i},
\quad
\xi_j \mapsto \xi_j 1^r_{d;i},
\quad
\tau_j \mapsto \tau_j 1^r_{d;i}.
\end{equation}
In diagrammatic terms
this map is given by tensoring with 
a single string $i\rightarrow i$.
With the following 
lemma we construct some much less familiar maps
between Hecke algebras for different $r$.
These are given
by tensoring with 
some explicitly known diagram
$\bi\rightarrow \bi$ on the {\em left}.

\begin{Lemma}\label{tl}
Fix $r \geq 1$ and $d \geq 0$. 
If $\max(I_r) = \max(I_{r+1})$
we define
\begin{align*}
p_j &:= \big|\{i\:|\:c_i = 0, n_i \geq j\}\big|,
& a &:= \max\{n_i\:|\:c_i = 0\},\\
 s &:= \min(I_{r+1})-1,
&\eps &:= 1,\\
\bi&:=
(s+a)^{p_a} \cdots (s+2)^{p_2} (s+1)^{p_1}
&d_r &:= \sum_{c_i = 0} n_i.
\end{align*} 
Instead if $\min(I_r) = \min(I_{r+1})$, we let 
\begin{align*}
p_j &:= \big|\{i\:|\:c_i = 1, n_i \geq j\}\big|,
& a &:= \max\{n_i\:|\:c_i = 1\},\\
 s &:= \max(I_{r+1})+1,
&\eps &:= -1,\\
\bi&:=
(s-a)^{p_a} \cdots (s-2)^{p_2} (s-1)^{p_1}
&d_r &:= \sum_{c_i = 1} n_i.
\end{align*} 
(In either case $\bi$ is a word in $I_{r+1}^{d_r}$.)
\begin{itemize}
\item[(i)]
There exists an
explicit idempotent $e^r_d \in H^{r+1}_{d_r+d}$
such that the map
$$
\phi^r_d:H^r_d \stackrel{\sim}{\rightarrow} e^r_d H^{r+1}_{d_r+d} e^r_d,\qquad 
1_\bj \mapsto  1_{\bi\bj} e^r_d ,
\quad
\xi_j\mapsto \xi_{d_r+j}e^r_d,\quad \tau_j\mapsto \tau_{d_r+j}e^r_d 
$$ 
is a well-defined algebra isomorphism.
Moreover $e^r_d 1_\bk = 1_{\bk} e^r_d = 0$ unless
$\bk = \bi \bj$ for some $\bj \in I_r^d$.
\item[(ii)]
There exists
an isomorphism
$\theta^r_d:T^r_d \stackrel{\sim}{\rightarrow} e^{r}_d T^{r+1}_{d_r+d}$ in $\mathcal C$,
such that the following diagram commutes
for all $h \in H^r_d$:
\begin{equation*}
\begin{CD}
T^{r}_{d} &@>h>>& T^r_d\\
@V\theta^r_dVV&&@VV\theta^r_dV\\
e^r_d T^{r+1}_{d_r+d} &@>>\phi^r_d(h)>& e^r_d T^{r+1}_{d_r+d}.
\end{CD}
\end{equation*}
\item[(iii)]
For $i \in I_r$ the map
$\iota^{r+1}_{d_r+d;i}$ from (\ref{theinclusion})
sends $e^r_d$ to $1^{r+1}_{d_r+d;i} e^r_{d+1}$.
\item[(iv)]
For $i \in I_r$ we have that
$\phi_{d+1}^r(1_{d;i}^r) = 
1_{d_r+d;i}^{r+1}e^r_{d+1}$.
\end{itemize}
\end{Lemma}

\begin{proof}
A straightforward check shows in $\bigwedge^{\bn,\bc} V_I$ that
\begin{equation}\label{bandon}
f_{s+\eps 1}^{(p_1)}
f_{s+\eps 2}^{(p_2)}
\cdots f_{s+\eps a}^{(p_a)} v_{\kappa^{r+1}} = v_{\kappa^r},
\end{equation}
where we write $f_i^{(m)}$ for the divided power $f_i^m / m!$.
Now recall for each $i\in I_{r+1}$ and $m \geq 1$ 
that there is an explicit idempotent $b_m \in 1_{i^m} QH_{I_{r+1},m} 1_{i^m}$
such that the summand $F_i^{(m)} := b_m F_i^m$ of the functor $F_i^m$
induces 
$f_i^{(m)}$ at the level of the Grothendieck group; see \cite[Lemma 4.1]{R}.
Hence for each $d \geq 0$ there is an explicit idempotent
$e^{r}_d \in QH^{|\kappa^{r+1}|}_{I_{r+1},d_r+d}=H^{r+1}_{d_r+d}$ such that
\begin{equation}
e^r_d F_{I_{r+1}}^{d_r+d} L(\kappa^{r+1})
=
F_{I_r}^d F_{s+\eps 1}^{(p_1)}\cdots F_{s+\eps a}^{(p_a)} L(\kappa^{r+1}).
\end{equation}
Since $L(\kappa^{r+1})$ and $L(\kappa^r)$ are standard objects, the identity (\ref{bandon}) implies that there exists a (unique up to scalars) isomorphism
$L(\kappa^r)
\stackrel{\sim}{\rightarrow} F_{s+\eps 1}^{(p_1)}\cdots F_{s+\eps a}^{(p_a)} L(\kappa^{r+1})$.
Applying the functor $F_{I_r}^d$ to this
and using the above equality we obtain
the isomorphism
$\theta^r_d: T^r_d \stackrel{\sim}{\rightarrow} T^{r+1}_{d_r+d}e^r_d$.
This definition ensures that $\theta^r_d$ intertwines the actions of $1_\bj, \xi_j, \tau_j \in H^r_d$
and $1_{\bi\bj}e^r_d, \xi_{d_r+j} e^r_d,  \tau_{d_r+j}e^r_d \in e^r_d H^{r+1}_{d_r+d}e^r_d$.
Also $e^r_d H^{r+1}_{d_r+d} e^r_d = \End_{\mathcal C}(e^r_d T_{d_r+d}^{r+1})
\cong
\End_{\mathcal C}(T^r_d) = H^r_d$.
Parts (i) and (ii) follow.

For the final parts of the lemma, the homomorphism
$\iota^{r+1}_{d_r+d;i}$ can be viewed 
simply as an application of the functor $F_i$, so it
maps $e^r_d$ to 
$F_i e^r_d$,
while $e^r_{d+1} = F_{I_r} e^r_d$ by its analogous definition.
This implies (iii).
Part (iv) is clear.
\end{proof}

In the notation of the lemma, we
set $e^r := \sum_{d \geq 0} e^r_d$, $\phi^r := \sum_{d \geq 0} \phi^r_d$
and $\theta^r := \sum_{d \geq 0} \theta^r_d$.
Thus $e^r\in H^{r+1}$ is an idempotent, $\phi^r :H^r \stackrel{\sim}{\rightarrow} e^r H^{r+1}e^r$ is an algebra isomorphism, and
$\theta^r :T^r\stackrel{\sim}{\rightarrow} e^r T^{r+1} $ is an isomorphism 
in $\mathcal C$.
In particular this gives us a tower
$$
H^1 \hookrightarrow H^2 \hookrightarrow H^3 \hookrightarrow \cdots
$$
of cyclotomic quiver Hecke algebras
which will play a key role.
If $M$ is a right (resp.\ left) $H^{r+1}$-module we will implicitly view 
$M e^r$ (resp.\ $e^r M$) as a right (resp.\ left) $H^r$-module
via $\phi^r$. 

More generally for $r \leq s$ we set
\begin{align}
\phi^{r,s} &:= \phi^{s-1} \circ \cdots\circ \phi^r,
&e^{r,s} &:= \phi^{r,s}(1_{H_r}),
&\theta^{r,s} &:= \theta^{s-1} \circ \cdots \circ \theta^{r}
\end{align}
Thus $e^{r,s}$ is an idempotent in $H^s$,
$\phi^{r,s}:H^r \stackrel{\sim}{\rightarrow} e^{r,s} H^s e^{r,s}$ is an algebra isomorphism,
and $\theta^{r,s}:T^r \rightarrow e^{r,s} T^s$ is an
isomorphism in $\mathcal C$.
If $M$ is a right (resp.\ left) $H^s$-module we will implicitly view $M e^{r,s}$ 
(resp.\ $e^{r,s} M$) as an $H^r$-module
via $\phi^{r,s}$.

\subsection{Stable modules and the double centralizer property}
With the following definition we introduce an auxiliary category which is actually a little too big; we will cut it down to size in Definition~\ref{rmod2} below.

\begin{Definition}\rm\label{rmod}
Let $\rmod{H}^\infty$ 
be the category whose objects are diagrams
$$
\begin{CD}
  M= (M^1 &@>\iota^1>>&M^2&@>\iota^2>> M^3&@>\iota^3>> \cdots)\\
\end{CD}
$$
such that $M^r \in \rmod{H^r}$ for each $r$ and $\iota^r$ gives an $H^r$-module
isomorphism $M^r \stackrel{\sim}{\rightarrow} M^{r+1} e^r$
for all $r \geq 1$. A morphism $f:M
\rightarrow N$ in $\rmod{H}^\infty$ means a sequence $(f^r)_{r \geq 1}$ of
$H^r$-module homomorphisms $f^r:M^r \rightarrow N^r$ such that the
following diagram commutes:
\begin{equation}\label{morphism}
  \begin{CD}
    M^1 &@>>>&M^2&@>>> M^3&@>>> \cdots\\
    @Vf^1VV&&@Vf^2VV&@Vf^3VV\\
    N^1 &@>>>&N^2&@>>> N^3&@>>> \cdots.
  \end{CD}
\end{equation}
We write simply $\Hom_H(M,N)$ for the morphisms in $\rmod{H}^\infty$.
It is easy to see that $\rmod{H}^\infty$ is an Abelian category.
We stress however that there is no algebra $H$ in sight here.
\end{Definition}

For any $s \geq 1$, an object $M \in\rmod{H}^\infty$ is determined uniquely up to isomorphism just by 
knowledge of its tail $(M^s \rightarrow M^{s+1}\rightarrow\cdots)$.
We make this statement precise by introducing the category
$\rmod{H^{\geq s}}$ consisting of diagrams $M = (M^s \stackrel{\iota^s}{\rightarrow} M^{s+1} \stackrel{\iota^{s+1}}{\rightarrow}\cdots)$ with $M^r \in \rmod{H^r}$
and $\iota^r:M^r \stackrel{\sim}{\rightarrow} M^{r+1} e^r$ for each $r \geq s$
(just like in 
Definition~\ref{rmod} but starting at $r$ not $1$).
Then there is a forgetful functor
\begin{equation}
\tail_{s}: \rmod{H}^\infty \rightarrow \rmod{H}^{\geq s}
\end{equation}
sending $(M^1\rightarrow M^2 \rightarrow \cdots)$ to its
tail $(M^s \rightarrow M^{s+1} \rightarrow\cdots)$.
Let
\begin{equation}
\head_{s}:\rmod{H}^{\geq s} \rightarrow \rmod{H}^\infty
\end{equation}
be the functor sending $(M^s \stackrel{\iota^s}{\rightarrow} M^{s+1} \stackrel{\iota^{s+1}}{\rightarrow} \cdots)$
to $(M^1 \stackrel{\iota^1}{\rightarrow} \cdots\stackrel{\iota^{s-1}}{\rightarrow} M^s \stackrel{\iota^{s+1}}{\rightarrow}\cdots)$,
where for $r < s$ we let $M^r := M^s e^{r,s} \in \rmod{H^r}$;
the maps $\iota^r:M^r \rightarrow M^{r+1}$ for $r < s$ are simply the inclusions.
Obviously $\tail_{s} \circ \head_{s} = \operatorname{id}$.
It is an easy exercise to show moreover that
$\head_{s} \circ \tail_{s} \cong \operatorname{id}$.
Thus the functors $\tail_{s}$ and $\head_{s}$ are quasi-inverse
equivalences of categories.

For each $r \geq 1$, we define two more functors
\begin{equation}
\top_r:\rmod{H}^{\geq r} \rightarrow \rmod{H^r},\qquad
\top_r^!:\rmod{H^r} \rightarrow \rmod{H}^{\geq r}.
\end{equation}
The first of these is
defined simply by projecting $M$ onto its top term $M^r$.
The second is defined on $M \in \rmod{H^r}$
by $\top_r^! M := (M^r \rightarrow M^{r+1}\rightarrow \cdots) \in \rmod{H}^{\geq r}$,
where
$M^s := M \otimes_{H_r} e^{r,s} H^s \in \rmod{H^s}$. 
The linear maps $\iota^s:M^s \rightarrow M^{s+1}$
are the maps
$M^s \rightarrow M^{s+1}, v \otimes h \mapsto v \otimes \phi^s(h)$.
Finally we set
\begin{align}
\pr_r &:= \top_r \circ \tail_{r}: \rmod{H}^\infty \rightarrow \rmod{H^r},\label{mmm}\\
\pr_r^! &:= \head_{r} \circ \top_r^!: \rmod{H^r} \rightarrow \rmod{H}^\infty.\label{ungradedpr}
\end{align}
The first of these $\pr_r$ is of course just the obvious projection onto the $r$th component.
It is also clear that $\pr_r \circ \pr_r^! \cong \operatorname{id}$.

\begin{Lemma}\label{leftadjoint} 
The functor $\pr_r^!$ is left adjoint to $\pr_r$.
\end{Lemma}

\begin{proof}
It suffices to check that $\top_r^!$ is left adjoint to $\top_r$.
The counit of the adjunction on object
$M = (M^r \stackrel{\iota^r}{\rightarrow} M^{r+1} \stackrel{\iota^{r+1}}{\rightarrow}\cdots)$ 
is
$(\eta^{r,s})_{s \geq r}: \top_r^! (\top_r M) \rightarrow M$
defined from
\begin{equation}\label{topmaps}
\eta^{r,s}:M^r \otimes_{H^r} e^{r,s} H^s \mapsto M^s,
\qquad v \otimes h \mapsto \iota^{r,s}(v) h,
\end{equation}
setting $\iota^{r,s} := \iota^{s-1}\circ\cdots\circ \iota^r$.
We leave the routine checks to the reader.
\end{proof}

\begin{Definition}\label{rmod2}\rm
We say that $M \in \rmod{H}^\infty$ is {\em $r$-stable} if it is in the essential image of the functor $\pr^!_r$; equivalently the maps
(\ref{topmaps}) are isomorphisms for all $s > r$.
Then $M$ is {\em stable} if it is $r$-stable for some $r \geq 1$.
Finally let $\rmod{H}$ be the full subcategory of $\rmod{H}^\infty$ consisting of all
stable objects.
(We will see soon that $\rmod{H}$ is itself an Abelian category but this is not obvious as it is not a Serre subcategory of $\rmod{H}^\infty$.)
\end{Definition}

\begin{Lemma}\label{okay}
If $M \in \rmod{H}$ is
$r$-stable and $N \in \rmod{H}^\infty$ is any object
then $\pr_r:\Hom_H(M,N) \rightarrow \Hom_{H^r}(M^r, N^r)$
is an isomorphism.
\end{Lemma}

\begin{proof}
We have that $M \cong \pr_r^! (M^r)$. Now use 
Lemma~\ref{leftadjoint}.
\end{proof}

Now we bring the category $\mathcal C$ back into the picture.
Let
\begin{equation}\label{ur}
\V^r:= \Hom_{\mathcal C}(T^r,-):\mathcal C \rightarrow \rmod{H^{r}}.
\end{equation}
Then let
$\V:\mathcal C \rightarrow \rmod{H}^\infty$
be the functor sending $M \in \mathcal C$
to
\begin{equation}\label{atlast}
\begin{CD}
\V M := (\V^1 M &@>\iota^1>>&\V^2 M&@>\iota^2>> \V^3 M&@>\iota^3>> \cdots),
\end{CD}
\end{equation}
where the isomorphism $\iota^r:
\Hom_{\mathcal C}(T^r, M) \stackrel{\sim}{\rightarrow}\Hom_{\mathcal C}(e^r T^{r+1},M)$
here is the map $\phi \mapsto \phi \circ (\theta^r)^{-1}$.
On a morphism $f:M \rightarrow N$ 
we let $\V f := (\V^r f)_{r \geq 1}$.
By Lemma~\ref{we} and (\ref{tr}), each
$T^r$ is a prinjective object of $\mathcal C$, hence
the functors $\V^r$ and $\V$ are both exact.
Also $\V^r = \pr_r \circ \V$.
Recalling the definition (\ref{circledweights}), we let
\begin{align}\label{young3}
Y(\lambda) = (Y^1(\lambda)\rightarrow Y^2(\lambda)\rightarrow\cdots) &:=
\V P(\lambda),\\
D(\mu) = (D^1(\mu)\rightarrow D^2(\mu)\rightarrow\cdots) &:=
\V L(\mu),
\end{align}
for $\lambda \in \Lambda$ and $\mu \in \Lambda^\circ$.
By Theorem~\ref{prinj} for the category $\mathcal C_r$
(and the usual theory of quotient functors)
the objects
$\{D^r(\mu)\:|\:\mu \in \Lambda^\circ_r\}$ give a complete set of pairwise non-isomorphic irreducible $H^r$-modules.
Moreover 
$Y^r(\mu)$ is the projective cover of $D^r(\mu)$
in $\rmod{H^r}$ for each $\mu \in \Lambda^\circ_r$.

\begin{Theorem}\label{abelian1}
The essential image of the functor $\V$ is the subcategory
$\rmod{H}$ of $\rmod{H}^\infty$. 
\end{Theorem}
\begin{proof}
  
We first show that $\V M$ is stable
for any $M \in \mathcal C$.
Given $M$ pick $r \geq 1$ so that
\begin{equation}\label{possible}
\bigcup_{\substack{\mu \in \Lambda \\ [M:L(\mu)] \neq 0}}
\left\{\nu \in \Lambda\:\big|\:[P(\mu):L(\nu)] \neq 0\right\}
\subseteq \Lambda_r.
\end{equation}
This is possible as the set on the left hand side here is finite.
We claim for this $r$ that $\V M$ is $r$-stable.
This amounts to showing for each $s > r$ that the $H^s$-module homomorphism
$$
\Hom_{\mathcal C}(T^r, M) \otimes_{H^r} e^{r,s} H^s
\rightarrow \Hom_{\mathcal C}(T^s, M),
\qquad
f \otimes h \mapsto f \circ (\theta^{r,s})^{-1} \circ h
$$
is an isomorphism.
For surjectivity,
we split $T^s$ and $T^r$ 
into indecomposables
$T^s = 
P_1 \oplus\cdots\oplus P_n$
and $T^r = Q_1 \oplus \cdots \oplus Q_m$, so that
$$
\Hom_{\mathcal C}(T^s,M) = \bigoplus_{i=1}^n \Hom_{\mathcal C}(P_i, M),
\quad
\Hom_{\mathcal C}(T^r,M) = \bigoplus_{j=1}^m \Hom_{\mathcal C}(Q_j, M).
$$
By the assumption on $r$ all composition factors of $M$ are
of the form $L(\mu)$ for $\mu \in \Lambda_r$.
Hence $\Hom_{\mathcal C}(P_i, M) = 0$ unless $P_i \cong P(\mu)$
for some $\mu \in \Lambda^\circ_r$. In that case there is a summand
$Q_j$ with $Q_j \cong P_i$.
Thus
any $f\in \Hom_{\mathcal C}(P_i, M) \subseteq \Hom_{\mathcal C}(T^s,M)$
factors as $g \circ k$
for some $g \in \Hom_{\mathcal C}(Q_j, M) \subseteq \Hom_{\mathcal C}(T^r,M)$
and $k\in \Hom_{\mathcal C}(P_i, Q_j) \subseteq \Hom_{\mathcal C}(T^s, T^r)$.
Since $\Hom_{\mathcal C}(T^s, T^r) \cong e^{r,s} H^s$, 
we deduce that $f= g \circ (\theta^{r,s})^{-1} \circ h$
for some $h \in e^{r,s} H^s$, and surjectivity follows.

For injectivity, let $K$ be the kernel of the map, so that there is a short exact sequence
$$
0 \rightarrow K \rightarrow \Hom_{\mathcal C}(T^r, M) \otimes_{H^r} e^{r,s} H^s
\rightarrow \Hom_{\mathcal C}(T^s, M) \rightarrow 0.
$$
On truncating with the idempotent $e^{r,s}$, the second map becomes an isomorphism, 
hence $K e^{r,s} = 0$.
Thus all composition factors of $K$ are of the form
$\{D^s(\nu)\:|\:\nu \in \Lambda^\circ_s \setminus \Lambda^\circ_r\}$.
On the other hand all composition factors of $\Hom_{\mathcal C}(T^s, M)$ are of the form $D^s(\mu)$
for $\mu \in \Lambda_s^\circ$ such that
such that $[M:L(\mu)] \neq 0$.
Now the choice of $r$ ensures for such $\nu$ and $\mu$ that
$L(\nu)$ is not a composition factor of $P(\mu)$.
Hence $D^s(\nu)$ is not a composition factor of $Y^s(\mu)$, which is the projective cover of $D^s(\mu)$.
It follows that $\Ext^1_{H^s}(D^s(\mu), D^s(\nu)) = 0$, and
we have proved that the above short exact sequence splits.
But then we get that
\begin{align*}
\Hom_{H^s}(K, D^s(\nu)) &\twoheadleftarrow
\Hom_{H^s}(\Hom_{\mathcal C}(T^r, M) \otimes_{H^r} e^{r,s} H^s, D^s(\nu))\\
&\cong 
\Hom_{H^r}(\Hom_{\mathcal C}(T^r, M), D^s(\nu) e^{r,s}) = 0
\end{align*}
for any $\nu \in \Lambda^\circ_s \setminus \Lambda^\circ_r$.
This implies that $K = 0$.

So $\V$ restricts to a well-defined functor
$\V:\mathcal C \rightarrow \rmod{H}$.
We next construct a left adjoint
$\V^!:\rmod{H} \rightarrow \mathcal C$
such that $\V \circ \V^! \cong \operatorname{id}$,
implying in particular that this new $\V$ is essentially surjective.
On an object $M \in \rmod{H}$ we let $r \geq 1$ be minimal such that $M$ is $r$-stable then set
$\V^! M := M^r \otimes_{H^r} T^r \in \mathcal C_{\leq r},$
where $M^r \otimes_{H^r} -$ here is the tensor product
functor from the category of $H^r$-module objects in $\mathcal C$ to $\mathcal C$.
Note for $s\geq r$ that there is an isomorphism
$$
f^{r,s}: M^r \otimes_{H^r} T^r \stackrel{\sim}{\rightarrow} M^s \otimes_{H^s} T^s
$$
defined by the composition of the following canonical isomorphisms
$$
M^r \otimes_{H^r} T^r \stackrel{\operatorname{id} \otimes \theta^{r,s}}{\longrightarrow}
M^r \otimes_{H^r} e^{r,s} T^s
\cong
M^r \otimes_{H^r} e^{r,s} H^s \otimes_{H^s} T^s
\stackrel{\eta^{r,s} \otimes \operatorname{id}}{\rightarrow} M^s \otimes_{H^s} T^s,
$$
for $\eta^{r,s}$ coming from (\ref{topmaps}).
Then on a morphism $f = (f^r)_{r \geq 1}:M \rightarrow N$
we define $\V^! f$ by picking $t$ such that both $M$ and $N$ are $t$-stable then setting $\V^! f := (f^{s,t})^{-1} \circ 
f^t \circ f^{r,t}: M^r \otimes_{H^r} T^r \rightarrow M^s \otimes_{H^s} T^s$.
One needs to observe that this is well defined independent of the choice of $t$ as $\theta^{s,t} \circ \theta^{r,s} = \theta^{r,t}$ and $\eta^{s,t} \circ \eta^{r,s} = \eta^{r,t}$.
Then it follows easily that this is a functor.
To see that $\operatorname{id} \cong \V \circ \V^!$,
take
the natural isomorphism defined on $M$
by the canonical isomorphisms
\begin{multline*}
M \stackrel{\sim}{\rightarrow} \operatorname{pr}_r^!(\operatorname{pr}_r M)
=
\operatorname{pr}_r^!(M^r)\\
\cong
\pr_r^!(M^r \otimes_{H^r} \Hom_{\mathcal C}(T^r, T^r))
\stackrel{\sim}{\rightarrow} 
\pr_r^!(\Hom_{\mathcal C}(T^r, M^r \otimes_{H^r} T^r))\\
=
\pr_r^! (\V^r (\V^! M))=
\pr_r^! (\pr_r(\V (\V^! M)))
\stackrel{\sim}{\rightarrow} \V(\V^! M)
\end{multline*}
where $r$ is minimal so that $M$ is $r$-stable. Here the morphisms on the first and last lines come from the counit of adjunction 
$\operatorname{pr}_r^! \circ \operatorname{pr}_r \cong \operatorname{id}$
and the middle one is the obvious isomorphism.
In particular this isomorphism gives the unit 
$\epsilon:\operatorname{id}\rightarrow \V \circ \V^!$
of the claimed adjunction.
We proceed to write down the counit
$\eta:\V^! \circ \V \rightarrow \operatorname{id}$.
To define this on $M \in \mathcal C$, let $r$ be minimal such that $\V M$ is $r$-stable.
Then we take $\eta_M$ to be the obvious evaluation
\begin{equation}\label{onto}
\V^!(\V M)
= \Hom_{H^r}(T^r, M) \otimes_{H^r} T^r \stackrel{\operatorname{ev}_r}{\rightarrow}
M.
\end{equation}
This can also be obtained 
as the composition
$$
\V^!(\V M)
= \Hom_{H^r}(T^r, M) \otimes_{H^r} T^r \stackrel{f^{r,s}}{\rightarrow}
\Hom_{H^s}(T^s, M) \otimes_{H^s} T^s 
\stackrel{\operatorname{ev}_s}{\rightarrow}
M
$$
for any $s \geq r$; in particular this makes it clear that it is surjective.
We leave the remaining checks to the reader.
\end{proof}

\begin{Remark}\rm\label{howbig}
For $M \in \mathcal C$, one can consider the integer
$$
r_M := \min\{r \geq 1\:|\:\text{$\V M$ is $r$-stable}\}.
$$
Of course this depends implicitly on the choices of the intervals
$I_1 \subset I_2 \subset\cdots$.
We observed during the proof of Theorem~\ref{abelian1}
that the map (\ref{onto}) with $r=r_M$ is surjective. 
Hence all constituents of the head of $M$ are of the form
$L(\lambda)$ for $\lambda \in \Lambda_r$
and $M \in \mathcal C_{\leq r}$;
these properties give a lower bound for $r_M$.
For a (surely much too big) 
upper bound one can take the smallest $r$ satisfying (\ref{possible}).
\end{Remark}

\begin{Theorem}\label{abelian2}
The category $\rmod{H}$ is Schurian
with a complete set
of pairwise non-isomorphic irreducibles
given by the objects $\{D(\mu)\:|\:\mu \in \Lambda^\circ\}$.
Moreover $Y(\mu)$ is the projective cover of $D(\mu)$ in $\rmod{H}$
for each $\mu \in \Lambda^\circ$.
Finally
\begin{equation}
\V :\mathcal C \rightarrow \rmod{H}
\end{equation}
satisfies the universal property of
the quotient of $\mathcal C$ by the Serre subcategory generated by
the irreducible objects $\{L(\lambda)\:|\:\lambda\in\Lambda\setminus\Lambda^\circ\}$.
\end{Theorem}

\begin{proof}
First we check that the functor $\V:\mathcal C \rightarrow \rmod{H}$
has the universal property of quotients. We need to show for any Abelian category
$\mathcal C'$ and any exact functor $\mathbb{F}:\mathcal C \rightarrow \mathcal C'$
such that $\mathbb{F} L(\lambda) = 0$ for all $\lambda \in \Lambda \setminus \Lambda^\circ$
that there exists a unique (up to isomorphism) functor $\bar{\mathbb{F}}:\rmod{H} \rightarrow \mathcal C'$
such that $\bar{\mathbb{F}} \circ \V \cong \mathbb{F}$.
Composing on the right with $\V^!$
we see at once that the only choice (up to isomorphism) is to take
$\bar{\mathbb{F}} := \mathbb{F} \circ \V^!$.
The counit of adjunction $\eta$ gives
a natural transformation
$\mathbb{F} \eta: \mathbb{F} \circ \V^! \circ \V
= \bar{\mathbb{F}} \circ \V \Rightarrow \mathbb{F}$.
To see this is an isomorphism, take $M \in \mathcal C$ and let $r$ be minimal such that $\V M$ is $r$-stable. Then $(\mathbb{F} \eta)_M : \bar{\mathbb{F}}(\V M)
\rightarrow \mathbb{F}M$ is the morphism
obtained by applying $\mathbb{F}$ to the second arrow in the following short exact sequence:
$$
0 \rightarrow K \rightarrow
\Hom_{\mathcal C}(T^r, M) \otimes_{H^r} T^r \stackrel{\eta_M}{\rightarrow} M
\rightarrow 0.
$$
To see that this map is an isomorphism it suffices by exactness of $\mathbb{F}$ to show that
$\mathbb{F}K = 0$. This follows because $\V K = 0$.

Now we let $\mathcal C^\circ$ be the quotient of $\mathcal C$
by the Serre subcategory generated by the objects $\{L(\lambda)\:|\:\lambda \in \Lambda \setminus \Lambda^\circ\}$ and $\pi:\mathcal C \rightarrow \mathcal C^\circ$
be the quotient functor.
By the general theory of quotients
this is a Schurian category with irreducibles $\{\pi L(\mu)\:|\:\mu \in \Lambda^\circ\}$; the projective cover of $\pi L(\lambda)$ is  $\pi P(\lambda)$.
The universal property established in the previous paragraph gives us
a
functor $\bar \pi:\rmod{H} \rightarrow \mathcal C^\circ$ such that
$\pi \cong \bar \pi \circ \V$.
On the other hand by the universal property of $\mathcal C^\circ$
there is a functor
$\bar{\V}:\mathcal C^\circ \rightarrow \rmod{H}^\infty$
such that $\V \cong \bar{\V} \circ \pi$; we are being careful here since we do not yet know that $\rmod{H}$ is itself Abelian.
The essential image of $\bar{\V}$ 
is $\rmod{H}$, i.e. it 
is actually a functor $\mathcal C^\circ \rightarrow \rmod{H}$.
Then the usual argument with uniqueness shows that $\bar{\V}$ and $\bar\pi$ are quasi-inverse equivalences. The theorem now follows directly.
\end{proof}

At last we can prove the appropriate analog of the double centralizer property
from Theorem~\ref{dcp}.

\begin{Theorem}\label{ffp}
The functor $\V:\mathcal C \rightarrow \rmod{H}$
is fully faithful on projectives. 
Moreover for each $\lambda \in \Lambda$ the object
$Y(\lambda) = 
\V P(\lambda) \in \rmod{H}$
is independent (up to isomorphism)
of the particular choice of $\mathcal C$.
\end{Theorem}

\begin{proof}
Take projectives $P, Q \in \mathcal C$
and choose $r$ so that both $\V P$ and $\V Q$
are $r$-stable.
By Remark~\ref{howbig}
this means that $P$ and $Q$ are projective objects also in $\mathcal C_r$.
Hence by the double centralizer property from Theorem~\ref{dcp} the composition $\V^r$ of the following two 
maps is an isomorphism:
$$
\Hom_{\mathcal C}(P, Q)
\stackrel{\V}{\rightarrow} \Hom_H(\V P, \V Q)
\stackrel{\pr_r}{\rightarrow}
\Hom_{H^r}(\V^r P, \V^r Q).
$$
Also the second map here is an isomorphism because of Lemma~\ref{okay}.
Hence the first map is an isomorphism, proving that $\V$ is fully faithful on projectives.

Now suppose we are given another tensor product categorification $\mathcal C'$.
Define $\V':\mathcal C' \rightarrow \rmod{H}$
exactly as above then set 
$$
Y'(\lambda) = ({Y'}^1(\lambda)\rightarrow {Y'}^2(\lambda) \rightarrow\cdots) := \V' P'(\lambda).
$$
Pick $r$ so that both $Y(\lambda)$ and $Y'(\lambda)$ are $r$-stable.
Then $P(\lambda)$ is projective in $\mathcal C_r$ and $P'(\lambda)$ is
projective in $\mathcal C_r'$.
By the final part of Theorem~\ref{dcp}
we get that $Y^r(\lambda) \cong {Y'}^r(\lambda)$
in $\rmod{H^r}$, hence
$Y(\lambda) \cong \pr_r^! Y^r(\lambda) \cong \pr_r^! {Y'}^r(\lambda)
\cong Y'(\lambda)$.
\end{proof}

\begin{Remark}\label{stronger2}
\rm
The slightly stronger result from Remark~\ref{stronger} also holds in the present situation; the proof is the same.
\end{Remark}

\subsection{Categorical action on stable modules}\label{secca}
Next we are going to introduce a categorical $\mathfrak{sl}_I$-action onto the category $\rmod{H}$ in
such a way that the quotient functor 
$\V:\mathcal C \rightarrow \rmod{H}$ is strongly equivariant.
For each $r \geq 1$ we have the usual induction and restriction functors $F^r, E^r:\rmod{H^r}
\rightarrow \rmod{H^r}$
from (\ref{forthis}).
In terms of cyclotomic quiver Hecke algebras, these are the direct sums 
over all $i \in I_r$ of the
$i$-induction and $i$-restriction functors
\begin{equation}\label{Fir}
F^r_i:\rmod{H^r} \rightarrow \rmod{H^r},
\qquad
E^r_i:\rmod{H^r} \rightarrow \rmod{H^r}
\end{equation}
defined as follows.
Recalling the idempotent (\ref{inductionidempts}),
$F^r_i$ is given by 
tensoring over $H^r_d$ with the bimodule $1_{d;i}^r H^r_{d+1}$,
viewing $1_{d;i}^r H^r_{d+1}$ as an
$(H^r_d, H^r_{d+1})$-bimodule via the homomorphism (\ref{theinclusion}).
Its canonical right adjoint $E^r_i$ is given on a right $H^r_{d+1}$-module
simply
by right multiplication by this idempotent, viewing the result as a right
$H^r_d$-module via $\iota^r_{d;i}$ again.
The endomorphism 
$\xi \in \End(F^r_i)$ is induced by the 
endomorphism of the bimodule
$1_{d;i}^r H^r_{d+1}$ defined by left multiplication by $\xi_{d+1}$.
To define 
$\tau \in \Hom(F_j^r \circ F_i^r, F_i^r \circ F_j^r)$, note obviously that
$$
1_{d;i}^r H^r_{d+1} \otimes_{H^r_{d+1}} 1_{d+1;j}^r H^r_{d+2}
\cong 1_{d;ij} H^r_{d+2}
\quad
\text{where}
\quad
1_{d;ij} := 
\sum_{\substack{\bi \in I_r^{d+2}\\i_{d+1}=i\\i_{d+2}=j}} 1_\bi \in H^r_{d+2}.
$$
Then $\tau$ comes from the bimodule homomorphism
$1_{d;ij}^r H^r_{d+2} \rightarrow 1_{d;ji}^r H^r_{d+2}$
defined by left multiplication by
$\tau_{d+1}$.
The 
canonical adjunction making each $(F_i^r, E_i^r)$ into an adjoint pair comes 
simply from adjunction of tensor and hom.

We are ready to define $F_i:\rmod{H}\rightarrow \rmod{H}$.
Take $M = (M^1 \stackrel{\iota^1}{\rightarrow} M^2 \stackrel{\iota^2}{\rightarrow}\cdots) \in \rmod{H}$.
Suppose to start with that $r$ is chosen so that $i \in I_r$.
By Lemma~\ref{tl}(iv),
the restriction of 
$\phi_{d+1}^r$
 gives 
a right $H_{d+1}^r$-module homomorphism
$$
\psi_{d;i}^r:1_{d;i}^r H_{d+1}^r
\hookrightarrow 1_{d_r+d;i}^{r+1} H_{d_r+d+1}^{r+1} e^r_{d+1}.
$$
Let $\psi_i^r 
:= \bigoplus_{d \geq 0} \psi^r_{d;i}$ so that
the map
$\iota^r \otimes \psi^r: F_i^r M^r \rightarrow (F_i^{r+1} M^{r+1})e^r$
is an $H^r$-module homomorphism.
Now we assume that $r$ is minimal such that $M$ is $r$-stable and $i \in I_r$.
Then define
\begin{equation}\label{stfdef}
F_iM := \head_{r} (F_i^r M^r \stackrel{\iota^r \otimes \psi^{r}}{\longrightarrow}
F_i^{r+1} M^{r+1} \stackrel{\iota^{r+1}\otimes \psi^{r+1}}{\longrightarrow} \cdots).
\end{equation}
For this to even make sense we need to justify that the maps
$\iota^s \otimes \psi^s:F_i^s M^s \rightarrow (F_i^{s+1} M^{s+1}) e^s$
are isomorphisms for all $s \geq r$, so that the object in parentheses really is an object of $\rmod{H^{\geq r}}$.
This follows at once from the next lemma, which shows moreover
that $F_i M$ is $r$-stable.

\begin{Lemma}\label{indcheck}
If $M\in\rmod{H}$ is $r$-stable and $i \in I_r$ then
the maps
$$
\eta^{s}:F_i^s M^s \otimes_{H^s} e^s H^{s+1}
\rightarrow F_i^{s+1} M^{s+1}, 
\qquad
v \otimes h' \otimes h \mapsto \iota^s(v)\otimes \psi^s(h') h
$$
are isomorphisms for all $s \geq r$.
Hence $F_i$ sends $r$-stable objects to $r$-stable objects.
\end{Lemma}

\begin{proof}
Suppose that $M^s \in \rmod{H^s_d}$.
To prove the lemma we need to show that the map
\begin{align*}
M^s \otimes_{H_d^s} 1_{d;i}^s H_{d+1}^s \otimes_{H^s_{d+1}} e^s_{d+1} H^{s+1}_{d_r+d+1}
&\rightarrow
M^{s+1} \otimes_{H_{d_r+d}^{s+1}} 1_{d_r+d;i}^{s+1} H_{d_r+d+1}^{s+1}\\
v \otimes h' \otimes h &\mapsto \iota^s(v) \otimes \psi^{s}(h') h
\end{align*}
is an isomorphism. 
Taking $h' = 1^s_{d;i}$ and contracting the second tensor product
using Lemma~\ref{tl}(iv),
this is equivalent to showing that
the map
\begin{align*}
\iota^s \otimes \operatorname{id}: M^s \otimes_{H_d^s} 1^{s+1}_{d_r+d;i} e^s_{d+1}H^{s+1}_{d_r+d+1}
&\rightarrow
M^{s+1} \otimes_{H_{d_r+d}^{s+1}} 1_{d_r+d;i}^{s+1} H_{d_r+d+1}^{s+1}
\end{align*}
is an isomorphism.
As $M$ is $r$-stable the following map is an isomorphism:
\begin{align*}
M^s \otimes_{H^s_d} e^s_d H^{s+1}_{d_r+d} \otimes_{H_{d_r+d}^{s+1}} 1_{d_r+d;i}^{s+1} H_{d_r+d+1}^{s+1}
&\rightarrow
 M^{s+1} \otimes_{H_{d_r+d}^{s+1}} 1_{d_r+d;i}^{s+1} H_{d_r+d+1}^{s+1}\\
 v \otimes h' \otimes h & \mapsto \iota^s(v) h' \otimes h.
\end{align*}
Taking $h' = e^s_d$ and contracting the second tensor product
using Lemma~\ref{tl}(iii), gives exactly the desired isomorphism.
\end{proof}

So now we have defined the functor $F_i$ on objects.
On a morphism 
$f:M \rightarrow N$
we just pick $r$ so that $i \in I_r$ and $M$ is $r$-stable.
Then $F_i M$ is $r$-stable too, so by Lemma~\ref{okay} 
there is a unique morphism $F_i f: F_i M \rightarrow F_i N$
such that $(F_i f)^r =  F_i^r f^r$.
This is independent of the choice of $r$, which is all that is needed
to check that $F_i:\rmod{H} \rightarrow \rmod{H}$ is a well-defined functor.

Next we define natural transformations 
\[\xi \in \End(F_i),\qquad\quad\tau \in \Hom(F_j \circ F_i,
F_i \circ F_j).\]
Take some $M \in \rmod{H}$ and pick $r$ so that $i \in I_r$ and $M$ is $r$-stable.
Then by Lemma~\ref{okay} again 
there is a unique morphism $\xi_M:F_i M \rightarrow F_i M$
such that $(\xi_M)^r=\xi_{M^r}$;
similarly, there is a unique $\tau_M:F_j F_i M \rightarrow F_i F_j M$
such that $(\tau_M)^r = \tau_{M^r}$.
The naturality of $\xi$ and $\tau$ follows because these definitions of $\xi_M$ and $\tau_M$
are independent of the particular choice of $r$, as may be checked using
Lemma~\ref{tl}.

\begin{Lemma}\label{dan}
There are isomorphisms $\zeta_i:F_i \circ \mathbb{U} 
\stackrel{\sim}{\rightarrow} \mathbb{U} \circ F_i$ for each $i \in I$
such that $\zeta_i \circ \xi \mathbb{U} = \mathbb{U} \xi \circ
\zeta_i$ in $\Hom(F_i \circ \mathbb{U}, \mathbb{U} \circ F_i)$
and $\zeta_i F_j \circ F_i \zeta_j \circ \tau \mathbb{U} = \mathbb{U}
\tau \circ \zeta_j F_i \circ F_j \zeta_i$ in $\Hom(F_j \circ F_i \circ
\mathbb{U}, \mathbb{U} \circ F_i \circ F_j)$.
\end{Lemma}

\begin{proof}
To define $\zeta_i$ on an object $M \in \mathcal C$,
take any $r$ such that $i \in I_r$ and both
$\V M$ and $\V (F_i M)$ are $r$-stable.
By Lemma~\ref{seq},
there exists a canonical isomorphism
$\zeta_i^r:\V^r \circ F_i^r \stackrel{\sim}{\rightarrow}
F_i^r \circ \V^r$.
Then we define $(\zeta_i)_M:\V (F_i M) \rightarrow
F_i (\V M)$ to be the unique morphism
with $r$th component equal to $(\zeta_i^r)_M$.
For the naturality one just needs to observe that this is independent of the choice of $r$.
The $\xi$- and $\tau$-equivariance
properties follow because the $\zeta_i^r$ 
satisfy analogous equivariance properties for sufficiently large $r$.
\end{proof}

We turn our attention to
$E_i:\rmod{H} \rightarrow \rmod{H}$.
For this we start by defining a functor
$E_i:\rmod{H}^\infty \rightarrow \rmod{H}^\infty$.
Let $r$ be minimal such that $i \in I_r$.
Then for $M = (M^1 \rightarrow M^2 \rightarrow \cdots) \in \rmod{H}^\infty$
we let \begin{equation}
E_i M := (0 \rightarrow \cdots \rightarrow 0\rightarrow E_i^r M^r \rightarrow E_i^{r+1} M^{r+1} \rightarrow\cdots)
\end{equation}
where the maps are simply the restrictions.
Lemma~\ref{easych} below verifies that this is indeed an object of $\rmod{H}^\infty$.
On a morphism $f:M \rightarrow N$ we define $(E_i f)^s := E_i^s f^s$
if $i \in I_s$ and $(E_i f)^s := 0$ otherwise.

\begin{Lemma}\label{easych}
Suppose we are given $M \in \rmod{H}^\infty$, $i \in I$ and $r \geq 1$.
If $i \in I_r$ then
the restriction of $\iota^r:M^r \stackrel{\sim}{\rightarrow} M^{r+1} e^r$
is an isomorphism $E_i^r M^r \stackrel{\sim}{\rightarrow} (E_i^{r+1} M^{r+1})e^r$.
If $i \in I_{r+1} \setminus I_r$ then $(E_i^{r+1} M^{r+1}) e^r = 0$.
\end{Lemma}

\begin{proof}
We may assume that $M^r \in \rmod{H^r_{d+1}}$.
Then we just have to recall that
$E_i^r M^r = M^r 1_{d;i}^r$.
So using also Lemma~\ref{tl}(iv) 
we deduce that $\iota^r$ maps it isomorphically to
$M^{r+1} 1^{r+1}_{d_r+d;i}e^r_{d+1}$,
which is all of $(E_i^{r+1} M^{r+1})e_{d+1}^r$ as required.
For the second statement just note that
$(E_i^{r+1} M^{r+1}) e_{d+1}^r = M^{r+1} 1^{r+1}_{d_r+d;i} e^r_{d+1}$
which is zero by Lemma~\ref{tl}(i)
since $i \notin I_r$.
\end{proof}

\begin{Lemma}\label{presstab}
The functor $E_i:\rmod{H}^\infty\rightarrow\rmod{H}^\infty$ sends stable modules to stable modules,
hence it restricts to $E_i:\rmod{H}\rightarrow\rmod{H}$.
\end{Lemma}

\begin{proof}
In view of Theorem~\ref{abelian1}, 
it suffices to show $E_i \circ \V (M) \cong \V \circ E_i(M)$ for each
$M \in \mathcal C$.
Take some large enough $r$ so that $i \in I_r$ and all composition
factors of $M$ are of the form $L(\lambda)$ for $\lambda \in
\Lambda_r$.
We may assume moreover that $M$ is homogeneous of degree $d$ in
$\mathcal C_r$, i.e.
$\V^r (M) = \Hom_{\mathcal C}(T^r_d, M)$.
We have that $E_i^r \circ \V^r (M)
\cong
\V^r \circ E_i (M)$ by Lemma~\ref{seq}.
The canonical isomorphism here is the map
$f^r:E_i^r \circ \V^r(M) 
=
\Hom_{\mathcal C}(F_i T^r_{d-1},M)
\stackrel{\sim}{\rightarrow} \Hom_{\mathcal C}(T^r_{d-1}, E_i M)$
defined by the adjunction between $F_i$ and $E_i$.
Similarly we have 
$f^{r+1}:E_i^{r+1} \circ \V^{r+1}(M) 
=
\Hom_{\mathcal C}(F_i T^{r+1}_{d_r+d-1},M)
\stackrel{\sim}{\rightarrow} \Hom_{\mathcal C}(T^{r+1}_{d_r+d-1}, E_i
M)$. Now we observe that
the following diagram commutes:
$$
\begin{CD}
E_i^r \circ \V^r (M)&@>f^r >>&\V^r \circ E_i(M)\\
@V\iota^r VV&&@VV \iota^r V\\
E_i^{r+1} \circ \V^{r+1}(M)&@>>f^{r+1}>&\V^{r+1} \circ E_i(M).
\end{CD}
$$
This follows by the naturality of the adjunction, noting that 
the left hand vertical map is induced by $F_i \theta^r_{d-1}:F_i
T_{d-1}^r \rightarrow F_i T_{d_r+d-1}^{r+1}$
as that agrees with $\theta^{r}_{d}$ on the summand
$F_i T_{d-1}^{r}$ of $T_d^r$ by the construction in
Lemma~\ref{tl}.
The isomorphisms $f^r$ for all sufficiently large $r$
define the required isomorphism $f:E_i \circ \V(M) \stackrel{\sim}{\rightarrow} \V \circ E_i(M)$.
\end{proof}

So now we have defined the endofunctors $F_i$ and $E_i$ on $\rmod{H}$.
It remains to define an adjunction making $(F_i, E_i)$ into an 
adjoint pair.
Given an object $M$ we pick $r$ large enough so that $i \in I_r$,
and $M, E_i M$ and $E_i F_i M$ are $r$-stable.
Then we take the unit and counit of the adjunction on object $M$
to be induced by the ones for $E_i^r$ and $F_i^r$ on $M^r \in \rmod{H^r}$.
As usual to prove naturality one needs to observe that the resulting morphisms are independent of the choice of $r$.

\begin{Theorem}\label{strongeq}
The category $\rmod{H}$, together with the adjoint pairs $(F_i, E_i)$
for each $i \in I$ and the natural transformations $\xi$ and $\tau$,
is an $\mathfrak{sl}_I$-categorification in the sense of Definition~\ref{catdef}.
Moreover the quotient functor
$\V:\mathcal C \rightarrow \rmod{H}$
is strongly equivariant.
\end{Theorem}

\begin{proof}
We have already defined all of the required data for a categorical
action.
To show that $\V$ is strongly equivariant, we use the 
isomorphism $\zeta_i:F_i \circ \V \stackrel{\sim}{\rightarrow} \V
\circ F_i$ from Lemma~\ref{dan}. We already showed this satisfies the
$\xi$- and $\tau$-versions of properties (E2)--(E3). It satisfies (E1)
since the composition of the natural transformation in (E1) with $\pr_r$
is an isomorphism on each
$M \in \mathcal C$ and sufficiently large $r$ thanks to
Lemma~\ref{seq}. Therefore it induces an isomorphism $\V \circ E_i \cong
E_i \circ \V$ (which we already exploited in the
proof of Lemma~\ref{presstab}).
It remains to verify the axioms (SL1$'$)--(SL4$'$).
The first two follow  via Lemma~\ref{okay} and the truth of 
the corresponding axioms
on each $\rmod{H^r}$.
For (SL3$'$), we can produce a second adjunction making $(E_i, F_i)$
into an adjoint pair by pushing some choice of such an adjunction in
$\mathcal C$ through the functor $\V$; this argument uses our
isomorphisms
$F_i \circ \V \cong \V
\circ F_i$ and
$\V \circ E_i \cong
E_i \circ \V$.
Finally (SL4$'$) holds because the left adjoint functor $\V^!$ 
embeds $[\rmod{H}]$ into $[\mathcal C]$, and the latter is integrable.
\end{proof}

\subsection{Proof of Theorem~\ref{lwmain} for infinite intervals}\label{uniq}
Now we can complete the proof of Theorem~\ref{lwmain} for infinite
intervals. Existence follows from
Theorem~\ref{catact}, so we just need to establish the uniqueness. 
Suppose we are 
given another $\mathfrak{sl}_I$-tensor product categorification $\mathcal C'$ of the same type as $\mathcal C$. Introduce the primed analog $\V'$ 
of the quotient functor $\V$,
and set $Y'(\lambda) :=
\V' P'(\lambda)$ for each $\lambda \in \Lambda$.
Letting $A$ and $A'$ be the basic algebras underlying $\mathcal C$ and
$\mathcal C'$ as in (\ref{A1}) and (\ref{A2}),
the double centralizer property from Theorem~\ref{ffp} 
implies the existence of isomorphisms
(\ref{C1})--(\ref{C2}).
Applying 
the last assertion of Theorem~\ref{ffp}, we pick isomorphisms $Y(\lambda)\cong Y'(\lambda)$ in $\rmod{H}$ for each $\lambda \in \Lambda$.
These induce algebra isomorphisms $A \cong A'$ hence an isomorphism
of categories $\rmod{A} \stackrel{\sim}{\rightarrow} \rmod{A'}$.
Since $\mathcal C$ is equivalent to $\rmod{A}$ and $\mathcal C'$
is equivalent to $\rmod{A'}$, we get the desired equivalence
$\mathbb{G}:\mathcal C \rightarrow \mathcal C'$.
To see that this equivalence is strongly equivariant we just need to 
introduce the bimodules $B$ and $B'$ from (\ref{B1}) and (\ref{B2}).
These satisfy (\ref{D1})--(\ref{E2}) for the same reasons as in $\S$\ref{finite},
using Theorem~\ref{strongeq} in place of Lemma~\ref{seq}.

\section{Graded tensor product categorifications}\label{gtpc}

We are in the business now of constructing graded lifts of the 
structures introduced so far. We are going to 
use the same notation in this section for 
graded versions as was used in the earlier sections in the ungraded setting.
To avoid confusion we add bars to all our earlier notation.
For example we denote the natural $\mathfrak{sl}_I$-module $V_I$ now by $\overline{V}_I$, so that the notation $V_I$ can be reused for
its quantum analog.

\subsection{Quantized enveloping algebras}\label{qgp}
Consider the field $\Q(q)$ equipped with the {\em bar involution}
defined by $\overline{f(q)} := f(q^{-1})$.
For an interval $I$, the quantized enveloping algebra $U_q \mathfrak{sl}_I$
is the $\Q(q)$-algebra 
with generators $\{f_i, e_i, k_i, k_i^{-1}\:|\:i \in I\}$ subject
to well-known relations.
We will often appeal to facts from \cite{Lubook}; note for this that our $q$ is Lusztig's $v^{-1}$ while our $f_i, e_i, k_i$ are Lusztig's $F_i, E_i, K_i^{-1}$.
We make $U_q \mathfrak{sl}_I$ into a Hopf algebra as in \cite{Lubook} with 
comultiplication $\Delta$ defined from
\begin{equation*}
\Delta(f_i) := 1 \otimes f_i + f_i \otimes k_i,\quad
\Delta(e_i) := k_i^{-1} \otimes e_i + e_i \otimes 1,\quad
\Delta(k_i) := k_i \otimes k_i.
\end{equation*}
There is a linear algebra antiautomorphism $*:U_q\mathfrak{sl}_I \rightarrow U_q\mathfrak{sl}_I$
such that
\begin{align}\notag
f_i^* &:= q e_i k_i,
&e_i^* &:= q f_i k_i^{-1},
&k_i^* &:= k_i.\\
\intertext{This is a coalgebra automorphism whose square is the identity.
Also $U_q \mathfrak{sl}_I$ possesses an antilinear 
algebra automorphism 
$\psi$, also usually called the {bar involution}, which is defined from}
\label{one}
\psi(f_i) &:= f_i,
&\psi(e_i) &:= e_i,
&\psi(k_i) &:= k_i^{-1}.\\\intertext{Let $\psi^* := * \circ \psi \circ *$, so that $\psi^*$ is another antilinear
algebra automorphism with}
\label{two}
\psi^*(f_i^*) &= f_i^*,
&\psi^*(e_i^*) &= e_i^*,
&\psi^*(k_i^*) &= {k_i^*}^{-1}.\\\intertext{Equivalently}
\notag
\psi^*(f_i) &= q^2 f_i k_i^{-2},&
\psi^*(e_i) &= q^2 e_i k^2_i,&
\psi^*(k_i) &= k_i^{-1}.
\end{align}

For $\varpi \in P_I$ the $\varpi$-weight space of a
$U_q \mathfrak{sl}_I$-module $M$ is the subspace
$$
M_\varpi := \{v \in M\:|\:k_i v = q^{\varpi\cdot\alpha_i} 
v\text{ for each $i \in I$}\}.
$$
Then the notion of integrable module is defined in the same way as for $\mathfrak{sl}_I$.
For $n \geq 0$ and $c \in \{0,1\}$,
let $\bigwedge^{n,c} V_I$ be the integrable $U_q\mathfrak{sl}_I$-module
on basis $\{v_\lambda\:|\:\lambda \in \Lambda_{I;n,c}\}$, with $f_i$ and $e_i$
acting by the same formulae (\ref{cact1})--(\ref{cact2}) as before
and 
$k_i v_\lambda := q^{\lambda_i-\lambda_{i+1}} v_\lambda$. 
Note that $f_i^*$ and $e_i^*$ act on $\bigwedge^{n,c} V_I$ in exactly the same way as $e_i$ and $f_i$, respectively.

More generally given a type $(\bn,\bc)$ of level $l$ we have the tensor product
$\textstyle
\bigwedge^{\bn,\bc} V_I
:= \bigwedge^{n_1,c_1} V_I\otimes\cdots\otimes \bigwedge^{n_l,c_l} V_I$.
It has a 
monomial basis $\{v_\lambda\:|\:\lambda \in \Lambda_{I;\bn,\bc}\}$ just like before.
The actions of $f_j$ and $e_j$ on these basis vectors are given explicitly by the formulae
\begin{align}\label{eact}
f_j v_\lambda&=
\sum_{\substack{1 \leq i \leq l \\ \lambda_{ij}=1 \\ \lambda_{i(j+1)}=0}} q^{(|\lambda_{i+1}|+\cdots+|\lambda_l|) \cdot\alpha_j} v_{t_{ij}(\lambda)},\\
e_j v_\lambda&=
\sum_{\substack{1 \leq i \leq l \\ \lambda_{ij}=0 \\ \lambda_{i(j+1)}=1}} q^{-(|\lambda_1|+\cdots+|\lambda_{i-1}|)\cdot \alpha_j} v_{t_{ij}(\lambda)},\label{fact}
\end{align}
writing $t_{ij}(\lambda)$ for the $01$-matrix obtained from $\lambda$ 
by flipping its entries $\lambda_{ij}$ and $\lambda_{i(j+1)}$.
In general it is no longer the case that $f_i^*$ and $e_i^*$ act in the same way as $e_i$ and $f_i$.
Instead we have that
\begin{equation}\label{form}
(u v,w) = (v,u^* w)
\end{equation}
for $u \in U_q \mathfrak{sl}_I$ and $v,w \in \bigwedge^{\bn,\bc} V_I$,
where $(-,-)$ is the symmetric bilinear form on $\bigwedge^{\bn,\bc} V_I$
with respect to which the monomial basis is orthonormal.

\subsection{Graded lifts}
In this subsection, we review briefly some basic notions 
related to graded lifts of Abelian categories and functors. 
The ideas here have their origins in \cite{Soergel, BGS}; our exposition is based on
\cite[Appendix E]{AJS}. 

\begin{Definition}\rm
By a {\em graded category} we mean a
category $\mathcal C$ equipped with an adjoint pair
$(Q,Q^{-1})$ of self-equivalences; 
the adjunction induces canonical isomorphisms
$Q^m \circ Q^n \stackrel{\sim}{\rightarrow} Q^{m+n}$
for all $m,n \in \Z$ making
the obvious square of isomorphisms from $Q^m\circ Q^n \circ Q^l$ to
$Q^{m+n+l}$ commute.
Let $\widehat{\mathcal C}$ denote
the category with the same
objects as $\mathcal C$ and
$$
\Hom_{\widehat{\mathcal C}}(M,N) := \bigoplus_{m\in\Z} 
\Hom_{\mathcal C}(M,N)_m
$$
where
$\Hom_{\mathcal C}(M,N)_m$ denotes 
$\Hom_{\mathcal C}(Q^m M, N)
\cong\Hom_{\mathcal C}(M, Q^{-m}N)$;
composition is induced by that of $\mathcal C$ 
making
$\widehat{\mathcal C}$ into a category enriched in graded vector spaces.
We refer to elements of
$\Hom_{\mathcal C}(M,N)_m$
as {\em homogeneous morphisms of degree $m$}.
Thus morphisms in $\mathcal C$ itself are the
homogeneous morphisms of degree zero in $\widehat{\mathcal C}$.
Assuming $\mathcal C$ is Abelian, we define $$\Ext^n_{\widehat{\mathcal C}}(M,N) := \bigoplus_{m \in \Z}
\Ext^n_{\mathcal C}(M,N)_m$$ 
similarly. 

A {\em graded functor} $F:\mathcal C \rightarrow \mathcal C'$ between
two graded categories means a 
functor between the underlying categories 
plus the additional data of an isomorphism of functors
$\gamma_F:F \circ Q \stackrel{\sim}{\rightarrow} Q' \circ F$; using 
the given adjunctions we get from this canonical isomorphisms
$F \circ Q^n\stackrel{\sim}{\rightarrow} {Q'}^n\circ F$ for all $n
\in \Z$ making the obvious pentagon of isomorphisms from 
$F \circ Q^m \circ Q^n$
to ${Q'}^{m+n} \circ F$ commute for all $m,n\in\Z$.
There is an induced functor $\widehat{F}:\widehat{\mathcal C}
\rightarrow \widehat{\mathcal C}'$ which is equal to $F$ on objects;
on a morphism
$f \in \Hom_{\mathcal C}(M,N)_m$
we define
$\widehat{F} f \in \Hom_{\mathcal C'}(FM,FN)_m$ from $FM
\stackrel{Ff}{\rightarrow} F Q^{-m} N\stackrel{\sim}{\rightarrow}
{Q'}^{-m} FN$.
There is an obvious way to compose two graded functors to obtain
another graded functor. The identity functor
$\operatorname{id}:\mathcal C \rightarrow \mathcal C$ is a graded functor
with $\gamma_{\operatorname{id}}$ being $1_Q$. Also each $Q^n$ is a
graded functor in a canonical way, e.g. $\gamma_Q = 1_{Q^2}$.
A {\em graded equivalence} $F:\mathcal C \rightarrow \mathcal C'$
between two graded categories 
is a
graded functor that is also an equivalence of categories.
A {\em graded duality} is a graded equivalence
$D:\mathcal C \rightarrow \mathcal C^{\operatorname{op}}$,
viewing $\mathcal C^{\operatorname{op}}$ as a graded category via 
the inverse adjoint pair $(Q^{-1},Q)$ of self-equivalences to $(Q,Q^{-1})$.

A {\em graded natural transformation}
$\alpha:F \rightarrow G$ 
between graded functors $F, G:\mathcal C \rightarrow \mathcal C'$
is a natural transformation in the usual sense with the additional
property that
$Q' \alpha \circ \gamma_F = \gamma_G \circ \alpha Q$.
We write $\Hom(F,G)_0$ for the vector space of all graded natural
transformations $\alpha:F \rightarrow G$.
In fact, there is a strict 2-category 
whose objects are 
all (small) graded categories, whose 1-morphisms are
graded functors, and whose 2-morphisms are graded natural
transformations.
The morphism categories $\mathcal{H}om(\mathcal C, \mathcal C')$ in
this 2-category are themselves graded categories with grading shift
functors $Q$ and $Q^{-1}$ defined by horizontally composing on the left with the ones
in $\mathcal C'$.
This means that we also have the spaces $\Hom(F,G)_n := \Hom(Q^n F, G)_0$ for $n \in \Z$
consisting of {\em homogeneous natural transformations of degree $n$}.
\end{Definition}

Suppose $\mathcal C$ is a graded Abelian category
such that all objects have finite length,
there are enough projectives, and the endomorphism algebras of the 
irreducible objects are one dimensional.
Assume further that $\mathcal C$ is {\em acyclic} in the sense that
$L \not\cong Q^n L$ for each irreducible $L$ and $n \neq 0$.
Fix a choice 
of representatives 
$\{L(\lambda)\:|\:\lambda\in\Lambda\}$
for the homogeneous isomorphism classes of irreducible objects in $\mathcal C$,
and let $P(\lambda)$ be a projective cover of $L(\lambda)$.
Then the assignment $q := [Q]$ makes
$K_0(\mathcal C)$ and $G_0(\mathcal C)$ into free $\Z[q,q^{-1}]$-modules
with bases $\{[P(\lambda)]\:|\:\lambda \in \Lambda\}$ and
$\{[L(\lambda)]\:|\:\lambda \in \Lambda\}$, respectively.
Set $[\mathcal{C}]_q := \Q(q) \otimes_{\Z[q,q^{-1}]} K_0(\mathcal C)$
and $[\mathcal{C}]_q^* := \Q(q) \otimes_{\Z[q,q^{-1}]} G_0(\mathcal C)$.
Given $M, L\in\mathcal C$
with $L$ irreducible, 
we write $[M:L]_q$ for 
$\sum_{n \in \Z} q^n [M:Q^n L]$; in particular
$[M:L(\lambda)]_q = \qdim \Hom_{\widehat{\mathcal C}}(P(\lambda), M).$

Continuing with $\mathcal C$ as in the previous paragraph,
let $\boldrmod A$ be the category of finite dimensional locally unital graded
right modules over the locally unital graded algebra
\begin{equation}\label{gradedA}
A := \bigoplus_{\lambda,\mu \in\Lambda} \Hom_{\widehat{\mathcal C}}(P(\lambda),P(\mu)).
\end{equation}
This can be viewed as a graded category with $Q$ being the
grading shift functor, i.e. 
$QM$ is the same underlying module with new
grading defined from $(QM)_n := M_{n-1}$.
We write simply $\Hom_A(M,N)$ for morphisms in the enriched category
$\widehat{\boldrmod{A}}$.
Then the functor
\begin{equation}\label{stupid2}
\mathbb{H}:\mathcal C \rightarrow \boldrmod A, \qquad
M \mapsto \bigoplus_{\lambda \in \Lambda}
\Hom_{\widehat{\mathcal C}}(P(\lambda), M)
\end{equation}
is a graded equivalence.
The right ideals $1_\lambda A$ are all finite dimensional.
The left ideals $A 1_\lambda$ are all
finite dimensional if and only if $\mathcal C$ has enough injectives; in that case $\mathcal C$ is an {\em acyclic graded Schurian category}.

\begin{Definition}\rm\label{liftdef}
By a {\em graded lift} of a Schurian category $\overline{\mathcal C}$, we mean a
graded Abelian category $\mathcal C$ 
together with a fully faithful functor
$\nu:\widehat{\mathcal C} \rightarrow \overline{\mathcal C}$
such that 
\begin{itemize}
\item[(GL1)]
$\nu$ is {dense on projectives}, i.e.
every projective object $\overline{P} \in \overline{\mathcal C}$ is isomorphic to $\nu P$ for some projective object $P \in \mathcal C$;
\item[(GL2)]
$\nu \circ \widehat{Q} \cong \nu$.
\end{itemize}
The hypotheses that $\nu$ is fully faithful and dense on projectives imply
that the restriction of $\nu$ to $\mathcal C$ is exact.
\end{Definition}

Given a graded lift $\mathcal C$ of $\overline{\mathcal C}$,
the {\em gradable objects} 
of $\overline{\mathcal C}$ mean the objects in the essential image of $\nu$.
All projective objects of $\overline{\mathcal C}$ are gradable by
the definition (GL1). We will see shortly that all irreducible and all injective objects of
$\overline{\mathcal C}$ are gradable too.
An important point is that graded lifts of an indecomposable object
are unique up to homogeneous isomorphism; see
\cite[Lemma 2.5.3]{BGS}.
Note also that the functor $\nu$ induces a canonical isomorphism
$$
\Ext^n_{\widehat{\mathcal C}}(M,N) \stackrel{\sim}{\rightarrow}
\Ext^n_{\overline{\mathcal C}}(\nu  M, \nu N)
$$
for any $M, N \in \mathcal C$ and $n \geq 0$.

The problem of finding a graded lift $\mathcal C$
of a Schurian category $\overline{\mathcal C}$ is easy to understand in terms of algebras as follows.
Let $\{\overline{L}(\lambda)\:|\:\lambda\in\Lambda\}$
be a complete set of pairwise non-isomorphic irreducible objects of
$\overline{\mathcal C}$. 
Let $\overline{P}(\lambda)$ be a projective cover of $\overline{L}(\lambda)$ in $\overline{\mathcal C}$. Recall from (\ref{stupid1}) that $\overline{\mathcal C}$
is equivalent to the category $\rmod A$ where
\begin{equation}\label{ungradedA}
A := \bigoplus_{\lambda,\mu \in \Lambda} \Hom_{\overline{\mathcal C}}(\overline{P}(\lambda), \overline{P}(\mu)).
\end{equation}
Now suppose that $\mathcal C$ is a graded lift of $\overline{\mathcal C}$, and
pick graded lifts $P(\lambda)$ of each $\overline{P}(\lambda)$. 
As $\End_{\widehat{\mathcal C}}(P(\lambda)) \cong \End_{\overline{\mathcal C}}(\overline{P}(\lambda))$ which is local, $P(\lambda)$
is an indecomposable projective object of $\mathcal C$
with irreducible head denoted $L(\lambda)$. It is easy to see that 
$L(\lambda)$ is a graded lift of $\overline{L}(\lambda)$, i.e. all irreducible objects of $\overline{\mathcal C}$ are gradable. 
Moreover if $L$ is any irreducible object of $\mathcal C$ then $\nu L$ is
irreducible in $\overline{\mathcal C}$, from which we get that 
$L \cong Q^n L(\lambda)$ for some $n \in \Z$ and $\lambda \in \Lambda$.
Also each $\End_{\widehat{\mathcal C}}(L(\lambda))$ is one-dimensional,
hence using also (GL2) 
we see
that $\mathcal C$ is acyclic.
Thus $
\{Q^n L(\lambda)\:|\:\lambda\in\Lambda, n \in \Z\}
$ 
is a complete set of pairwise non-isomorphic irreducible objects in $\mathcal C$.
Finally note that every object of $\mathcal C$ has finite length by the exactness of $\nu$.
This puts us in the setup of (\ref{stupid2}). 
Using $\nu$ to identify the algebras
$A$ from (\ref{gradedA}) and (\ref{ungradedA}), 
the following diagram of functors commutes up to
isomorphism:
\begin{equation}\label{ged}
\begin{CD}
\widehat{\mathcal C}&@>\widehat{\mathbb{H}}>>&\widehat{\boldrmod A}\phantom{\,.}\\
@V\nu VV&&@VV\nu V\\
\overline{\mathcal C}&@>>\overline{\mathbb{H}}>&\rmod A\,.
\end{CD}
\end{equation}
Here, the functor $\nu$ on the right is the obvious functor that forgets the grading.
In this way we see that a choice of a graded lift of $\overline{\mathcal C}$ amounts to choosing a $\Z$-grading on the underlying basic algebra $A$
with respect to which the idempotents $1_\lambda \in A$ are
homogeneous.

\begin{Lemma}\label{acycliclemma}
Let $\mathcal C$ be a graded lift of Schurian category $\overline{\mathcal C}$
with notation as above.
For each $\lambda \in \Lambda$ 
the irreducible object $L(\lambda)$ has an injective hull $I(\lambda)$
in $\mathcal C$, which is a graded lift of the injective hull of $\overline{L}(\lambda)$ in $\overline{\mathcal C}$.
Hence $\mathcal C$ is an acyclic graded Schurian category.
\end{Lemma}

\begin{proof}
In view of (\ref{ged}),
we may assume that $\overline{\mathcal C} = \rmod{A}$ and $\mathcal C
= \boldrmod{A}$ for a locally unital graded algebra $A$.
The indecomposable injectives in $\rmod A$ are the linear duals of the (necessarily finite dimensional) left ideals $A
1_\lambda$. They are naturally graded, hence give indecomposable injective objects in $\boldrmod{A}$ too. This proves the first statement.
Hence $\mathcal C$ has enough injectives. All the other properties of an acyclic graded Schurian category have already been verified above.
\end{proof}

Finally let
 $\overline{F}:\overline{\mathcal C} \rightarrow \overline{\mathcal C}'$ be a functor between two Schurian categories, and $\mathcal C$ and $\mathcal C'$ be graded lifts
of $\overline{\mathcal C}$ and $\overline{\mathcal C}'$, respectively. A {\em graded lift}
of $\overline{F}$ means a graded functor $F:\mathcal C\rightarrow\mathcal C'$ such that $\nu' \circ \widehat{F} \cong \overline{F} \circ \nu$.
Assuming $\overline{F}$ has a right adjoint,
it corresponds to a functor
between the underlying module categories 
$\rmod{A}$ and $\rmod{A'}$
that is defined 
by tensoring with an $(A,A')$-bimodule.
Then a choice of graded lift amounts to choosing a $\Z$-grading on this bimodule
making it into a graded $(A, A')$-bimodule.

\subsection{Graded highest weight categories}
The next definition is the graded analog of
Definition~\ref{hwdef};
the basic example to keep in mind is
the category $\boldrmod{A}$ of finite dimensional graded modules over a
graded quasi-hereditary algebra $A$ in the sense of \cite{CPSgr}.

\begin{Definition}\rm\label{ghwdef}
A {\em graded highest weight category} is an acyclic graded Schurian category
$\mathcal C$ plus the data of a 
distinguished set of irreducible objects 
$\{L(\lambda)\:|\:\lambda\in\Lambda\}$
indexed by some interval-finite poset $\Lambda$ such that the following two axioms hold.
\begin{itemize}
\item[(GHW1)]
Every irreducible object of $\mathcal C$ is isomorphic to $Q^n L(\lambda)$
for unique $\lambda \in \Lambda$ and $n \in \Z$.
Hence letting $P(\lambda)$ be a projective cover of $L(\lambda)$ in $\mathcal C$,
the objects $\{Q^n P(\lambda)\:|\:\lambda \in \Lambda, n \in \Z\}$ give a complete set of pairwise non-isomorphic indecomposable projective objects in $\mathcal C$.
\item[(GHW2)]
Define the {\em standard object}
$\Delta(\lambda)$ to be the largest quotient of $P(\lambda)$
such that $[\Delta(\lambda):L(\mu)]_q = \delta_{\lambda,\mu}$
for $\mu \not< \lambda$.
Then $P(\lambda)$ has a filtration with top section isomorphic to
$\Delta(\lambda)$ and other sections of the form $Q^n\Delta(\mu)$ for 
$n \in \Z$ and $\mu > \lambda$.
\end{itemize}
There is also an equivalent dual formulation of this definition in terms of 
 indecomposable injective objects $I(\lambda)$
and costandard objects $\nabla(\lambda)$.
\end{Definition}

An object $M$ of a graded highest weight category $\mathcal C$ has a {\em graded $\Delta$-flag} if it has a filtration with sections of the form $Q^n \Delta(\lambda)$ for $n \in \Z$ and $\lambda \in \Lambda$.
The notion of a {\em graded $\nabla$-flag} is defined similarly.
We let $\mathcal C^\Delta$ be the exact subcategory of $\mathcal C$ consisting of all objects
with a graded $\Delta$-flag.
The $\Q(q)$-form 
$[\mathcal C^\Delta]_q$ 
of its 
Grothendieck group 
has basis $\{[\Delta(\lambda)]\:|\:\lambda \in \Lambda\}$. 

The other basic facts about highest weight categories from $\S$\ref{rhw}
also extend to the graded setting. 

If $\mathcal C$ is a graded highest weight category and $\overline{\mathcal C}$ is the underlying Schurian category, i.e. $\mathcal C$ is a graded lift of $\overline{\mathcal C}$, then it is easy to check that $\overline{\mathcal C}$ is a highest weight category in the sense of Definition~\ref{hwdef} with the same weight poset and standard objects
$\overline{\Delta}(\lambda) := \nu \Delta(\lambda)$. The following lemma establishes the converse of this statement.

\begin{Lemma}\label{tol}
Suppose that $\mathcal C$ is a graded lift of a highest weight category 
$\overline{\mathcal C}$.
For each $\lambda \in \Lambda$
let $L(\lambda)$ be some choice of graded lift
of the irreducible object $\overline{L}(\lambda)$
of $\overline{\mathcal C}$.
Then $\mathcal C$ is 
a graded highest weight category with $\{L(\lambda)\:|\:\lambda\in\Lambda\}$
as its distinguished irreducible objects.
Moreover its standard objects are graded lifts of the ones in $\overline{\mathcal C}$.
\end{Lemma}

\begin{proof}
We know that $\mathcal C$ is an acyclic graded Schurian category by Lemma~\ref{acycliclemma}.
Let $A$ be as in (\ref{gradedA}).
We may as well assume that
$\overline{\mathcal C}$ is $\rmod A$ and
$\mathcal C$ is 
$\boldrmod A$.
So we can take
$$
\overline{P}(\lambda) = 1_\lambda A = P(\lambda),
\qquad
\Hom_{\overline{\mathcal C}}(\overline{P}(\mu), \overline{P}(\lambda))
= 1_\lambda A 1_\mu = \Hom_{\widehat{\mathcal C}}(P(\mu), P(\lambda)).
$$
By (HW) the standard object $\overline{\Delta}(\lambda)$ in $\overline{\mathcal C}$ is 
the quotient of $P(\lambda)$ 
by the submodule generated by the images of all (not necessarily homogeneous)
homomorphisms $P(\mu) \rightarrow P(\lambda)$ 
for $\mu > \lambda$.
The image of an arbitrary homomorphism is contained in the sum of the images of its homogeneous pieces, so $\overline{\Delta}(\lambda)$
can also be described as the quotient of $P(\lambda)$ by the submodule generated by the images of all homogeneous homomorphisms $Q^n P(\mu) \rightarrow P(\lambda)$
of degree zero for $\mu > \lambda$ and $n \in \Z$. Thus it is gradable and a graded lift is given by
$\Delta(\lambda)$ as defined by (GHW2).
Similarly each $\nabla(\lambda)$ is a graded lift of the corresponding costandard object
$\overline{\nabla}(\lambda)$ of $\overline{\mathcal C}$.
Then an argument mimicking the usual proof of the criterion for a module
to have $\Delta$-flag shows 
that $P(\lambda)$ has a graded $\Delta$-flag.
\end{proof}

\subsection{Graded categorifications}
Next we formulate the graded version of the definition of $\mathfrak{sl}_I$-categorification from
Definition~\ref{catdef}.

\begin{Definition}\rm\label{qcatdef}
A {\em $U_q\mathfrak{sl}_I$-categorification} is an acyclic graded Schurian category
$\mathcal C$ 
together with the data of graded endofunctors $F_i, E_i, K_i$ and $K_i^{-1}$ 
for each $i \in I$,
an adjunction making $F_i^* := Q E_i K_i$ into a right adjoint 
to $F_i$,
and homogeneous natural transformations $\xi \in \Hom(F_i, F_i)_2$
and $\tau \in \Hom(F_j \circ F_i, F_i\circ F_j)_{-\alpha_i\cdot\alpha_j}$ for each $i,j \in I$,
such that the following axioms hold.
\begin{itemize}
\item[(GSL1)]
There is a decomposition $\mathcal C = \bigoplus_{\varpi \in P_I} \mathcal C_{\varpi}$
such that 
$K_i|_{\mathcal C_\varpi}\cong Q^{\varpi\cdot\alpha_i}$ and
$K_i^{-1}|_{\mathcal C_\varpi}\cong Q^{-\varpi\cdot\alpha_i}$.
\item[(GSL2)]
Letting $F := \bigoplus_{i \in I} F_i$,
the endomorphisms $\xi_j := \widehat{F}^{d-j} \xi \widehat{F}^{j-1}$
and $\tau_k := \widehat{F}^{d-k-1} \tau \widehat{F}^{k-1}$ of $\widehat{F}^d$ 
plus the projections $1_\bi$ of $\widehat{F}^d$ onto its summands
$\widehat{F}_\bi := \widehat{F}_{i_d}\circ\cdots\circ \widehat{F}_{i_1}$ for each $\bi
\in I^d$ satisfy the relations
of the quiver Hecke algebra $QH_{I,d}$.
\item[(GSL3)]
Each functor $E_i^* := Q F_i K_i^{-1}$ is isomorphic to a right adjoint of $E_i$.
\item[(GSL4)] The endomorphisms $f_i, e_i$ and $k_i$ of $[\mathcal C]_q$
  induced by the functors $F_i, E_i$ and $K_i$ make $[\mathcal C]_q$ into an integrable $U_q\mathfrak{sl}_I$-module.
\end{itemize}
The last axiom here has the following equivalent dual formulation.
\begin{itemize}
\item[(GSL4$^*$)] The endomorphisms $f_i, e_i$ and $k_i$ of $[\mathcal
  C]^*_q$ induced by the functors $F_i, E_i$ and $K_i$ make $[\mathcal C]^*_q$ into an integrable $U_q\mathfrak{sl}_I$-module.
\end{itemize}

There is also a graded analog of Definition~\ref{sev}:
a {\em strongly equivariant graded functor} (resp. {\em equivalence}) between $U_q
\mathfrak{sl}_I$-categorifications $\mathcal C$ and $\mathcal C'$
is a graded functor (resp. equivalence)
$\mathbb{G}:\mathcal C \rightarrow \mathcal C'$ which is strongly equivariant as before, 
such that the required isomorphisms
$\zeta_i:F_i' \circ \mathbb{G} \stackrel{\sim}{\rightarrow} \mathbb{G}
\circ F_i$ is also a graded natural transformation.
\end{Definition}

The axioms (GSL1)--(GSL4) are equivalent to saying that the given data defines a 
homogeneous action of the Kac-Moody 2-category associated to
$\mathfrak{sl}_I$ on $\mathcal C$ making it into a graded integrable 
$2$-representation in the sense of \cite{R}. 
This is a variation on \cite[Theorem 5.30]{R}; see the proofs of the next two lemmas.

\begin{Lemma}\label{forward}
If $\mathcal C$ is a $U_q\mathfrak{sl}_I$-categorification that is a graded lift of a Schurian category $\overline{\mathcal C}$,
then there is an induced structure of $\mathfrak{sl}_I$-categorification on 
$\overline{\mathcal C}$.
\end{Lemma}

\begin{proof}
Let us work in terms of graded modules over the algebra $A$ from (\ref{gradedA}).
Then $F$ and $E$ are defined by tensoring with certain graded
$(A,A)$-bimodules, and $\xi$, $\tau$ and the various adjunctions
become homogeneous graded bimodule homomorphisms. Forgetting the grading we get
endofunctors $\overline{F}$ and $\overline{E}$
of $\rmod{A}$, natural transformations $\overline{\xi}$ and $\overline{\tau}$
satisfying the relations of (SL2$'$),
and adjunctions both ways round between $\overline{F}$ and $\overline{E}$.
It remains to verify the axioms (SL1$'$) and (SL4$'$).
The first follows as $\xi$ acts nilpotently on any graded projective
by degree considerations, hence $\overline{\xi}$ is locally nilpotent too.
Finally (SL4$'$) follows from (GSL4) specialized at $q=1$.
\end{proof}

\begin{Lemma}\label{backwards}
Let $\overline{\mathcal C}$ be an $\mathfrak{sl}_I$-categorification, denoting its various functors and natural transformations by $\overline{F}_i, \overline{E}_i,
\overline{\xi}$ and $\overline{\tau}$.
For each $\varpi \in P_I$ 
let $\overline{\mathcal C}_{\varpi}$ be the full 
subcategory of $\overline{\mathcal C}$ 
consisting of all objects $\overline{M}$ such that $[\overline{M}]$
lies in the $\varpi$-weight space of the integrable
$\mathfrak{sl}_I$-module $[\overline{\mathcal C}]$. 
Suppose that we are given the following additional data:
\begin{itemize}
\item[(i)]
graded lifts $\mathcal C_{\varpi}$ of
each $\overline{\mathcal C}_\varpi$, hence a graded lift
$\mathcal C = \bigoplus_{\varpi \in P_I} \mathcal C_\varpi$
of $\overline{\mathcal C}$;
\item[(ii)] graded functors $K_i$ and $K_i^{-1}$ satisfying (GSL1);
\item[(iii)]
graded lifts $F_i$ and $E_i$ of the functors 
$\overline{F}_i$ and $\overline{E}_i$ together with an adjunction making
$F_i^* = Q E_i K_i$ into a right adjoint to $F_i$;
\item[(iv)]
graded lifts $\xi \in \Hom(F_i, F_i)_2$ and $\tau
\in \Hom(F_j \circ F_i, F_i \circ F_j)_{-\alpha_i\cdot\alpha_j}$
of $\overline{\xi}$ and $\overline{\tau}$, meaning
 that
$\chi \circ \nu \xi= \overline{\xi}\nu \circ \chi$
in $\Hom(\nu\circ\widehat{F}, \overline{F}\circ \nu)$
and $\overline{F}\chi \circ \chi\widehat{F} \circ \nu \tau = \overline{\tau} \nu \circ \overline{F}\chi \circ \chi\widehat{F}$
in $\Hom(\nu \circ \widehat{F}^2, \overline{F}^2 \circ \nu)$
for some choice of the
isomorphism $\chi:\nu \circ \widehat{F} \stackrel{\sim}{\rightarrow}
\overline{F} \circ \nu$.
\end{itemize}
Then $\mathcal C$ is a $U_q\mathfrak{sl}_I$-categorification.
(Note we do not insist that the graded adjunction in (iii) is a lift of the given adjunction between $\overline{F}_i$ and $\overline{E}_i$.)
\end{Lemma}

\begin{proof}
We know that $\mathcal C$ is an acyclic graded Schurian category by Lemma~\ref{acycliclemma}.
In view of \cite[Theorem 5.30]{R}, we have the data required to define a
homogeneous action of the Kac-Moody 2-category associated to $\mathfrak{sl}_I$
on $\mathcal C$, making it into a graded integrable
$2$-representation. The axioms (GSL1)--(GSL2) follow immediately from
the definition of the latter, while (GSL3) follows from \cite[Theorem
5.16]{R} (or \cite[Theorem 4.3]{Bnew}).
To check (GSL4) we need to verify the relations of $U_q \mathfrak{sl}_I$.
Almost all of them follow from the $2$-relations in the Kac-Moody
2-category.
We are just left with the quantum Serre relations which follow from \cite[Lemma 3.13]{R}.
\end{proof}

\subsection{Graded tensor product categorifications}\label{second}
Given a type $(\bn,\bc)$ of level $l$, we can at last formulate the graded analog of Definition~\ref{tpcdef}.

\begin{Definition}\rm\label{qtpcdef}
A {\em $U_q\mathfrak{sl}_I$-tensor product categorification}
 of type $(\bn,\bc)$
is a graded highest weight category $\mathcal C$
together with graded endofunctors $F_i, E_i, K_i$ and $K_i^{-1}$,
an adjunction making $F_i^* := Q E_i K_i$ into a right adjoint 
to $F_i$,
and homogeneous natural transformations $\xi$ and $\tau$ as above such that
such that (GSL1)--(GSL3) and (GTP1)--(GTP3) hold.
\begin{itemize}
\item[(GTP1)]
Same as (TP1).
\item[(GTP2)]
The exact functors $F_i$ and $E_i$ send objects with graded
$\Delta$-flags to objects with graded $\Delta$-flags.
\item[(GTP3)]
The linear isomorphism
$[\mathcal C^\Delta] \stackrel{\sim}{\rightarrow} \bigwedge^{\bn,\bc} V_I,
[\Delta(\lambda)] \mapsto v_\lambda$ intertwines the endomorphisms
$f_i, e_i$ and $k_i$ of $[\mathcal C^\Delta]$ induced by 
$F_i, E_i$ and $K_i$ with the endomorphisms of $\bigwedge^{\bn,\bc} V_I$ arising from
the actions of $f_i, e_i, k_i \in U_q\mathfrak{sl}_I$.
\end{itemize}
Since $[\mathcal C]_q$ embeds into $[\mathcal C^\Delta]_q =
\bigwedge^{\bn,\bc} V_I$, these axioms imply that (GSL4) holds, so
that
$\mathcal C$ is a $U_q\mathfrak{sl}_I$-categorification in the sense of Definition~\ref{qcatdef} too.
\end{Definition}

\begin{Lemma}\label{oppy}
Suppose that $\mathcal C$ is a $U_q \mathfrak{sl}_I$-tensor product categorification as above.
Then $\mathcal C^{\operatorname{op}}$ is a $U_q \mathfrak{sl}_I$-tensor product categorification with
categorification functors $F_i^{\operatorname{op}} := Q^2 F_i K_i^{-2}$,
$E_i^{\operatorname{op}} := Q^2 E_i K_i^2$,
$K_i^{\operatorname{op}} := K_i^{-1}$ and
$(K_i^{\operatorname{op}})^{-1} := K_i$, taking the distinguished irreducible objects and other required
adjunctions and natural transformations to be the same as in $\mathcal C$.
\end{Lemma}

\begin{proof}
All of the axioms follow immediately except for (GTP2)--(GTP3).
To see these 
note using (GTP2)--(GTP3) for $\mathcal C$
and (\ref{fact}) that $F_j \Delta(\lambda)$
has a graded $\Delta$-flag 
with sections
$Q^{(|\lambda_{i+1}|+\cdots+|\lambda_l|)\cdot \alpha_j} \Delta(t_{ij}(\lambda))$
for $i=1,\dots,l$ such that $\lambda_{ij}=1$ and $\lambda_{i(j+1)}=0$.
Hence, by
an argument involving the adjoint pair $(F_j, F_j^*)$
and the homological criteria for graded $\Delta$- and $\nabla$-flags,
$F_j^* \nabla(\lambda)$ has a graded $\nabla$-flag
with sections
$Q^{-(|\lambda_{i+1}|+\cdots+|\lambda_l|)\cdot \alpha_j} \nabla(t_{ij}(\lambda))$
for $i=1,\dots,l$ such that $\lambda_{ij}=0$ and $\lambda_{i(j+1)}=1$.
Now rescale to deduce that
$E_j^{\operatorname{op}} \nabla(\lambda)$ has a filtration with sections
$Q^{(|\lambda_{1}|+\cdots+|\lambda_{i-1}|)\cdot \alpha_j} \nabla(t_{ij}(\lambda))$
for $i=1,\dots,l$ such that $\lambda_{ij}=0$ and $\lambda_{i(j+1)}=1$.
Comparing with (\ref{eact}) and bearing in mind that it is the shift functor
$Q^{-1}$ on $\mathcal C^{\operatorname{op}}$ that
induces $q$ on its Grothendieck group,
this checks (GTP2) and the equivariance from (GTP3) for $e_j$. A similar argument 
works for $f_j$, while the equivariance with respect to each $k_j$ is obvious.
\end{proof}

The truncation construction of $\S$\ref{trun} can obviously be applied also in the graded setting.
Thus, given some $U_q \mathfrak{sl}_I$-tensor product categorification $\mathcal C$ and a subinterval $J \subset I$,
the subquotient $\mathcal C_J = \mathcal C_{\leq J} / \mathcal C_{<J}$
has a naturally induced structure of $U_q \mathfrak{sl}_J$-tensor
product categorification; cf. Theorem~\ref{subq}.

If $\mathcal C$ is a $U_q\mathfrak{sl}_I$-tensor product categorification 
then the underlying 
Schurian category $\overline{\mathcal C}$ is an $\mathfrak{sl}_I$-tensor product categorification in sense of Definition~\ref{tpcdef}; 
this is just like in 
Lemma~\ref{forward}.
The main goal now is to prove the following theorem
going in the other direction; cf. Lemma~\ref{backwards}.

\begin{Theorem}\label{main}
Let $I$ be an interval and $\overline{\mathcal C}$ be an $\mathfrak{sl}_I$-tensor product categorification of any type.
\begin{itemize}
\item[(i)]
There exists a graded lift $\mathcal C$ of $\overline{\mathcal C}$
together with graded functors $F_i$, $E_i$, $K_i$ and $K_i^{-1}$, an adjunction making $F_i^* = Q E_i K_i$ into a right adjoint to $F_i$,
and homogeneous natural transformations $\xi$ and $\tau$,
satisfying all the hypotheses of Lemma~\ref{backwards}(i)--(iv).
\item[(ii)]
Given any choice for the data in (i),
there exist unique (up to isomorphism and a global shift)
graded lifts 
$L(\lambda) \in \mathcal C$ of the irreducible
objects of $\overline{\mathcal C}$
such that $\mathcal C$ 
satisfies (GTP2)--(GTP3),
viewing $\mathcal C$ as a graded highest weight category with these lifts as its distinguished irreducible objects.
Thus $\mathcal C$ becomes a
$U_q\mathfrak{sl}_I$-tensor product categorification.
\item[(iii)]
If $\mathcal C'$ 
is another $U_q\mathfrak{sl}_I$-tensor product categorification lifting
$\overline{\mathcal C}$ as in (i)--(ii), there is a strongly equivariant graded equivalence
$\mathbb{G}:\mathcal C \stackrel{\sim}{\rightarrow}
\mathcal C'$ with
$\nu' \circ \mathbb{G} \cong \nu$ and $\mathbb{G} L(\lambda) \cong L'(\lambda)$
for each weight $\lambda$.
\end{itemize}
\end{Theorem}

We will prove Theorem~\ref{main} in the next two subsections;
see also \cite[Corollary 6.3]{LW} for a closely related result when $I$ is finite.
Before we do that we record the following variation which combines Theorems~\ref{lwmain} and \ref{main}.

\begin{Theorem}\label{known}
Let $I$ be any interval and $(\bn,\bc)$ be any type.
Then there exists a
$U_q\mathfrak{sl}_I$-tensor product categorification $\mathcal C$ of type $(\bn,\bc)$.
It is unique in the sense that
if $\mathcal C'$ is another $U_q\mathfrak{sl}_I$-tensor product
categorification of an equivalent type 
then there exists a strongly equivariant graded equivalence
$\mathbb{G}:\mathcal C \rightarrow \mathcal C'$
such that $\mathbb{G} L(\lambda) \cong L'(\lambda)$ for each weight $\lambda$.
\end{Theorem}

\begin{proof}
The existence is clear from Theorems~\ref{lwmain} and \ref{main}(i)--(ii).
For the uniqueness, given $\mathcal C$ and $\mathcal C'$, the underlying Schurian categories $\overline{\mathcal C}$ and
$\overline{\mathcal C}'$ are tensor product categorifications of equivalent types.
Hence by Theorem~\ref{lwmain} there exists a strongly equivariant
equivalence $\overline{\mathbb{G}}:\overline{\mathcal C}
\stackrel{\sim}{\rightarrow} \overline{\mathcal C}'$.
Thus $\mathcal C$ together with the functor
$\overline{\mathbb G} \circ \nu:\widehat{\mathcal C} \rightarrow \overline{\mathcal C}'$ is a graded lift of $\overline{\mathcal C}'$ in the sense of Theorem~\ref{main}(i)--(ii), as of course is $\mathcal C'$ with its given forgetful functor
$\nu':\widehat{\mathcal C}' \rightarrow \overline{\mathcal C}'$.
Hence Theorem~\ref{main}(iii) gives us the desired 
strongly equivariant graded equivalence
$\mathbb{G}:\mathcal C \rightarrow \mathcal C'$
with $\nu' \circ \mathbb{G} \cong \overline{\mathbb{G}} \circ \nu$
and $\mathbb{G} L(\lambda) \cong L'(\lambda)$.
\end{proof}

\begin{Corollary}\label{gradedduality}
Any $U_q\mathfrak{sl}_I$-tensor product categorification $\mathcal C$ admits a 
graded duality
$\circledast$ 
with
$\circledast \circ F_i^* \cong F_i^* \circ \circledast$ and
$\circledast \circ E_i^* \cong E_i^* \circ \circledast$ (as graded
functors);
moreover, $L(\lambda) \cong L(\lambda)^\circledast$ for each weight
$\lambda$.
Similarly its category of projectives has a graded duality
$\#$
with
$\# \circ F_i \cong F_i \circ \#$ and
$\# \circ E_i \cong E_i \circ \#$; moreover
$P(\lambda) \cong P(\lambda)^\#$ for each
$\lambda$.
\end{Corollary}

\begin{proof}
View $\mathcal C^{\operatorname{op}}$ as a
$U_q \mathfrak{sl}_I$-tensor product categorification as in
Lemma~\ref{oppy}.
Then Theorem~\ref{known} implies that there exists
a strongly equivariant graded equivalence
$\circledast:\mathcal C \rightarrow \mathcal C^{\operatorname{op}}$.
In particular this means that
$\circledast \circ F_i \cong Q^2 F_i K_i^{-2} \circ \circledast$,
which is equivalent to the assertion that
$\circledast \circ E_i^* \cong E_i^* \circ \circledast$.
Similarly $\circledast \circ F_i^* \cong F_i^* \circ \circledast$.
The graded duality $\#$ on projectives is defined by the formula \eqref{star-rep}
as before.
\end{proof}

\subsection{Proof of Theorem~\ref{main} for finite intervals}\label{fi}
Let $I$ be finite and $\overline{\mathcal C}$ be an 
$\mathfrak{sl}_I$-tensor product categorification of some fixed type.
The hardest part of the proof of Theorem~\ref{main} in this situation
is to show that there exists
a $U_q\mathfrak{sl}_I$-tensor product categorification $\mathcal C$ 
lifting $\overline{\mathcal C}$ as in Theorem~\ref{main}(i)--(ii).
Fortunately this is already established in the literature. 
Briefly, in view of Theorem~\ref{lwmain}
and \cite[Theorem 3.12]{LW}, we may assume that $\overline{\mathcal C}$ is the category of modules over a tensor product algebra in the sense of \cite{W1}.
This algebra is naturally graded and its category of graded modules
gives us the desired graded lift $\mathcal C$.
The graded functors $F_i$ and $E_i$ and the other data of a $U_q\mathfrak{sl}_I$-categorification are constructed in \cite{W1}.
The axioms (GTP2)--(GTP3) follow from \cite[Proposition 5.5]{W1}.

\begin{Remark}\label{unfinished}\rm
There are at least two other approaches to the construction of $\mathcal C$ in the literature. It can be realized following
\cite{BGS}, \cite{Back} in terms of 
Soergel's graded lift of parabolic category $\mathcal O$;
see \cite[Corollary 9.10]{W1} for 
the equivalence of graded parabolic category $\mathcal O$ 
with the realization of $\mathcal C$ arising from the tensor product algebras. 
The appropriate 
definition of the graded functors $F_i$ and $E_i$ on graded parabolic category $\mathcal O$ is recorded in
\cite{FKS}, but a direct proof of the axiom (GTP3) via this approach is still missing in the literature.
Also Hu and Mathas \cite{HM} have given another construction of
$\mathcal C$ in terms of their version of quiver Schur algebras; these algebras
are graded Morita equivalent to the tensor product algebras by e.g. \cite[Theorem 6.17]{HM}.
However again this is not sufficient by itself for our purposes
as Hu and Mathas do not 
discuss the graded categorical actions.
\end{Remark}

This proves Theorem~\ref{main}(i), and for this particular choice of $\mathcal C$
it also establishes the existence of the graded lifts $L(\lambda)$ in (ii) 
making $\mathcal C$ into a $U_q\mathfrak{sl}_I$-tensor product categorification.
The uniqueness of the graded lifts in (ii) follows
by passing to the graded 
Grothendieck group $[\mathcal C^\Delta]_q$ then applying the following
elementary combinatorial lemma.

\begin{Lemma}
Suppose that we are given integers $\{n_\lambda\:|\:\lambda\in\Lambda\}$
such that the 
map $\bigwedge^{\bn,\bc} V_I \rightarrow \bigwedge^{\bn,\bc} V_I,
v_\lambda \mapsto q^{n_\lambda} v_\lambda$ is a $U_q\mathfrak{sl}_I$-module homomorphism.
Then all the integers $n_\lambda$ are equal.
\end{Lemma}

\begin{proof}
Exercise.
\end{proof}

To finish the proof,
let $\mathcal C$ be some fixed 
$U_q\mathfrak{sl}_I$-tensor product categorification lifting
$\overline{\mathcal C}$ in the sense of Theorem~\ref{main}(i)--(ii).
Let 
$\kappa := \kappa_{I;\bn,\bc}$ and
\begin{equation}
T = \bigoplus_{d \geq 0} T_d := \bigoplus_{d \geq 0} F^d L(\kappa) \in \mathcal C,
\end{equation}
which is a graded lift of 
$\overline{T} = \bigoplus_{d \geq 0} \overline{T}_d 
= \bigoplus_{d \geq 0} \overline{F}^d \overline{L}(\kappa) \in \overline{\mathcal C}$.
We can identify $\End_{\widehat{\mathcal C}}(T_d)$ with $\End_{\overline{\mathcal C}}(\overline{T}_d)$. Hence, 
applying the first part of Theorem~\ref{dcp},
we can identify
$\End_{\widehat{\mathcal C}}(T) = \End_{\overline{\mathcal C}}(T)$
with the graded algebra
\begin{equation}
H = \bigoplus_{d \geq 0} H_d := \bigoplus_{d \geq 0} QH_{I,d}^{|\kappa|}.
\end{equation}
Note by the definition from Lemma~\ref{backwards}(iv) that the actions of
 $1_\bi, \xi_j, \tau_k \in H_d$ on $T_d$ obtained in this way agree with the endomorphisms
induced by the action of $QH_{I,d}$ on $\widehat{F}^d$.
Now we are going to exploit the graded functor
\begin{align}\label{newu}
\V &:= \Hom_{\widehat{\mathcal C}}(T,-):
\mathcal{C} \rightarrow\boldrmod{H}.
\end{align}
This is a graded lift of $\overline{\V} := \Hom_{\overline{\mathcal C}}(\overline{T},-):
\overline{\mathcal{C}} \rightarrow\rmod{H}$, i.e. the following
diagram of functors commutes up to equivalence:
\begin{equation}\label{newc}
\begin{CD}
\widehat{\mathcal C} &@>\widehat{\V}>>&\widehat{\boldrmod H}\\
@V\nu VV&&@VV\nu V\\
\overline{\mathcal C}&@>>\overline{\V}>&\rmod H.
\end{CD}
\end{equation}
The bottom functor is fully faithful on projectives by
the second part of Theorem~\ref{dcp}, as are the vertical functors,
hence so is the top functor.

The category $\boldrmod{H}$ has the structure of a
$U_q \mathfrak{sl}_I$-categorification, which lifts the categorical $\mathfrak{sl}_I$-action on $\rmod{H}$
in the sense of Lemma~\ref{backwards}.
On $\boldrmod{H_d}$ the functor $F_i$ is the graded induction functor 
$- \otimes_{H_d} 1_{d;i} H_{d+1}$,
where $1_{d;i} := \sum_{\bi \in I^{d+1}, i_{d+1}=i} 1_\bi$.
Its 
natural transformations $\xi$ and $\tau$ defined exactly as explained
in the paragraph after (\ref{Fir}) are automatically homogeneous of the right degree.
The right adjoint $F_i^*$ is the restriction functor
defined on $M \in \boldrmod{H_{d+1}}$ by right multiplication by $1_{d;i}$.
Then $E_i$ is defined so that $F_i^* = Q E_i K_i$,
where $$
K_i M = \bigoplus_{\bj \in I^d}
Q^{(|\kappa|-\alpha_{j_1}-\cdots-\alpha_{j_d})\cdot\alpha_i}
M_\bj
$$
for $M \in \boldrmod{H_d}$.
The graded functor $\V:\mathcal C \rightarrow \boldrmod{H}$
is strongly equivariant in the graded sense; this follows
because the natural transformation $\zeta$ constructed in the proof of
Lemma~\ref{seq} is a graded natural transformation.

For $\lambda\in\Lambda$ we let
\begin{align}\label{gradedy}
Y(\lambda) &:= \V P(\lambda)
\in \boldrmod{H},\\
\overline{Y}(\lambda) &:= \overline{\V} \,\overline{P}(\lambda)
\in \rmod{H}.\label{ungradedy}
\end{align}
Thus $Y(\lambda)$ is a graded lift of
the indecomposable module $\overline{Y}(\lambda)$.
Letting $A$ be the graded algebra from (\ref{gradedA}), 
the functor $\widehat{\V}$ defines a canonical 
isomorphism of graded algebras
$$
A \cong \bigoplus_{\lambda,\mu \in \Lambda} \Hom_{H}(Y(\lambda), Y(\mu)).
$$
From (\ref{stupid2}) we get a canonical graded equivalence 
$\mathbb{H}:\mathcal{C} \stackrel{\sim}{\rightarrow} \boldrmod A$
fitting into a commuting square analogous to (\ref{ged}).
Just like we did in the ungraded case, we lift the 
$U_q\mathfrak{sl}_I$-categorification structure on $\mathcal{C}$ to $\boldrmod{A}$
so that $\mathbb{H}$ becomes a strongly equivariant
graded equivalence.
The endofunctor $F_i:\boldrmod{A}\rightarrow\boldrmod{A}$ is given by tensoring over $A$ with the 
graded $(A,A)$-bimodule 
$$
B_i := \bigoplus_{\lambda,\mu\in\Lambda}\Hom_{\mathcal C}(P(\lambda),F_i P(\mu)).
$$
As in (\ref{D1})--(\ref{E1}) 
we then apply $\widehat{\V}$ and the strong equivariance of $\V$
to obtain isomorphisms
\begin{align}\label{B3}
B_i &\cong 
\bigoplus_{\lambda,\mu \in \Lambda}
\Hom_H(Y(\lambda), F_i Y(\mu)),
\\
B_i \otimes_A B_j &\cong
\bigoplus_{\lambda,\mu \in \Lambda}
\Hom_H(Y(\lambda), F_j F_i Y(\mu)).\label{B4}
\end{align}
Now the homogeneous natural transformations $\xi$ and $\tau$ come from the homogeneous 
bimodule 
homomorphisms $\xi:B_i \rightarrow B_i$ and
$\tau:B_i \otimes_A B_j \rightarrow B_j \otimes B_i$ 
induced by $\xi$ and $\tau$ at the level of $\boldrmod{H}$.

Suppose that $\mathcal C'$ is another graded lift of $\overline{\mathcal C}$
equipped with graded endofunctors $F'_i, E_i', K'_i, {K_i'}^{-1}$,
and an adjunction and homogeneous natural transformations $\xi', \tau'$ just like in Theorem~\ref{main}(i).
Fix a (unique up to homogeneous isomorphism) graded lift $L'(\kappa)
\in \mathcal C'$ of $\overline{L}(\kappa)$. Then we repeat
the above definitions
to get a graded lift $T' \in \mathcal C'$ of $\overline{T}$
such that $\End_{\widehat{\mathcal C}'}(T') = H$, and
a graded functor 
$\V':\mathcal C' \rightarrow \boldrmod H$ 
that is fully faithful on projectives and fits into another commuting square like (\ref{newc}).
Next we must make a coherent choice of graded lifts $L'(\lambda) \in \mathcal C'$ of the other irreducibles $\overline{L}(\lambda) \in \overline{\mathcal C}$,
thereby making $\mathcal C'$ into a graded highest weight category; cf. 
Lemma~\ref{tol}.
Equivalently
we choose graded lifts $P'(\lambda)$ of the indecomposable projectives according to the following lemma.

\begin{Lemma}\label{de}
For each $\lambda \in \Lambda$ there exists a (unique up to isomorphism)
graded lift $P'(\lambda) \in \mathcal C'$ of $\overline{P}(\lambda)$ such that
$Y'(\lambda) := \V' P'(\lambda)$ is isomorphic to $Y(\lambda)$ as a graded $H$-module.
\end{Lemma}

\begin{proof}
Let $P' \in \mathcal C'$
be some arbitrary choice of graded lift of $\overline{P}(\lambda)$.
Then $Y' := \V' P'\in\boldrmod H$ is a graded lift of $\overline{Y}(\lambda)$, as is $Y(\lambda)$. As $\overline{Y}(\lambda)$ is indecomposable there exists a unique $n \in \Z$
such that $Q^n Y' \cong Y(\lambda)$.
We then {define} $P'(\lambda)$ to be $Q^n P'$.
\end{proof}

For $Y'(\lambda)$ as in Lemma~\ref{de},
we then introduce the graded algebra $A'$ and graded $(A',A')$-bimodules
$B_i'$ as above so that
\begin{align*}
A' &\cong \bigoplus_{\lambda,\mu\in\Lambda}\Hom_{H}(Y'(\lambda), Y'(\mu)),\\
B_i' &\cong \bigoplus_{\lambda,\mu\in\Lambda}\Hom_{H}(Y'(\lambda), F_i Y'(\mu)),\\
B_i'\otimes_{A'} B_j' &\cong
\bigoplus_{\lambda,\mu\in\Lambda}\Hom_{H}(Y'(\lambda), F_j F_i Y'(\mu)).
\end{align*}
We make $\boldrmod{A'}$ into a $U_q\mathfrak{sl}_I$-categorification
with $F_i := -\otimes_{A'} B_i'$
just like before,
so that there is a strongly equivariant graded equivalence
$\mathbb{H}':\mathcal C' \stackrel{\sim}{\rightarrow} \boldrmod{A'}$
fitting into another analog of (\ref{ged}).
Next
choose
graded $H$-module isomorphisms $Y(\lambda) \cong Y'(\lambda)$ 
for each $\lambda \in \Lambda$. These induce graded algebra isomorphisms
between $A$ and $A'$, 
and 
isomorphisms between the graded bimodules $B_i$ and $B_i'$ with
appropriate equivariance properties.
Hence we get
a strongly equivariant graded isomorphism of categories
$\boldrmod{A}\stackrel{\sim}{\rightarrow}\boldrmod{A'}$ 
sending each $1_\lambda A$ to $1_\lambda A'$.
Since $\mathcal C$ is strongly equivariantly equivalent to $\boldrmod{A}$
and $\mathcal C'$ is strongly equivariantly equivalent to $\boldrmod{A'}$, we deduce 
that there is a strongly equivariant graded equivalence $\mathbb{G}:\mathcal C \stackrel{\sim}{\rightarrow} \mathcal C'$ with $\nu' \circ \mathbb{G} \cong \nu$
and $\mathbb{G}L(\lambda) \cong L'(\lambda)$ for each $\lambda \in \Lambda$.

The existence of $\mathbb{G}$ implies that $\mathcal C'$ also satisfies
(GTP2)--(GTP3), so that it is indeed a tensor product categorification, completing
the proof of Theorem~\ref{main}(ii) for an arbitrary
choice of $\mathcal C$. We have also essentially proved (iii). More
precisely, under the assumptions of (iii), we have proved that there
exists
a strongly equivariant graded equivalence $\mathbb{G}:\mathcal C
\rightarrow \mathcal C'$ with $\nu' \circ \mathbb{G} \cong \nu$.
By the uniqueness in (ii) there exists $n \in \Z$ such that $L'(\lambda) \cong {Q'}^n
\mathbb{G} L(\lambda)$
for all $\lambda$. Now replace $\mathbb{G}$ by ${Q'}^{n} \circ \mathbb{G}$.
This completes the proof of Theorem~\ref{main} for finite $I$.

The uniqueness just established implies that
the {\em graded Young module} $Y(\lambda)$ from (\ref{gradedy})
does not depend on the particular choice of $\mathcal C$ (up to isomorphism in $\boldrmod{H}$). The following lemma gives a slightly different characterization
of this important module.
It also gives a first glimpse of the significance of the combinatorial statistic
{\em defect} introduced way back in (\ref{defect}).
Note in the proof of the lemma 
we use some definitions from $\S$\ref{cc} below, but none of the intermediate theorems.
(In fact more is true here: $Y(\lambda)$ is self-dual with respect to 
a natural graded duality $\#$
on $\boldrmod{H}$.)

\begin{Lemma}\label{shift}
For $\lambda \in \Lambda$,
the graded $H$-module
$Y(\lambda)$ from (\ref{gradedy})
is
the unique (up to isomorphism) graded lift of 
$\overline{Y}(\lambda)$ such that $Q^{-\operatorname{def}(\lambda)} Y(\lambda)$ is self-dual as a graded vector space.
In fact each word space
$Q^{-\operatorname{def}(\lambda)}Y(\lambda) 1_\bi$
is self-dual as a graded vector space.
\end{Lemma}

\begin{proof}
We just check that each $Q^{-\operatorname{def}(\lambda)}Y(\lambda)1_\bi$ is graded-self-dual.
Let $E_i^! := Q^{-1} F_iK_i^{-1}$, which is left adjoint to $E_i$.
We may assume that
$\bi = (i_1,\dots,i_d) \in I^d$ satisfies
$\alpha_{i_1}+\cdots+\alpha_{i_d} = |\kappa|-|\lambda|$,
since otherwise $Y(\lambda) 1_\bi$ is zero.
Then a little calculation shows that
$$
Q^{-\operatorname{def}(\lambda)}Y(\lambda)1_\bi \cong
\Hom_{\widehat{\mathcal C}}(Q^{\operatorname{def}(\lambda)}F_{\bi} L(\kappa), P(\lambda))
\cong
\Hom_{\widehat{\mathcal C}}(E^!_{i_d} \cdots E^!_{i_1} L(\kappa), P(\lambda)).
$$
Hence we are reduced to showing that
$$
\qdim 
\Hom_{\widehat{\mathcal C}}(E^!_{i_d} \cdots E^!_{i_1} L(\kappa), P(\lambda))
=
\qdim 
\Hom_{\widehat{\mathcal C}}(L(\kappa), E_{i_1} \cdots E_{i_d} P(\lambda))
$$
is bar-invariant.
Using the notation from (\ref{formies})--(\ref{b}) below,
setting $e := e_{i_1}\cdots e_{i_d}$ for short,
we have that
$$
\qdim 
\Hom_{\widehat{\mathcal C}}(L(\kappa), E_{i_1} \cdots E_{i_d} P(\lambda))
= \overline{(v_\kappa, \psi^*(e b_\lambda))}
=
(\psi(v_\kappa), e b_\lambda)) = (v_\kappa, e b_\lamdba).
$$
By the symmetry of this form, this equals
$$
(e b_\lambda, v_\kappa) = (\psi(e b_\lambda), v_\kappa)
= \overline{(e b_\lambda, \psi^*(v_\kappa))}
= \overline{(e b_\lambda, v_\kappa)}.
$$
So it is bar-invariant.
\end{proof}

\begin{Corollary}\label{shift2}
If $Y'(\lambda)$ is any graded lift of $\overline{Y}(\lambda)$
such that $Q^{-\operatorname{def}(\lambda)}Y'(\lambda) 1_\bi$ is a non-zero self-dual graded vector space
for some word $\bi$ then $Y'(\lambda) \cong Y(\lambda)$.
\end{Corollary}

\subsection{Proof of Theorem~\ref{main} for infinite intervals}
Now the interval $I$ is infinite.
Fixing a type $(\bn,\bc)$ we set $\Lambda := \Lambda_{I;\bn,\bc}$
and let
$\overline{\mathcal C}$ be some given $\mathfrak{sl}_I$-tensor product
categorification of type $(\bn,\bc)$.
Apart from adding bars to almost everything in sight, we are going to
adopt all of the notation
from Section~\ref{sstable}.
Thus we have chosen subintervals $I_1 \subset I_2 \subset
\cdots\subset I$, leading to subquotients $\overline{\mathcal C}_r$ of
$\overline{\mathcal C}$ with
weight posets 
$\Lambda_1 \subset \Lambda_2\subset\cdots\subset\Lambda$.
Also $\kappa^r := \kappa_{I_r;\bn,\bc} \in \Lambda_r$ and
we have the tensor spaces $\overline{T}^r=\bigoplus_{d \geq 0} \overline{T}^r_d
\in \overline{\mathcal{C}}$ defined like in (\ref{tr}).
Each $\End_{\overline{\mathcal C}}(\overline{T}^r)$ is identified
with the algebra $H^r$ from (\ref{hr}).
We have the Abelian category $\rmod{H}$ of stable modules 
from Definition~\ref{rmod2},
and
the exact functor $\overline{\V}:\overline{\mathcal
  C}\rightarrow \rmod{H}$ from Theorem~\ref{abelian1}.
This functor is fully
faithful on projectives by Theorem~\ref{ffp}. We have the
objects
$\overline{Y}(\lambda)
= (\overline{Y}^1(\lambda)\rightarrow \overline{Y}^2(\lambda)\rightarrow\cdots)
:= \overline{\V}\,\overline{P}(\lambda)\in \rmod{H}
$ 
as in (\ref{young3}).
We have the
functor $\overline{\pr}_r:\rmod{H} \rightarrow \rmod{H^r}$
and its left adjoint $\overline{\pr}_r^!$ from (\ref{mmm})--(\ref{ungradedpr}).
Thus $\overline{\pr}_r \circ \overline{\V}$ is equal to
$\overline{\V^r} := 
\Hom_{\overline{\mathcal C}}(\overline{T}^r, -):\overline{\mathcal C} \rightarrow \rmod{H^r}.$

The idempotents $e^r_d$ and the isomorphisms
$\phi^r_d$ from Lemma~\ref{tl} are evidently homogeneous.
Thus 
$\phi^r:H^{r} \stackrel{\sim}{\rightarrow} e^r H^{r+1} e^r$ is an
isomorphism of graded algebras. 
So it makes sense to introduce a graded version $\boldrmod{H}$ of
the category $\rmod{H}$ of stable modules from Definition~\ref{rmod2}.
Its objects are diagrams
$$
\begin{CD}
M= (M^1 &@>\iota^1>>&M^2&@>\iota^2>> M^3&@>\iota^3>> \cdots)\\
\end{CD}
$$
such that $M^r \in \boldrmod{H^r}$,
$\iota^r:M^r \stackrel{\sim}{\rightarrow} M^{r+1} e^r$ 
is an isomorphism of graded $H^r$-modules for each $r \geq 1$,
and the maps (\ref{topmaps}) are isomorphisms for $r \gg 1$.
Morphisms are tuples $(f^r)_{r\geq 1}$ such that each $f^r$ is a morphism
in $\boldrmod{H^r}$ making the analogous diagram to (\ref{morphism}) commute.
The category $\boldrmod{H}$ is graded with $Q$ being the obvious grading shift functor.
Let $\nu:\widehat{\boldrmod{H}} \rightarrow \rmod{H}$
be the obvious functor that forgets the grading on each $M^r$.
Let $\pr_r:\boldrmod{H} \rightarrow \boldrmod{H^r}$
be the $r$th projection, which is clearly graded. It has an obvious left adjoint
$\pr_r^!$ which is a graded lift of
$\overline{\pr}_r^!$; it is defined in exactly the same way as before just working in the graded module categories.

\begin{Lemma}
The category $\boldrmod{H}$ is a graded lift of $\rmod{H}$
in the sense of Definition~\ref{liftdef}.
\end{Lemma}

\begin{proof}
We need to check that $\boldrmod{H}$ is an Abelian category
and that $\nu$ is dense on projectives.

Let $f:M \rightarrow N$ be a morphism in $\boldrmod{H}$.
Let $\overline{f}:\overline{M} \rightarrow \overline{N}$
be the corresponding morphism in $\rmod{H}$ obtained by applying the forgetful functor $\nu$.
We know already that $\rmod{H}$ is an Abelian category.
To show that $\boldrmod{H}$ is Abelian too
it suffices to check that
$\ker \overline{f} \rightarrow \overline{M}$ 
and $\overline{N} \rightarrow 
\operatorname{coker} \overline{f}$ are gradable.
Pick $r$ so that all of $\overline{M},
\overline{N}, \ker \overline{f}$ and $\operatorname{coker} \overline{f}$
are $r$-stable.
Then we have simply that $(\ker \overline{f} \rightarrow \overline{M})
 \cong \overline{\pr}_r^! (\ker \overline{f}^r\rightarrow \overline{M}^r)$.
Hence it is graded by $\pr_r^! (\ker f^r \rightarrow M^r)$.
Similarly the cokernel is gradable.

To see that $\nu$ is dense on projectives, note that
every projective in $\rmod{H}$ is a direct sum of summands of
$\overline{\V} \,\overline{T}^r\cong \overline{\pr}_r^! H^r $, 
where $H^r$ is the regular right $H^r$-module.
Each $\overline{\pr}_r^! H^r$ 
admits the graded lift $\pr_r^! H^r$, which is projective in $\boldrmod{H}$
by properties of adjoints.
\end{proof}

Hence by Lemma~\ref{acycliclemma} 
we see that $\boldrmod{H}$ is a graded Schurian category.
It also has 
a structure of $U_q \mathfrak{sl}_I$-categorification lifting
the categorical action on $\rmod{H}$ 
as in Lemma~\ref{backwards}.
The graded lifts $F_i$ and $E_i$ on $\boldrmod{H}$
are defined in exactly the same way as the original functors
$\overline{F}_i$ and $\overline{E}_i$ on $\rmod{H}$ 
were defined in $\S$\ref{secca}, just replacing each
$\overline{F}_i^r$ and $\overline{E}_i^r$ with its graded version
as defined in the previous subsection.
It is worth noting that 
$K_i:\boldrmod{H} \rightarrow \boldrmod{H}$
is defined on $M = (M^1 \stackrel{\iota^1}{\rightarrow} M^2 \stackrel{\iota^2}{\rightarrow} \cdots)$
so that 
it is the 
degree shift $Q^{(|\kappa^r|-\alpha_{j_1}-\cdots-\alpha_{j_d})\cdot \alpha_i}$ on
$1_\bj M^r$
for each $r \geq 1$ and $\bj \in I_r^d$.
One needs to check here that the maps 
$\iota^r$ remain homogeneous of degree zero after these shifts are performed;
this follows because
$|\kappa^r| = |\kappa^{r+1}| - \alpha_{i_1} - \cdots - \alpha_{i_{d_r}}$
where $\bi = i_1 \cdots i_{d_r}$ is the word appearing in the definition of
$\phi^r$ from Lemma~\ref{tl}.

\begin{Lemma}\label{key}
Suppose we are given 
$\overline{M} \in \overline{\mathcal C}$ and $r \geq 1$ such that
$\overline{\V}\,\overline{M}$ is $r$-stable.
Then $\overline{\V}\,\overline{M}$ is a gradable object of $\rmod{H}$
if and only if $\overline{\V}^r \overline{M}$ is a gradable object of 
$\rmod{H^r}$.
\end{Lemma}

\begin{proof}
The forward direction is clear. For the converse,
let $M^r$ be a graded lift of $\overline{M}^r := \Hom_{\overline{\mathcal C}}(\overline{T}^r, \overline{M})$.
Since $\overline{\V}\, \overline{M}$ is $r$-stable
it is isomorphic to $\overline{\pr}_r^! \overline{M}^r \in \rmod{H}$.
Hence $\pr_r^! M^r \in \boldrmod{H}$ gives the desired graded lift.
\end{proof}

Lemma~\ref{key} applies in particular to the objects $\overline{Y}(\lambda)$,
since we know already that 
$\overline{Y}^r(\lambda)$ is gradable
for sufficiently large $r$ by (\ref{gradedy})--(\ref{ungradedy}).
We need one more piece of book-keeping in order to make
a canonical choice of such a graded lift.

\begin{Lemma}\label{yuk}
Fix $r \geq 1$. Let $\bi$ and $p_1,\dots,p_a$ be as in Lemma~\ref{tl}.
Set
\begin{equation}\label{sigmar}
\sigma_r :=\textstyle \frac{1}{2} p_1 + \cdots + \frac{1}{2} p_a.
\end{equation}
Then
$\qdim M^r 1_\bj = q^{\sigma_r} \qdim M^{r+1} 1_{\bi \bj} / [p_1]!\cdots [p_a]!$
for each $M \in \boldrmod{H}$.
\end{Lemma}

\begin{proof}
This follows from the explicit form of the idempotent $e^r_d$
constructed in the proof of Lemma~\ref{tl}.
The essential point is that $b_m \in 1_{i^m} QH_{I_{r+1},m} 1_{i^m}$ 
has the property for any finite dimensional graded $QH_{I_{r+1},m}$-module $M$
that
$\qdim M b_m = q^{\frac{1}{2}m(m-1)}\qdim M 1_{i^m} / [m]!$.
To see this note that
$1_{i^m} QH_{I_{r+1},m} 1_{i^m}$ is a nil-Hecke algebra.
So it has a unique irreducible module $\overline{L}$, a graded lift $L$ of which has graded dimension $[m]!$.
Since 
$b_m = \tau_{w_0} \xi_1^{m-1}\xi_2^{m-2}\cdots \xi_{m-1}$ 
we deduce by degree considerations that $\qdim L b_m = q^{\frac{1}{2}m(m-1)}$.
\end{proof}

Then for $\sigma_r$ as in (\ref{sigmar}) we define
\begin{equation}
\Sigma_r := \sigma_1+\cdots+\sigma_{r-1}\label{Sigma}.
\end{equation}
Recall also the definition of $\operatorname{def}(\lambda)$
for infinite intervals from Lemma~\ref{defectL}.

\begin{Theorem}\label{lifts}
For $\lambda \in \Lambda$, there exists 
a unique (up to isomorphism)
graded lift
$
Y(\lambda) = (Y^1(\lambda)\rightarrow Y^2(\lambda)\rightarrow\cdots)
\in \boldrmod{H}
$ 
of $\overline{Y}(\lambda)$
such that $Q^{\Sigma_r - \operatorname{def}(\lambda)}Y^r(\lambda)$ 
is self-dual as a graded vector space for each $r \geq 1$.
In fact if $Y'(\lambda)$ is any graded lift of $\overline{Y}(\lambda)$
such that
$Q^{\Sigma_s - \operatorname{def}(\lambda)}{Y'}^s(\lambda)$
is non-zero and self-dual as a graded vector space for some $s \geq 1$,
then we have that $Y'(\lambda) \cong Y(\lambda)$ in $\boldrmod{H}$.
\end{Theorem}

\begin{proof}
Choose $r$ so that $\overline{Y}(\lambda)$ is $r$-stable.
We already observed by Lemma~\ref{key}
that $\overline{Y}(\lambda)$ is gradable.
In view of Lemma~\ref{shift} we can pick a graded lift 
$Y(\lambda)$ so that each word space of
$Q^{\Sigma_r - \operatorname{def}(\lambda)}Y^r(\lambda)$ is graded-self-dual.
Lemma~\ref{yuk} then implies immediately
that $Q^{\Sigma_s - \operatorname{def}(\lambda)} Y^s(\lambda)$
is graded-self-dual for each $s < r$; this depends also on the independence
from Lemma~\ref{defectL}.
The same argument together also with Corollary~\ref{shift2}
proves the same thing for $s > r$ too.

To establish the claim about uniqueness,
let $Y'(\lambda)$ be another graded lift of $\overline{Y}(\lambda)$
such that $Q^{\Sigma_s - \operatorname{def}(\lambda)} {Y'}^s(\lambda)$ is a non-zero
self-dual graded vector space for some $s \geq 1$.
Then ${Y'}^r(\lambda)$ is a graded lift of $\overline{Y}^r(\lambda)$
and the argument just given using Lemmas~\ref{yuk}, \ref{shift} 
and Corollary~\ref{shift2} imply that ${Y'}^r(\lambda)$ is also graded-self-dual.
Hence actually ${Y'}^r(\lambda) \cong Y^r(\lambda)$.
But then the $r$-stability implies that
$Y'(\lambda) \cong \pr_r^! ({Y'}^r(\lambda)) \cong 
\pr_r^!(Y^r(\lambda)) \cong Y(\lambda)$.
\end{proof}

With Theorem~\ref{lifts} in hand it is clear how to 
construct a graded lift $\mathcal C$ of $\overline{\mathcal C}$.
We fix a choice of object $Y(\lambda) = (Y^1(\lambda)\rightarrow Y^2(\lambda)\rightarrow\cdots) \in \boldrmod{H}$ as in Theorem~\ref{lifts}
for each $\lambda\in\Lambda$,
then define
\begin{equation}\label{lastA}
A := \bigoplus_{\lambda,\mu \in \Lambda} \Hom_{H}(Y(\lambda),Y(\mu)).
\end{equation}
This is a graded lift of the basic algebra underlying the original 
category $\overline{\mathcal C}$.
Thus $\mathcal C := \boldrmod{A}$ is
a graded lift of $\overline{\mathcal C}$,
with forgetful functor
$\nu:\widehat{\mathcal C} \rightarrow \overline{\mathcal C}$ that is the composite of the
forgetful functor $\widehat{\boldrmod{A}}\rightarrow \rmod{A}$
and the adjoint equivalence $\rmod{A} \rightarrow \overline{\mathcal{C}}$
to the usual equivalence 
$\overline{\mathbb{H}}:\overline{\mathcal C}\rightarrow \rmod{A}$. 

Next we introduce a categorical $U_q \mathfrak{sl}_I$-action 
on $\mathcal C$. Let
\begin{align}\label{B5}
B_i &:= \bigoplus_{\lambda,\mu\in\Lambda} \Hom_H(Y(\lambda), F_i Y(\mu)),\phantom{\qquad\text{a so that}}\\
\!\!\!\text{so that}\qquad\qquad\qquad B_i \otimes_A B_j &\cong \bigoplus_{\lambda,\mu\in\Lambda} \Hom_H(
Y(\lambda), F_j F_i Y(\mu)).\label{B6}
\end{align}
The natural transformations $\xi$ and $\tau$ from the categorical action on
$\boldrmod{H}$ induce 
induce homogeneous bimodule homomorphisms 
$\xi:B_i \rightarrow B_i$ and $\tau:B_i \otimes_A B_j \rightarrow B_j \otimes_A B_i$
as usual.
Let $F_i:\boldrmod{A} \rightarrow \boldrmod{A}$
be the functor defined by tensoring with $B_i$, with it homogeneous 
natural transformations
$\xi$ and $\tau$ induced by the preceding bimodule endomorphisms.
This gives us a choice of graded lifts of $\overline{F}_i:\overline{\mathcal C}
\rightarrow \overline{\mathcal C}$, $\overline{\xi}\in\End(\overline{F}_i)$ and $\overline{\tau}\in\Hom(\overline{F}_j\circ\overline{F}_i,\overline{F}_i\circ\overline{F}_j)$.
For $F_i^*$ we take the canonical graded right adjoint
$\bigoplus_{\lambda \in \Lambda} \Hom_A(1_\lambda B_i, -)$ to $F_i$. 
All the other required lifts are obvious, and then Lemma~\ref{backwards} tells us that all this data makes $\mathcal C$
into a $U_q \mathfrak{sl}_I$-categorification.

\begin{Theorem}\label{done}
There exists a (unique up to isomorphism and global shift) choice of distinguished irreducible 
objects $\{L(\lambda)\:|\:\lambda \in \Lambda\}$
making the 
graded lift $\mathcal C$ of $\overline{\mathcal{C}}$ just constructed into a 
$U_q\mathfrak{sl}_I$-tensor product categorification.
\end{Theorem}

\begin{proof}
We apply the truncation construction
to define subquotients $\mathcal C_r := \mathcal C_{I_r}$
for each $r \geq 1$.
The structures defined on $\mathcal C$ make each $\mathcal C_r$
into a graded lift of $\overline{\mathcal C}_r$ in the sense of
Theorem~\ref{main}(i). Thus by Theorem~\ref{main}(ii) for the finite interval $I_r$ there exists a unique (up to isomorphism and global shift) set $\{L_r(\lambda)\:|\:\lambda \in \Lambda_r\}$ of distinguished irreducible objects in $\mathcal C_r$ making it into a $U_q\mathfrak{sl}_{I_r}$-tensor product categorification.

Now $\mathcal C_r$ is a subquotient of $\mathcal C_{r+1}$, and 
the objects $\{L_{r+1}(\lambda)\:|\:\lambda\in\Lambda_r\}$ give another choice of distinguished irreducible objects of $\mathcal C_r$ making it into a
$U_q\mathfrak{sl}_{I_r}$-tensor product categorification.
By the uniqueness in Theorem~\ref{main}(ii) this means that
$L_r(\lambda) \cong Q^n L_{r+1}(\lambda)$ for some $n \in \Z$ and all $\lambda \in \Lambda_r$.
Replacing $L_{r+1}(\lambda)$ by $Q^n L_{r+1}(\lambda)$ this shows that we may assume that
$L_r(\lambda)  \cong L_{r+1}(\lambda)$ for each $\lambda \in \Lambda_r$.
Starting at $r=1$ and proceeding recursively in this way we can ensure this is the case for all $r \geq 1$.

Then we define irreducible objects $L(\lambda)\in\mathcal C$ for $\lambda \in \Lambda$ by taking $L(\lambda)$ to be any irreducible object of $\mathcal C$
such that $L(\lambda) \cong L_r(\lambda)$ in $\mathcal C_r$
whenever $\lambda \in \Lambda_r$.
This gives the desired set of distinguished irreducible objects
$\{L(\lambda)\:|\:\lambda\in\Lambda\}$ for $\mathcal C$.
This makes $\mathcal C$ into a graded highest weight category by Lemma~\ref{tol}. 
Moreover it satisfies (GTP2)--(GTP3) because they 
hold in each subquotient $\mathcal C_r$.
\end{proof}

\begin{Remark}\rm\label{lastr}
We will see momentarily that the irreducible objects $L(\lambda)$ in Theorem~\ref{done}
can be taken to be the irreducible heads of the indecomposable projectives
$P(\lambda) := 1_\lambda A \in \boldrmod{A}$.
\end{Remark}

We have now done the hard work, proving Theorem~\ref{main}(i)
and establishing the existence of graded lifts $L(\lambda)$ as in (ii) 
for one particular choice of $\mathcal C$.
The proof of the remainder of Theorem~\ref{main}
follows the same argument as
explained for finite intervals in the previous subsection. 
So
let $\mathcal C$ be some fixed 
$U_q\mathfrak{sl}_I$-tensor product categorification lifting
$\overline{\mathcal C}$ in the sense of Theorem~\ref{main}(i)--(ii).
Let 
\begin{equation}\label{clean1}
T^r = \bigoplus_{d \geq 0} T^r_d := Q^{\Sigma_r} 
\bigoplus_{d \geq 0} F_{I_r}^d L(\kappa^r) \in \mathcal C,
\end{equation}
which is a graded lift of 
$\overline{T}^r$, hence $\End_{\widehat{\mathcal C}}(T^r) = H^r$.
The shifts $\Sigma_r$ here are as in (\ref{Sigma}),
and their appearance is explained this time by the following lemma.

\begin{Lemma}\label{tricky}
For $r \geq 1$ there is an isomorphism
$\theta^r:T^r  \stackrel{\sim}{\rightarrow}  e^r T^{r+1}$ in $\mathcal C$
such that $\theta^r \circ h = \phi^r(h) \circ \theta^r$ for each $h \in H^r$.
\end{Lemma}

\begin{proof}
This follows from the proof of Lemma~\ref{tl}.
The only new observation in the graded setting is that
the quantum divided power functor $F_i^{(m)}$ is defined from
$F_i^{(m)} := Q^{-\frac{1}{2}m(m-1)} b_m F_i^m$, where $b_m \in 1_{i^m} QH_{I_{r+1},m} 1_{i^m}$
is the distinguished idempotent from \cite[Lemma 4.1]{R}.
This implies that $F_{I_r}^d L(\kappa^r) \cong Q^{-\sigma_r} 
e^r F_{I_{r+1}}^{d_r+d} L(\kappa^{r+1})$.
\end{proof}

Then we can introduce the graded functors
\begin{align}\label{clean2}
\V^r:\mathcal{C} \rightarrow\boldrmod{H^r},
\qquad
\V:
\mathcal{C} \rightarrow\boldrmod{H},
\end{align}
exactly like (\ref{ur})--(\ref{atlast}) but working in the graded categories.
Thus $\V^r = \pr_r \circ \V = \Hom_{\widehat{\mathcal C}}(T^r, -)$ is a graded lift of $\overline{\V}^r$, and
$\V$ is a graded lift of $\overline{\V}$, i.e. the 
analog of the diagram (\ref{newc}) commutes up to equivalence.
Applying Theorem~\ref{ffp} we deduce that $\V$ is fully faithful on projectives.
Moreover the proof of Theorem~\ref{strongeq} 
works also in the graded setup,
so $\V$ is a strongly equivariant graded functor.

For $\lambda\in\Lambda$ we let
\begin{equation}\label{clean3}
Y(\lambda) = (Y^1(\lambda)\rightarrow Y^2(\lambda)\rightarrow\cdots)
:= \V P(\lambda)
\in \boldrmod{H}.
\end{equation}
Thus $Y(\lambda)$ is a graded lift of $\overline{Y}(\lambda)$.
For any sufficiently large $r$,
Lemmas~\ref{defectL} and \ref{shift} imply
that $Q^{\Sigma_r-\operatorname{def}(\lambda)} 
Y^r(\lambda)$ is self-dual as a graded vector space.
Hence the uniqueness assertion of Theorem~\ref{lifts} implies that 
the object $Y(\lambda)$ just defined is isomorphic to the object $Y(\lambda)$
from that theorem. 
Hence the graded algebra from (\ref{gradedA}) is isomorphic to the
algebra (\ref{lastA}). As usual there is a strongly equivariant graded equivalence
between $\mathcal C$ and $\boldrmod{A}$ with respect to the graded categorical action 
on $\boldrmod{A}$ arising from the bimodules (\ref{B3})--(\ref{B4}); these bimodules are
equivariantly isomorphic to the ones appearing in (\ref{B5})--(\ref{B6}).
In the present situation, it is also clear that the 
projective indecomposable object $P(\lambda) \in \mathcal C$ 
corresponds under this equivalence to the graded $A$-module
$1_\lambda A$, so that this justifies the claim made in Remark~\ref{lastr}.

To complete the proof of Theorem~\ref{main} we take another
graded lift $\mathcal C'$ of $\overline{\mathcal C}$
as in Theorem~\ref{main}(i).
We fix choices of graded lifts $L'(\kappa^r) \in \mathcal C'$
of each $\overline{L}(\kappa^r)$ in $\overline{\mathcal{C}}$ such that
${F'_{s+\eps 1}}^{\!\!\!\!\!\!(p_1)}\cdots {F'_{s+\eps a}}^{\!\!\!\!\!\!(p_a)} L'(\kappa^{r+1}) \cong L'(\kappa^r)$
for each $r \geq 1$, notation as in Lemma~\ref{tl}.
Then we define ${T'}^r$ just like in (\ref{clean1}). 
The careful choices just made ensure that
the analog of Lemma~\ref{tricky} holds for $\mathcal C'$. So we can continue to 
define $\V':\mathcal C' \rightarrow \boldrmod{H}$ like in (\ref{clean2}),
which is fully faithful on projectives. Using this repeat the rest of the above constructions 
to get an algebra $A'$ as in (\ref{clean3}), bimodules $B_i'$, etc
\dots. Then run the logic from the previous subsection again
to finish the proof.

\subsection{Koszulity}
We say that a graded Schurian category $\mathcal C$ is 
{\em mixed} 
if the graded algebra $A$ defined by (\ref{gradedA}) is strictly 
positively graded, i.e.
$A = \bigoplus_{n \geq 0} A_n$
with $A_0 = \bigoplus_{\lambda\in\Lambda} \K 1_\lambda$. 
(The terminology ``mixed'' here goes back to \cite{BGS}; it is equivalent to 
property that $\Ext^1_{\mathcal C}(L, L') = 0$ for irreducibles $L, L'$
with $\operatorname{wt}(L) \leq \operatorname{wt}(L')$, where the
{\em weight} of $L \cong Q^n L(\lambda)$ is defined to be $-n$.)
A graded highest weight category is mixed 
if and only if both of the following holding for all weights $\lambda$ and $\mu$, respectively:
\begin{align}\label{pos1}
[P(\lambda)] &= [\Delta(\lambda)] +
\text{(a $q \N[q]$-linear combination of $[\Delta(\mu)]$ for $\mu > \lambda$),}\\\label{pos2}
[\Delta(\mu)] &= [L(\mu)] +
\text{(a $q \N[q]$-linear combination of $[L(\lambda)]$ for $\lambda < \mu$).}
\end{align}
If $\mathcal C$ possesses a graded duality fixing the distinguished 
irreducible objects, then we can use graded BGG reciprocity to deduce further that (\ref{pos1}) and (\ref{pos2})  are equivalent.

Now assume instead just
that $\mathcal C$ is a graded highest weight category with a graded duality.
Then $\mathcal C$ is {\em standard Koszul}
if the minimal projective resolution 
$$
\cdots\rightarrow P^1(\lambda) \rightarrow P^0(\lambda)
\rightarrow \Delta(\lambda) \rightarrow 0
$$ 
of each standard object $\Delta(\lambda)$
 is linear, i.e. $P^0(\lambda) \cong P(\lambda)$ and
for each $n \geq 1$ the object $P^n(\lambda)$ is a direct sum of objects $Q^n P(\mu)$
for $\mu > \lambda$.
This 
implies that (\ref{pos1}) holds for all $\lambda$, 
hence that $\mathcal C$ is mixed.

In fact standard Koszul implies Koszul in the usual sense, i.e. 
$\mathcal C$ is mixed and
$\Ext^n_{\mathcal C}(L(\lambda),L(\mu))_m = 0$ unless $m+n=0$ 
for any $\lambda,\mu \in \Lambda$.
This is proved for finite weight posets in
\cite[Theorem 1]{ADL}; the general case follows on passing to sufficiently large subquotients.

\begin{Theorem}\label{kosmain}
Let $\mathcal C$ be a 
$U_q\mathfrak{sl}_I$-tensor product categorification of any type and
for any interval $I$.
Then $\mathcal C$ is standard Koszul.
\end{Theorem}

\begin{proof}
When $I$ is finite this is already known. In fact there are a couple
of proofs available in the literature depending on which realization of $\mathcal C$ is
adopted; see \cite{W2} and \cite{HM}. Both proofs depend 
ultimately on
the known standard Koszulity of graded parabolic category $\mathcal O$
which follows from \cite{BGS}, \cite{Back}
and \cite[Corollary 3.8]{ADL}.
(If the gap mentioned in Remark~\ref{unfinished}
related to checking directly that graded parabolic
category $\mathcal O$ satisfies (GTP3) can be filled,
then Theorem~\ref{main} would allow \cite{W2} or \cite{HM} to be bypassed entirely here.)

For the general case,
take $\lambda \in \Lambda$ and consider a minimal
projective resolution 
$$
\cdots\rightarrow P^2(\lambda)\rightarrow P^1(\lambda) \rightarrow P^0(\lambda)
\rightarrow \Delta(\lambda)\rightarrow 0
$$
of $\Delta(\lambda)$.
Suppose for some $n$ that $P^n(\lambda)$ is not a direct sum of $Q^n P(\mu)$.
Pick a finite subinterval $J \subset I$ such that all composition factors of 
$P^m(\lambda)$ belong to $\Lambda_J$ for all $m \leq n$.
Then $P^n(\lambda)\rightarrow\cdots\rightarrow P^0(\lambda)\rightarrow \Delta(\lambda)\rightarrow 0$ is the start of a non-linear minimal projective resolution for $\Delta(\lambda)$ in the subquotient $\mathcal C_J$ too.
This contradicts the standard Koszulity of $\mathcal C_J$.
\end{proof}

\begin{Corollary}\label{pg}
All $U_q\mathfrak{sl}_I$-tensor product categorifications in the sense of 
Definition~\ref{qtpcdef}
are mixed.
\end{Corollary}

\begin{Remark}\rm
One very interesting question is whether any reasonable description of
the Koszul dual of $\mathcal{C}$ is possible.  In the finite case, the
usual singular-parabolic duality of \cite{Back} shows that the Koszul
dual of a weight space in an $\mathfrak{sl}_n$-tensor product
categorification of level $l$ is a weight space
in an $\mathfrak{sl}_l$-tensor product
categorification of level $n$; this is a
categorical version of skew Howe duality. It is not clear how
much of this structure survives in the infinite case.
\end{Remark}

\subsection{Kazhdan-Lusztig polynomials}\label{cc}
Let $\mathcal C$ be a $U_q\mathfrak{sl}_I$-tensor product
categorification of type $(\bn,\bc)$ and set $\Lambda := \Lambda_{I;\bn,\bc}$.
We conclude by 
proving the graded analog of the Kazhdan-Lusztig conjecture.
This
allows the polynomials
\begin{align}\label{decomp}
d_{\lambda,\mu}(q) &:= [\Delta(\mu):L(\lambda)]_q,\\
p_{\lambda,\mu}(q) &:= \sum_{n \geq 0} \dim \Ext^n_{\widehat{\mathcal
  C}}(\Delta(\lambda), L(\mu))q^n\label{decomp2}
\end{align}
to be computed in principle in terms of parabolic Kazhdan-Lusztig
polynomials in finite type $A$.
Note by the general theory of standard Koszul categories
that the two families of polynomials here are closely related:
the matrices $(d_{\lambda,\mu}(q))_{\lambda,\mu\in\Lambda}$
and $(p_{\lambda,\mu}(-q))_{\lambda,\mu \in \Lambda}$ are inverse to each other.

To start with, we have the graded Grothendieck groups $[\mathcal C]_q
\hookrightarrow [\mathcal C^\Delta]_q \hookrightarrow [\mathcal C]_q^*$. 
The functors $F_i, E_i$ and $K_i$ induce endomorphisms $f_i, e_i$ and $k_i$ 
making these Grothendieck groups into $U_q \mathfrak{sl}_I$-modules
and the inclusions above are 
module homomorphisms.
Recall also that $[\mathcal C^\Delta]_q$ is identified
with the $U_q \mathfrak{sl}_I$-module
$\bigwedge^{\bn,\bc} V_I$ 
so that 
$v_\lambda = [\Delta(\lambda)]$.
We let
$b_\lambda := [P(\lambda)]$ and $b_\lambda^* := [L(\lambda)]$;
in view of the uniqueness from
Theorem~\ref{known} these vectors are independent of the particular choice of the tensor product categorification
$\mathcal C$.
Thus 
we have constructed $U_q\mathfrak{sl}_I$-modules
$\textstyle[\mathcal C]_q \subseteq \bigwedge^{\bn,\bc} V_I \subseteq
[\mathcal C]_q^*$
with the distinguished bases
$\{b_\lambda\:|\:\lambda\in\Lambda\},
\{v_\lambda\:|\:\lambda\in\Lambda\}$ and
$\{b^*_\lambda\:|\:\lambda\in\Lambda\}$,
respectively.
When $I$ is finite, we have equalities
$\textstyle[\mathcal C]_q = \bigwedge^{\bn,\bc} V_I =
[\mathcal C]_q^*$, but in general this is not the case.

Let $\circledast:\mathcal C \rightarrow \mathcal C$ 
and $\#:\pC \rightarrow \pC$
be graded dualities
as in Corollary~\ref{gradedduality}.
Then define a bilinear pairing
$(-,-):[\mathcal C]_q \times [\mathcal C]^*_q \rightarrow \Q(q)$
by setting
\begin{align}\label{formies}
([P],[L]) &:= \qdim \Hom_{\widehat{\mathcal C}}(P^\#, L)
= \overline{\qdim \Hom_{\widehat{\mathcal C}}(P, L^\circledast)}.
\end{align}
Since $F^*_i$ is right adjoint to $F_i$ and $E^*_i$ is right adjoint to
$E_i$ this pairing has the property that $(u v,w) = (v,u^* w)$
for any $u \in U_q\mathfrak{sl}_I$.
Also it is immediate that 
the bases $\{b_\lambda\:|\:\lambda\in\Lambda\}$ and $\{b_\lambda^*\:|\:\lambda\in\Lambda\}$ are dual to each other.
Recall from (\ref{form}) that we have also already introduced a symmetric bilinear form $(-,-)$ on $\bigwedge^{\bn,\bc} V_I$ such that $(v_\lambda,v_\mu) = \delta_{\lambda,\mu}$. Our choice of notation here is not ambiguous
as the two forms agree on the intersection of their domains.
This follows because 
$$
(b_\lambda,v_\mu)
=\overline{\qdim \Hom_{\widehat{\mathcal C}}(P(\lambda), \nabla(\mu))}
=\qdim \Hom_{\widehat{\mathcal C}}(P(\lambda), \Delta(\mu))
= [\Delta(\mu):L(\lambda)]_q
$$
is the graded
multiplicity of $\Delta(\mu)$ in a graded $\Delta$-flag of $P$.
This argument recovers graded
{BGG reciprocity}: we have that
\begin{align}\label{trans}
b_\lambda &= \sum_{\mu \in \Lambda} d_{\lambda,\mu}(q) v_\mu,&
 v_\mu &= \sum_{\lambda \in \Lambda} d_{\lambda,\mu}(q)
b_\lambda^*,
\end{align}
recalling (\ref{decomp}).
When $I$ is finite we can invert these formulae to obtain also
\begin{align}\label{tran2}
v_\lambda &= \sum_{\mu \in \Lambda} p_{\lambda,\mu}(-q) b_\mu,
& b_\mu^* &= \sum_{\lambda \in \Lambda} p_{\lambda,\mu}(-q)
v_\lambda.
\end{align}
However the last two formulae do not make sense in general when $I$ is
infinite as the sums become infinite too; this issue can be bypassed
by introducing a suitable completion of $\bigwedge^{\bn,\bc} V_I$
but there is no need for us to do that here.

The dualities $\#$ and $\circledast$ induce antilinear involutions
\begin{equation}\label{psiies}
\psi:[\mathcal C]_q \rightarrow [\mathcal C]_q,
\qquad
\psi^*:[\mathcal C]_q^* \rightarrow [\mathcal C]_q^*,
\end{equation}
i.e. $\psi([P]) := [P^\#]$ and $\psi^*([L]) := [L^\circledast]$.
Equivalently, these are the unique antilinear involutions 
satisfying
\begin{equation}\label{a}
\psi(b_\lambda) = b_\lambda,
\qquad
\psi^*(b_\lambda^*) = b_\lambda^*
\end{equation}
for each $\lambda \in \Lambda$.
In view of \eqref{star-rep},
$\psi$ and $\psi^*$ are adjoint antilinear maps,
i.e. we have that $(\psi(v),w) = \overline{(v, \psi^*(w))}$.  
Also we have for any vector $v$ and any $u \in U_q\mathfrak{sl}_I$ that
\begin{equation}\label{b}
\psi(uv) = \psi(u) \psi(v),
\qquad
\psi^*(uv) = \psi^*(u) \psi^*(v),
\end{equation}
thanks to the equivariance properties from
Corollary~\ref{gradedduality}.

The following theorem is a special case of \cite[Proposition 7.3]{W3}, where it is proved using the tensor product algebra realization of $\mathcal C$. 
We give a slightly different proof here in terms of the axiomatic framework of \cite{LW}.

\begin{Theorem}\label{klthm}
Assume that $I$ is finite so that $[\mathcal C]_q =
\bigwedge^{\bn,\bc} V_I = [\mathcal C]_q^*$.
Then $\psi:\bigwedge^{\bn,\bc} V_I \rightarrow \bigwedge^{\bn,\bc} V_I$ 
coincides with Lusztig's bar involution
from \cite[$\S$27.3]{Lubook}.
\end{Theorem}

\begin{proof}
We prove this by induction on the level $l$ of $(\bn,\bc)$.
When $l=1$ the result is clear as both $\psi$ and Lusztig's bar involution
fix the highest weight vector $v_\kappa$ and commute with each $f_i$.
Now suppose that $l > 1$. Let $\bn^+ := (n_1,\dots,n_{l-1})$ and $\bc^+ := (c_1,\dots,c_{l-1})$, so that $\bigwedge^{\bn,\bc} V_I = \bigwedge^{\bn^+,\bc^+} V_I \otimes \bigwedge^{n_l,c_l} V_I$. 
Let $\Lambda^+ := \Lambda_{I;\bn^+,\bc^+}$.
To avoid potential confusion later on we denote the monomial basis vectors for $\bigwedge^{\bn^+,\bc^+} V_I$ by $v_\lambda^+$ instead of $v_\lambda$.
Let 
$\tilde\psi:\bigwedge^{\bn,\bc} V_I \rightarrow \bigwedge^{\bn,\bc} V_I$
and
$\tilde\psi^+:\bigwedge^{\bn^+,\bc^+} V_I \rightarrow \bigwedge^{\bn^+,\bc^+} V_I$
be Lusztig's bar involutions. It is easy to see from Lusztig's definition that the following two properties are satisfied:
\begin{itemize}
\item[(i)] $\tilde\psi(f_i v) = f_i \tilde\psi(v)$ for all $i \in I$ and $v \in \bigwedge^{\bn,\bc} V_I$;
\item[(ii)] $\tilde\psi(v \otimes v_{\kappa_l}) = \tilde\psi^+(v) \otimes v_{\kappa_l}$
for each $v \in \bigwedge^{\bn^+,\bc^+} V_I$.
\end{itemize}
Moreover these two properties characterize $\tilde\psi$ uniquely.
Thus to prove the theorem we must show that
$\psi(v \otimes v_{\kappa_l}) = \tilde\psi^+(v) \otimes v_{\kappa_l}$
too.
Let $\mathcal C^+$ be a 
$U_q \mathfrak{sl}_I$-tensor product categorification of
type $(\bn^+,\bc^+)$.
Its Grothendieck group $[\mathcal C^+] = \bigwedge^{\bn^+,\bc^+} V_I$
is equipped with the antilinear involution $\psi^+$
defined as above, so that $\psi^+(b^+_\lambda) = b^+_\lambda$ for each $\lambda \in \Lambda^+$ where $b^+_\lambda\in \bigwedge^{\bn^+,\bc^+} V_I$ denotes the class of the indecomposable projective $P(\lambda) \in \mathcal C^+$.
By the induction hypothesis $\psi^+ = \tilde\psi^+$.
So we are reduced to showing that
$\psi(v \otimes v_{\kappa_l}) = \psi^+(v) \otimes v_{\kappa_l}$;
equivalently we will show that
$\psi(b^+_{\lambda} \otimes v_{\kappa_l}) = b^+_\lambda \otimes v_{\kappa_l}$ for each $\lambda \in \Lambda^+$.

Now we make a particular choice for the category $\mathcal C^+$.
Let us identify
$\Lambda^+$ with the coideal
$\{\lambda \in \Lambda\:|\:\lambda_l = \kappa_l\}$ of $\Lambda$
so that 
$(\lambda_1,\dots,\lambda_{l-1}) \in \Lambda^+$
is identified with $(\lambda_1,\dots,\lambda_{l-1},\kappa_l) \in \Lambda$.
Let $\mathcal C^+$ be the
quotient category of $\mathcal C$ associated to this coideal. It is a
graded highest weight category with distinguished irreducible objects
$\{L(\lambda)\:|\:\lambda\in\Lambda^+\}$.
The functor $F_i$ obviously leaves invariant the Serre subcategory of $\mathcal C$ generated by the irreducible objects $\{Q^m L(\lambda)\:|\:m \in \Z,
\lambda \in \Lambda, |\lambda_l| < |\kappa_l|\}$. Hence it induces a
well-defined graded endofunctor $F_i^+:\mathcal C^+ \rightarrow \mathcal C^+$.
The homogeneous natural transformations $\xi$ and $\tau$ restrict to give
$\xi \in \End(F_i^+)_2$ and $\tau \in \Hom(F_j^+ \circ F_i^+,F_i^+\circ F_j^+)_{-\alpha_i\cdot\alpha_j}$.
The functors $K_i$ and $K_i^{-1}$ obviously descend to $\mathcal C^+$ too.
We claim further that there exists a graded endofunctor $E_i^+:\mathcal C^+ \rightarrow \mathcal C^+$ such that $Q E_i^+ K_i$ is right adjoint to $F_i^+$,
and that this data
endows the graded highest weight category 
$\mathcal C^+$ with the structure of
a $U_q\mathfrak{sl}_I$-tensor product categorification of type $(\bn^+,\bc^+)$.

The proof of the claim depends on the categorical splitting construction from \cite{LW}. To construct $E_i^+$ for a fixed $i \in I$, we let
$\mathcal C_i$ be the quotient of $\mathcal C$ associated to the coideal
$\Lambda_i := \{\lambda \in \Lambda\:|\:|\kappa_l|-|\lambda_l| \in \N \alpha_i\}$.
Both of the functors $F_i$ and $E_i$ descend to endofunctors of $\mathcal C_i$.
Let $\mathcal C_i^-$ be the subcategory of $\mathcal C_i$ associated to
the ideal $\{\lambda \in \Lambda_i\:|\:|\lambda_l| = |\kappa_l|-r\alpha_i\}$
where $r := \kappa_l\cdot\alpha_i$ (which happens in our special minuscule situation to be either $0$ or $1$).
Thus we have constructed categories and functors
$\mathcal C_i^-\stackrel{\iota}{\rightarrow} \mathcal C_i \stackrel{\pi}{\rightarrow} \mathcal C^+$.
The functor $F_i$ on $\mathcal C_i$ restricts to an endofunctor $F_i^-$ of 
$\mathcal C_i^-$.
Moreover \cite[Proposition 4.3]{LW} implies that
$\pi \circ E_i^{(r)} \circ \iota:\mathcal C_i^-\rightarrow \mathcal C^+$ is a graded equivalence of categories, which intertwines $F_i^-$ and $F_i^+$
by \cite[Lemma 4.7]{LW}.
Let $E_i^- := \iota^* \circ E_i \circ \iota:\mathcal C_i^- \rightarrow \mathcal C_i^-$. Since $Q E_i K_i$ is right adjoint to $F_i$ it is immediate that $Q E_i^- K_i$ is right adjoint to $F_i^-$.
Then we transfer $E_i^-$ through the equivalence
to obtain the desired functor 
$E_i^+:\mathcal C^+ \rightarrow \mathcal C^+$ such that $Q E_i^+ K_i$ is right adjoint to $F_i^+$.
Finally let $\overline{\mathcal C}^+$ be the underlying ungraded category.
In \cite[Theorem 4.10]{LW} it is shown that $\overline{F}_i^+$ and
$\overline{E}_i^+$ (together with the various other natural transformations induced by the ones constructed above)
make $\overline{\mathcal C}^+$ into an $\mathfrak{sl}_I$-tensor product categorification of type $(\bn^+,\bc^+)$.
We have in front of us graded lifts as in Theorem~\ref{main}(i). Then we apply
Theorem~\ref{main}(ii) to deduce that $\mathcal C^+$ is a 
$U_q\mathfrak{sl}_I$-tensor product categorification of type $(\bn^+,\bc^+)$.
So we have proved the claim.

We can now complete the proof of the theorem.
Let $\pi:\mathcal C \rightarrow \mathcal C^+$ be the quotient functor
and $\pi^!:\mathcal C^+ \rightarrow \mathcal C$ be a left adjoint.
In $\mathcal C$ we have that $\pi^! P(\lambda) \cong P(\lambda)$
and $\pi^! \Delta(\lambda) \cong \Delta(\lambda)$ for each $\lambda \in \Lambda^+$;
the latter isomorphism follows by the graded analog of Lemma~\ref{shriek}.
Thus $\pi^!$ induces a linear map
$\textstyle\bigwedge^{\bn^+,\bc^+} V_I \hookrightarrow
\bigwedge^{\bn,\bc} V_I$
such that $b^+_\lambda \mapsto b_\lambda$ and $v^+_\lambda \mapsto v_\lambda$
for each $\lambda \in \Lambda^+$.
It follows immediately that $b_\lambda = b_\lambda^+ \otimes v_{\kappa_l}$
for each $\lambda \in \Lambda^+$.
Hence $\psi(b_\lambda^+ \otimes v_{\kappa_l}) = b_\lambda^+ \otimes v_{\kappa_l}$
as required.
\end{proof}

Now we apply the positivity from Corollary~\ref{pg} (which we recall
depended itself on the results from
\cite{BGS}, \cite{Back} which exploit relations to geometry of flag
varieties) to get the following; see also \cite[Theorem 8.8]{W3} for the generalization of this to more general tensor products (with a different proof via geometry of certain quiver varieties).

\begin{Corollary}
For finite $I$,
$\{b_\lambda\:|\:\lambda\in\Lambda\}$ is Lusztig's canonical basis for
$\bigwedge^{\bn,\bc} V_I$ from \cite[$\S$27.3]{Lubook},
while $\{b_\lambda^*\:|\:\lambda\in\Lambda\}$ is the dual canonical basis.
\end{Corollary}

\begin{proof}
Corollary~\ref{pg}, (\ref{pos2}) and (\ref{a}) 
show that
$b_\lambda$ is a $\psi$-invariant vector in $v_\lambda + \sum_{\mu > \lambda}
q\Z[q] v_\mu$.
\end{proof}

This shows for finite $I$ that the polynomials $d_{\lambda,\mu}(q)$ from (\ref{decomp}) are the entries of the transition matrix from the canonical to the monomial basis of $\bigwedge^{\bn,\bc} V_I$, while the polynomials
$p_{\lambda,\mu}(-q)$ are the entries of the inverse transition
matrix.
In particular, the $p_{\lambda,\mu}(q)$ are certain finite type $A$ parabolic Kazhdan-Lusztig polynomials;
see e.g. \cite{FKK} or \cite[Remark 14]{Bdual} where this elementary
combinatorial identification is made explicitly.

To determine the polynomials $d_{\lambda,\mu}(q)$ and
$p_{\lambda,\mu}(q)$ when $I$ is infinite
it just remains to pick a finite subinterval $J \subset I$ such that
$\lambda,\mu \in \Lambda_J$. Then it is immediate from the definition 
(\ref{decomp}) and exactness of the quotient functor
that $d_{\lambda,\mu}(q)$ computed in $\mathcal C$ is the same as 
in $\mathcal C_J$. The same thing holds for $p_{\lambda,\mu}(q)$
in view of (\ref{ext1})--(\ref{ext2}).
Thus again all $p_{\lambda,\mu}(q)$ are identified
with some finite type $A$ parabolic Kazhdan-Lusztig polynomials.
If we specialize to $q=1$ this proves the super Kazhdan-Lusztig
conjecture as formulated in \cite[Conjecture 4.32]{B} (see also \cite{BcatO}), and all of its subsequent generalizations
to other Borels and parabolics. Note that this also establishes
\cite[Conjecture 3.13]{CLW} and \cite[Conjecture 2.28(i--ii)]{B}, showing the coefficients of this
canonical basis are positive, since they are identified with the
manifestly positive $d_{\lambda,\mu}(q)$, and similarly the coefficients
of the dual canonical basis are the manifestly alternating $p_{\lambda,\mu}(-q)$.

\begin{Remark}\rm
The basis called ``canonical basis'' in \cite{B} is a twisted version of the canonical basis here.
It corresponds to the indecomposable tilting objects
 rather than the indecomposable projectives in $\mathcal C$. In more detail,
let $T(\lambda) 
\in \mathcal C$ be the unique (up to isomorphism) 
$\circledast$-self-dual object
possessing a graded $\Delta$-flag with
$\Delta(\lambda)$ at the bottom and other sections of
the form $Q^n \Delta(\mu)$ for $\mu < \lambda$ and $n\in \Z$.
The existence of such an object follows by a construction due to Ringel
involving taking iterated 
extensions of standard objects;
cf. \cite{Btilt} which justifies 
in the context of super parabolic category $\mathcal O$
that Ringel's construction terminates after finitely many steps.
Since $\mathcal C$ is mixed the higher sections of a graded $\Delta$-flag
of $T(\lambda)$ are actually all
of the form $Q^n \Delta(\mu)$ for $\mu < \lambda$ and $n < 0$.
Let 
$$
\textstyle
\tilde b_\lambda := [T(\lambda)] \in \bigwedge^{\bn,\bc} V_I.
$$
When $I$ is finite this gives us the {\em twisted canonical basis}
$\{\tilde b_\lambda\:|\:\lambda \in \Lambda\}$ for $\bigwedge^{\bn,\bc} V_I$; 
each
$\tilde b_\lambda$ here is 
the unique $\psi^*$-invariant vector in
$v_\lambda + \sum_{\mu < \lambda} q^{-1} \Z[q^{-1}] v_\mu$.
In any case we have that
$$
\tilde b_\lambda = \sum_{\mu \in \Lambda} d_{\tilde \lambda, \tilde \mu}(q^{-1})
 v_\mu,
$$
where we write $\tilde\lambda$ for the 
$01$-matrix obtained from $\lambda$ by reversing the order of its rows.
This follows from 
\cite[Remark 3.10]{LW}, which implies that the
 Ringel dual of $\mathcal C$ has the
induced structure of a 
$U_q\mathfrak{sl}_I$-tensor product categorification 
of type $(\tilde \bn, \tilde \bc)$
where $\tilde \bn = (n_l,\dots,n_1)$ and $\tilde \bc = (c_l,\dots,c_1)$.
\end{Remark}

\begin{Remark}\rm
It was established already in \cite[Theorem 3.12]{CLW} that the structure constants describing the
actions of the quantum divided powers
$f_i^{(r)} / [r]!$ and $e_i^{(r)} / [r]!$ on 
the bases $\{b_\lambda\:|\:\lambda\in\Lambda\}$ and
$\{b_\lambda^*\:|\:\lambda\in\Lambda\}$
all belong to $\N[q,q^{-1}]$, proving \cite[Conjecture 2.28(iii--iv)]{B}.
Our results give a second proof of this conjecture since $f_i^{(r)}$ and $e_i^{(r)}$ have been
categorified by $F_i^{(r)}$ and $E_i^{(r)}$. 
In fact, this argument generalizes to show that any element of
Lusztig's canonical basis of the modified quantum algebra $\dot{U}_q \mathfrak{sl}_I$ 
acts with coefficients
in $\N[q,q^{-1}]$, since by \cite[Theorem A(b)]{W3} each of these basis
vectors can be lifted to a functor acting on any tensor product categorification.
\end{Remark}

\end{document}